\documentclass[onefignum,onetabnum]{siamart190516}

\usepackage{lipsum}
\usepackage{amsfonts}
\usepackage{graphicx}
\usepackage{epstopdf}
\usepackage{algorithmic}
\ifpdf
  \DeclareGraphicsExtensions{.eps,.pdf,.png,.jpg}
\else
  \DeclareGraphicsExtensions{.eps}
\fi

\newsiamremark{remark}{Remark}
\newsiamremark{hypothesis}{Hypothesis}
\crefname{hypothesis}{Hypothesis}{Hypotheses}
\newsiamthm{claim}{Claim}

\headers{Numerical Upscaling and Approximation via Subsampled Data}{Y. Chen and T.Y. Hou}

\title{Multiscale Elliptic PDEs Upscaling and Function Approximation via Subsampled Data\thanks{Submitted to the editors DATE: October 2020.
\funding{This research is in part supported by NSF Grants DMS-1912654 and DMS-1907977. Y. Chen is partly supported by the Caltech Kortchak Scholar Program.}}}

\author{Yifan Chen\thanks{Applied and Computational Mathematics, Caltech
  (\email{yifanc@caltech.edu}, \email{hou@cms.caltech.edu}).}
  \and Thomas Y. Hou\footnotemark[2]}

\usepackage{amsopn}

\usepackage{subfig}
\usepackage{float} 
\usepackage{hyperref} 
\usepackage{upgreek} 
\usepackage{xcolor} 
\usepackage{physics}
\usepackage{tikz}
\usepackage{mathdots}
\usepackage{yhmath}
\usepackage{cancel}
\usepackage{color}
\usepackage{siunitx}
\usepackage{array}
\usepackage{multirow}
\usepackage{tabularx}
\usepackage{booktabs}
\usetikzlibrary{fadings}
\usetikzlibrary{patterns}
\usetikzlibrary{shadows.blur}
\usetikzlibrary{shapes}
\usepackage{pgfplots}
\usepackage{amssymb}

\usepackage{diagbox} 

\newcommand{\cN}{\mathcal{N}}
\newcommand{\cL}{\mathcal{L}}

\newcommand{\bR}{\mathbb{R}}

\newcommand{\bN}{\mathbb{N}}

\newcommand{\rd}{\mathrm{d}}
\newcommand{\rN}{\mathrm{N}}

\newcommand{\sfd}{\mathsf{d}}

\newcommand{\sfP}{\mathsf{P}}

\begin{document}

\maketitle

\begin{abstract}
There is an intimate connection between numerical upscaling of multiscale PDEs and scattered data approximation of heterogeneous functions: the coarse variables selected for deriving an upscaled equation (in the former) correspond to the sampled information used for approximation (in the latter). As such, both problems can be thought of as recovering a target function based on some coarse data that are either artificially chosen by an upscaling algorithm, or determined by some physical measurement process. The purpose of this paper is then to study that, under such a setup and for a specific elliptic problem, how the lengthscale of the coarse data, which we refer to as the subsampled lengthscale, influences the accuracy of recovery, given limited computational budgets. Our analysis and experiments identify that, reducing the subsampling lengthscale may improve the accuracy, implying a guiding criterion for coarse-graining or data acquisition in this computationally constrained scenario, especially leading to direct insights for the implementation of the Gamblets method in the numerical homogenization literature. Moreover, reducing the lengthscale to zero may lead to a blow-up of approximation error if the target function does not have enough regularity, suggesting the need for a stronger prior assumption on the target function to be approximated. We introduce a singular weight function to deal with it, both theoretically and numerically. This work sheds light on the interplay of the lengthscale of coarse data, the computational costs, the regularity of the target function, and the accuracy of approximations and numerical simulations.
\end{abstract}

\begin{keywords}
  Multiscale PDEs, Numerical Upscaling, Function Approximation, Subsampled Data, Exponential Decay, Localization.
\end{keywords}

\begin{AMS}
  65D07, 65N15, 65N30, 35A35, 35J25, 65D05.
\end{AMS}
\section{Introduction}
\subsection{Background and Context}
\label{subsec: Background and Context}
In this paper, we are interested in studying a common approach for solving the following two categories of problems.
\subsubsection{Problem 1: Numerical Upscaling} The aim of this problem is to identify the coarse scale solution of a multiscale PDE via solving an upscaled equation for coarse variables. As a prototypical example, in $\Omega=[0,1]^d$, consider the elliptic equation for $u \in H_0^1(\Omega), f\in L^2(\Omega)$ and $\cL=-\nabla \cdot (a\nabla \cdot)$:
\begin{equation}
    \left\{
    \begin{aligned}
    \label{eqn: elliptic rough}
    \cL u&=f, \quad \text{in} \  \Omega\\
    u&=0, \quad \text{on} \  \partial\Omega\, ,
    \end{aligned}
    \right.
    \end{equation}
where the rough coefficient $a(x)$ satisfies $0<a_{\min}\leq a(x)\leq a_{\max}<\infty$ for $x \in \Omega$. Suppose we select the upscaled data of the solution: $[u,\phi_i], i\in I$ where  $\phi_i$ is some \textit{measurement function} that is often localized in space, $I$ is an index set and $[\cdot,\cdot]$ denotes the standard $L^2$ inner product. Then, the task is to derive an effective model for these upscaled variables and use them to approximate the solution of the PDE.

\subsubsection{Problem 2: Scattered Data Approximation} 
\label{subsec: Problem 2: Scattered Data Approximation}
This problem aims to recover a function $u$ (assume it has an underlying PDE model as \eqref{eqn: elliptic rough}) based on sampled data $[u,\phi_i], i\in I$. Here we intentionally use the same notation for the sampled data as that of the upscaled data in Problem 1 to make an explicit connection. We will also often call $[u,\phi_i], i\in I$ the coarse data in both problems.

\subsubsection{A Common Approach}
\label{subsec: a common approach}
Problem 1 is a standard task in multiscale PDEs computations, while Problem 2 has more of its backgrounds from data scientific investigations. Despite their distinguished origins, there is an approach that solves and connects the two -- studying of this method is the focus of the present paper.

To motivate the method, we start from Problem 1: a natural and ideal approach for getting the coarse data is to multiply the equation with the set of \textit{basis functions}: \[\operatorname{span}~\{\psi_i\}_{i\in I}=\operatorname{span}~\{ \cL^{-1}\phi_i\}_{i \in I}\, ,\] so that $[\psi_i, f], i \in I$, after an integration by part, matches the target $[u,\phi_i], i \in I$. 

Phrased in the language of Galerkin's method, $\{\psi_i\}_{i \in I}$ will constitute the test space; furthermore, one needs to select a trial space $V$ (with the same dimension) in order to get the ultimate numerical approximation of $u$. As such, this viewpoint has interpreted Problem 1 as a special case of Problem 2, of recovering $u$, from $[u,\phi_i], i \in I$, via choosing a space $V$. Often and conveniently, the trial space $V=\operatorname{span}~\{\psi_i\}_{i\in I}$ is chosen to be the same as the test space. Under such a choice and after selecting a suitable representative basis $\{\psi_i\}_{i\in I}$ of the linear space $V$ so that $[\psi_i,\phi_j]=\delta_{ij}$, we can write the final solution in a concise form:
\begin{equation}
\label{eqn: ideal sol}
    u^{\text{ideal}}:=\sum_{i\in I}[u,\phi_i]\psi_i\, .
\end{equation}
It is the ideal solution (here, ``ideal'' means that we have not accounted for the computational cost yet) in this setting, both to numerical upscaling and scattered data approximation. In practice, the basis function $\psi_i$ can have global support, and we need a localization step for efficient computation.

As a special case in numerical upscaling, if we choose $\phi_i$ to be piecewise linear tent functions, then we get the ideal LOD method \cite{malqvist_localization_2014}; if $\phi_i$ is set to be piecewise constant functions, then we obtain the Gamblet method in \cite{owhadi_multigrid_2017}. In their contexts, localization of $\{\psi_i\}_{i\in I}$ is achieved via an exponential decay property, and a provable accuracy guarantee has been established by controlling the coarse-graining error of using $u^{\text{ideal}}$ to approximate $u$ and the localization error of computing $\{\psi_i\}_{i\in I}$, respectively. 
\subsubsection{Our Goals}
The purposes of this paper are two folds. 
\begin{itemize}
    \item On the numerical upscaling side, we contribute a further discussion to this family of upscaling methods, concentrating on the fundamental role of a \textit{subsampled lengthscale} (defined in the next subsection) in choosing $\{\phi_i\}_{i\in I}$, with its highly non-trivial consequence on the localization of $\{\psi_i\}_{i \in I}$ and the solution accuracy of $u$. We will get a novel trade-off between approximation and localization regarding the subsampled scale.
    \item On the function approximation side, the above recovery method takes advantage of the underlying physical model \eqref{eqn: elliptic rough}, combining the merits of data and physics. In addition to contributing a detailed analysis of accuracy and comparisons to numerical upscaling, we will pay close attention to the regime where the subsampled lengthscale is small and approaches zero, in which we provide some numerical evidence that exemplifies, and extends our earlier work on function approximation via subsampled data \cite{chen2019function}.
\end{itemize} 

Our detailed contributions are outlined in Subsection \ref{subsec: our controbution}. 
\subsection{Subsampled Lengthscales}
\label{subsec: Subsampled Lengthscales}
We begin by introducing the concept of \textit{subsampled data}. For a demonstration of ideas, we work on the domain $\Omega=[0,1]^d$, and it is decomposed uniformly into cubes with side length $H$; this becomes our coarse grid. Let $I$ be the index set of these cubes such that its cardinality $|I|=1/H^d$. The measurement function $\phi_i^{h,H}$ (we use superscripts now for notational convenience) for each $i \in I$ is set to be the ($L^1$ normalized) indicator function of a cube with side length $0<h\leq H$, centered in the corresponding cube with side length $H$; see Figure \ref{fig:subsampled data illustration} for a two dimensional example\footnote{For illustration, the cube $\omega_i^{h,H}$ in the figure is centered in $\omega_i^{H}$. However, the relative position of the two cubes is not important in our analysis; see the proofs of Theorem \ref{thm: err ideal sol} and \ref{thm: error loc solution}. The key is that the subsampled Poincar\'e inequality developed in \cite{chen2019function} does not depend on the relative position of the subdomain and the domain.}. For each $i \in I$, these two cubes are denoted by $\omega_i^H$ and $\omega_i^{h,H}$ respectively; we assume they are closed sets, i.e., their boundaries are included.  We will call $H$ the coarse lengthscale, and $h$ is the \textit{subsampled lengthscale}.

The consideration of this subsampled lengthscale is natural both from the perspectives of function approximation and numerical upscaling. In the former scenario, the measurement data of a field function in physics is often the macroscopic averaged quantity, taking a similar form as $[u,\phi^{h,H}_i]$ for some $h\leq H$. In the latter problem, we have the freedom to choose the upscaled information of the multiscale PDEs, so taking a free parameter $h$ in the approach enables us to analyze the algorithm's behavior more thoroughly. Later on, we will see that the parameter $h$ has a non-trivial influence on the subsequent localization and accuracy of the approximation. 

 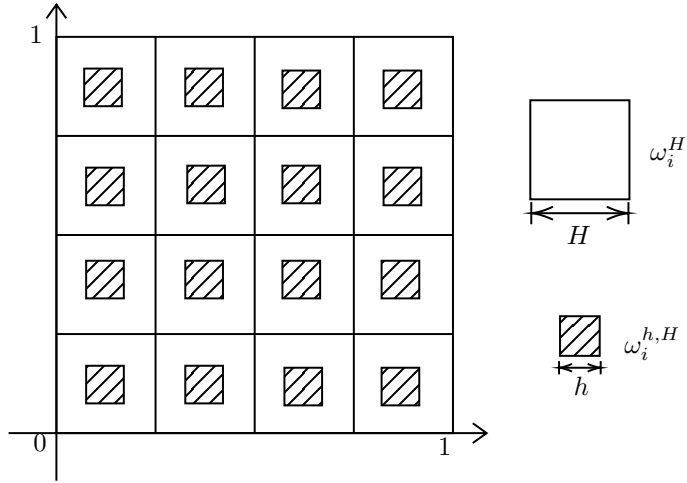
\begin{figure}[t]
    \centering
\tikzset{
pattern size/.store in=\mcSize, 
pattern size = 5pt,
pattern thickness/.store in=\mcThickness, 
pattern thickness = 0.3pt,
pattern radius/.store in=\mcRadius, 
pattern radius = 1pt}
\makeatletter
\pgfutil@ifundefined{pgf@pattern@name@_qbswpsvpp}{
\pgfdeclarepatternformonly[\mcThickness,\mcSize]{_qbswpsvpp}
{\pgfqpoint{0pt}{0pt}}
{\pgfpoint{\mcSize+\mcThickness}{\mcSize+\mcThickness}}
{\pgfpoint{\mcSize}{\mcSize}}
{
\pgfsetcolor{\tikz@pattern@color}
\pgfsetlinewidth{\mcThickness}
\pgfpathmoveto{\pgfqpoint{0pt}{0pt}}
\pgfpathlineto{\pgfpoint{\mcSize+\mcThickness}{\mcSize+\mcThickness}}
\pgfusepath{stroke}
}}
\makeatother

 
\tikzset{
pattern size/.store in=\mcSize, 
pattern size = 5pt,
pattern thickness/.store in=\mcThickness, 
pattern thickness = 0.3pt,
pattern radius/.store in=\mcRadius, 
pattern radius = 1pt}
\makeatletter
\pgfutil@ifundefined{pgf@pattern@name@_ws9792rnq}{
\pgfdeclarepatternformonly[\mcThickness,\mcSize]{_ws9792rnq}
{\pgfqpoint{0pt}{0pt}}
{\pgfpoint{\mcSize+\mcThickness}{\mcSize+\mcThickness}}
{\pgfpoint{\mcSize}{\mcSize}}
{
\pgfsetcolor{\tikz@pattern@color}
\pgfsetlinewidth{\mcThickness}
\pgfpathmoveto{\pgfqpoint{0pt}{0pt}}
\pgfpathlineto{\pgfpoint{\mcSize+\mcThickness}{\mcSize+\mcThickness}}
\pgfusepath{stroke}
}}
\makeatother

 
\tikzset{
pattern size/.store in=\mcSize, 
pattern size = 5pt,
pattern thickness/.store in=\mcThickness, 
pattern thickness = 0.3pt,
pattern radius/.store in=\mcRadius, 
pattern radius = 1pt}
\makeatletter
\pgfutil@ifundefined{pgf@pattern@name@_uactf712w}{
\pgfdeclarepatternformonly[\mcThickness,\mcSize]{_uactf712w}
{\pgfqpoint{0pt}{0pt}}
{\pgfpoint{\mcSize+\mcThickness}{\mcSize+\mcThickness}}
{\pgfpoint{\mcSize}{\mcSize}}
{
\pgfsetcolor{\tikz@pattern@color}
\pgfsetlinewidth{\mcThickness}
\pgfpathmoveto{\pgfqpoint{0pt}{0pt}}
\pgfpathlineto{\pgfpoint{\mcSize+\mcThickness}{\mcSize+\mcThickness}}
\pgfusepath{stroke}
}}
\makeatother

 
\tikzset{
pattern size/.store in=\mcSize, 
pattern size = 5pt,
pattern thickness/.store in=\mcThickness, 
pattern thickness = 0.3pt,
pattern radius/.store in=\mcRadius, 
pattern radius = 1pt}
\makeatletter
\pgfutil@ifundefined{pgf@pattern@name@_zwjo2m91v}{
\pgfdeclarepatternformonly[\mcThickness,\mcSize]{_zwjo2m91v}
{\pgfqpoint{0pt}{0pt}}
{\pgfpoint{\mcSize+\mcThickness}{\mcSize+\mcThickness}}
{\pgfpoint{\mcSize}{\mcSize}}
{
\pgfsetcolor{\tikz@pattern@color}
\pgfsetlinewidth{\mcThickness}
\pgfpathmoveto{\pgfqpoint{0pt}{0pt}}
\pgfpathlineto{\pgfpoint{\mcSize+\mcThickness}{\mcSize+\mcThickness}}
\pgfusepath{stroke}
}}
\makeatother

 
\tikzset{
pattern size/.store in=\mcSize, 
pattern size = 5pt,
pattern thickness/.store in=\mcThickness, 
pattern thickness = 0.3pt,
pattern radius/.store in=\mcRadius, 
pattern radius = 1pt}
\makeatletter
\pgfutil@ifundefined{pgf@pattern@name@_8p0iqeppb}{
\pgfdeclarepatternformonly[\mcThickness,\mcSize]{_8p0iqeppb}
{\pgfqpoint{0pt}{0pt}}
{\pgfpoint{\mcSize+\mcThickness}{\mcSize+\mcThickness}}
{\pgfpoint{\mcSize}{\mcSize}}
{
\pgfsetcolor{\tikz@pattern@color}
\pgfsetlinewidth{\mcThickness}
\pgfpathmoveto{\pgfqpoint{0pt}{0pt}}
\pgfpathlineto{\pgfpoint{\mcSize+\mcThickness}{\mcSize+\mcThickness}}
\pgfusepath{stroke}
}}
\makeatother

 
\tikzset{
pattern size/.store in=\mcSize, 
pattern size = 5pt,
pattern thickness/.store in=\mcThickness, 
pattern thickness = 0.3pt,
pattern radius/.store in=\mcRadius, 
pattern radius = 1pt}
\makeatletter
\pgfutil@ifundefined{pgf@pattern@name@_eamii8n54}{
\pgfdeclarepatternformonly[\mcThickness,\mcSize]{_eamii8n54}
{\pgfqpoint{0pt}{0pt}}
{\pgfpoint{\mcSize+\mcThickness}{\mcSize+\mcThickness}}
{\pgfpoint{\mcSize}{\mcSize}}
{
\pgfsetcolor{\tikz@pattern@color}
\pgfsetlinewidth{\mcThickness}
\pgfpathmoveto{\pgfqpoint{0pt}{0pt}}
\pgfpathlineto{\pgfpoint{\mcSize+\mcThickness}{\mcSize+\mcThickness}}
\pgfusepath{stroke}
}}
\makeatother

 
\tikzset{
pattern size/.store in=\mcSize, 
pattern size = 5pt,
pattern thickness/.store in=\mcThickness, 
pattern thickness = 0.3pt,
pattern radius/.store in=\mcRadius, 
pattern radius = 1pt}
\makeatletter
\pgfutil@ifundefined{pgf@pattern@name@_4uc27p9vl}{
\pgfdeclarepatternformonly[\mcThickness,\mcSize]{_4uc27p9vl}
{\pgfqpoint{0pt}{0pt}}
{\pgfpoint{\mcSize+\mcThickness}{\mcSize+\mcThickness}}
{\pgfpoint{\mcSize}{\mcSize}}
{
\pgfsetcolor{\tikz@pattern@color}
\pgfsetlinewidth{\mcThickness}
\pgfpathmoveto{\pgfqpoint{0pt}{0pt}}
\pgfpathlineto{\pgfpoint{\mcSize+\mcThickness}{\mcSize+\mcThickness}}
\pgfusepath{stroke}
}}
\makeatother

 
\tikzset{
pattern size/.store in=\mcSize, 
pattern size = 5pt,
pattern thickness/.store in=\mcThickness, 
pattern thickness = 0.3pt,
pattern radius/.store in=\mcRadius, 
pattern radius = 1pt}
\makeatletter
\pgfutil@ifundefined{pgf@pattern@name@_cwczrvvny}{
\pgfdeclarepatternformonly[\mcThickness,\mcSize]{_cwczrvvny}
{\pgfqpoint{0pt}{0pt}}
{\pgfpoint{\mcSize+\mcThickness}{\mcSize+\mcThickness}}
{\pgfpoint{\mcSize}{\mcSize}}
{
\pgfsetcolor{\tikz@pattern@color}
\pgfsetlinewidth{\mcThickness}
\pgfpathmoveto{\pgfqpoint{0pt}{0pt}}
\pgfpathlineto{\pgfpoint{\mcSize+\mcThickness}{\mcSize+\mcThickness}}
\pgfusepath{stroke}
}}
\makeatother

 
\tikzset{
pattern size/.store in=\mcSize, 
pattern size = 5pt,
pattern thickness/.store in=\mcThickness, 
pattern thickness = 0.3pt,
pattern radius/.store in=\mcRadius, 
pattern radius = 1pt}
\makeatletter
\pgfutil@ifundefined{pgf@pattern@name@_vgeuoys8b}{
\pgfdeclarepatternformonly[\mcThickness,\mcSize]{_vgeuoys8b}
{\pgfqpoint{0pt}{0pt}}
{\pgfpoint{\mcSize+\mcThickness}{\mcSize+\mcThickness}}
{\pgfpoint{\mcSize}{\mcSize}}
{
\pgfsetcolor{\tikz@pattern@color}
\pgfsetlinewidth{\mcThickness}
\pgfpathmoveto{\pgfqpoint{0pt}{0pt}}
\pgfpathlineto{\pgfpoint{\mcSize+\mcThickness}{\mcSize+\mcThickness}}
\pgfusepath{stroke}
}}
\makeatother

 
\tikzset{
pattern size/.store in=\mcSize, 
pattern size = 5pt,
pattern thickness/.store in=\mcThickness, 
pattern thickness = 0.3pt,
pattern radius/.store in=\mcRadius, 
pattern radius = 1pt}
\makeatletter
\pgfutil@ifundefined{pgf@pattern@name@_hoandl6qy}{
\pgfdeclarepatternformonly[\mcThickness,\mcSize]{_hoandl6qy}
{\pgfqpoint{0pt}{0pt}}
{\pgfpoint{\mcSize+\mcThickness}{\mcSize+\mcThickness}}
{\pgfpoint{\mcSize}{\mcSize}}
{
\pgfsetcolor{\tikz@pattern@color}
\pgfsetlinewidth{\mcThickness}
\pgfpathmoveto{\pgfqpoint{0pt}{0pt}}
\pgfpathlineto{\pgfpoint{\mcSize+\mcThickness}{\mcSize+\mcThickness}}
\pgfusepath{stroke}
}}
\makeatother

 
\tikzset{
pattern size/.store in=\mcSize, 
pattern size = 5pt,
pattern thickness/.store in=\mcThickness, 
pattern thickness = 0.3pt,
pattern radius/.store in=\mcRadius, 
pattern radius = 1pt}
\makeatletter
\pgfutil@ifundefined{pgf@pattern@name@_qk2okc6ka}{
\pgfdeclarepatternformonly[\mcThickness,\mcSize]{_qk2okc6ka}
{\pgfqpoint{0pt}{0pt}}
{\pgfpoint{\mcSize+\mcThickness}{\mcSize+\mcThickness}}
{\pgfpoint{\mcSize}{\mcSize}}
{
\pgfsetcolor{\tikz@pattern@color}
\pgfsetlinewidth{\mcThickness}
\pgfpathmoveto{\pgfqpoint{0pt}{0pt}}
\pgfpathlineto{\pgfpoint{\mcSize+\mcThickness}{\mcSize+\mcThickness}}
\pgfusepath{stroke}
}}
\makeatother

 
\tikzset{
pattern size/.store in=\mcSize, 
pattern size = 5pt,
pattern thickness/.store in=\mcThickness, 
pattern thickness = 0.3pt,
pattern radius/.store in=\mcRadius, 
pattern radius = 1pt}
\makeatletter
\pgfutil@ifundefined{pgf@pattern@name@_he3vxhz9l}{
\pgfdeclarepatternformonly[\mcThickness,\mcSize]{_he3vxhz9l}
{\pgfqpoint{0pt}{0pt}}
{\pgfpoint{\mcSize+\mcThickness}{\mcSize+\mcThickness}}
{\pgfpoint{\mcSize}{\mcSize}}
{
\pgfsetcolor{\tikz@pattern@color}
\pgfsetlinewidth{\mcThickness}
\pgfpathmoveto{\pgfqpoint{0pt}{0pt}}
\pgfpathlineto{\pgfpoint{\mcSize+\mcThickness}{\mcSize+\mcThickness}}
\pgfusepath{stroke}
}}
\makeatother

 
\tikzset{
pattern size/.store in=\mcSize, 
pattern size = 5pt,
pattern thickness/.store in=\mcThickness, 
pattern thickness = 0.3pt,
pattern radius/.store in=\mcRadius, 
pattern radius = 1pt}
\makeatletter
\pgfutil@ifundefined{pgf@pattern@name@_gsb68qa3i}{
\pgfdeclarepatternformonly[\mcThickness,\mcSize]{_gsb68qa3i}
{\pgfqpoint{0pt}{0pt}}
{\pgfpoint{\mcSize+\mcThickness}{\mcSize+\mcThickness}}
{\pgfpoint{\mcSize}{\mcSize}}
{
\pgfsetcolor{\tikz@pattern@color}
\pgfsetlinewidth{\mcThickness}
\pgfpathmoveto{\pgfqpoint{0pt}{0pt}}
\pgfpathlineto{\pgfpoint{\mcSize+\mcThickness}{\mcSize+\mcThickness}}
\pgfusepath{stroke}
}}
\makeatother

 
\tikzset{
pattern size/.store in=\mcSize, 
pattern size = 5pt,
pattern thickness/.store in=\mcThickness, 
pattern thickness = 0.3pt,
pattern radius/.store in=\mcRadius, 
pattern radius = 1pt}
\makeatletter
\pgfutil@ifundefined{pgf@pattern@name@_vtz415f9b}{
\pgfdeclarepatternformonly[\mcThickness,\mcSize]{_vtz415f9b}
{\pgfqpoint{0pt}{0pt}}
{\pgfpoint{\mcSize+\mcThickness}{\mcSize+\mcThickness}}
{\pgfpoint{\mcSize}{\mcSize}}
{
\pgfsetcolor{\tikz@pattern@color}
\pgfsetlinewidth{\mcThickness}
\pgfpathmoveto{\pgfqpoint{0pt}{0pt}}
\pgfpathlineto{\pgfpoint{\mcSize+\mcThickness}{\mcSize+\mcThickness}}
\pgfusepath{stroke}
}}
\makeatother

 
\tikzset{
pattern size/.store in=\mcSize, 
pattern size = 5pt,
pattern thickness/.store in=\mcThickness, 
pattern thickness = 0.3pt,
pattern radius/.store in=\mcRadius, 
pattern radius = 1pt}
\makeatletter
\pgfutil@ifundefined{pgf@pattern@name@_oyha77j1e}{
\pgfdeclarepatternformonly[\mcThickness,\mcSize]{_oyha77j1e}
{\pgfqpoint{0pt}{0pt}}
{\pgfpoint{\mcSize+\mcThickness}{\mcSize+\mcThickness}}
{\pgfpoint{\mcSize}{\mcSize}}
{
\pgfsetcolor{\tikz@pattern@color}
\pgfsetlinewidth{\mcThickness}
\pgfpathmoveto{\pgfqpoint{0pt}{0pt}}
\pgfpathlineto{\pgfpoint{\mcSize+\mcThickness}{\mcSize+\mcThickness}}
\pgfusepath{stroke}
}}
\makeatother

 
\tikzset{
pattern size/.store in=\mcSize, 
pattern size = 5pt,
pattern thickness/.store in=\mcThickness, 
pattern thickness = 0.3pt,
pattern radius/.store in=\mcRadius, 
pattern radius = 1pt}
\makeatletter
\pgfutil@ifundefined{pgf@pattern@name@_6u2xvqdwk}{
\pgfdeclarepatternformonly[\mcThickness,\mcSize]{_6u2xvqdwk}
{\pgfqpoint{0pt}{0pt}}
{\pgfpoint{\mcSize+\mcThickness}{\mcSize+\mcThickness}}
{\pgfpoint{\mcSize}{\mcSize}}
{
\pgfsetcolor{\tikz@pattern@color}
\pgfsetlinewidth{\mcThickness}
\pgfpathmoveto{\pgfqpoint{0pt}{0pt}}
\pgfpathlineto{\pgfpoint{\mcSize+\mcThickness}{\mcSize+\mcThickness}}
\pgfusepath{stroke}
}}
\makeatother

 
\tikzset{
pattern size/.store in=\mcSize, 
pattern size = 5pt,
pattern thickness/.store in=\mcThickness, 
pattern thickness = 0.3pt,
pattern radius/.store in=\mcRadius, 
pattern radius = 1pt}
\makeatletter
\pgfutil@ifundefined{pgf@pattern@name@_nyvzlbke7}{
\pgfdeclarepatternformonly[\mcThickness,\mcSize]{_nyvzlbke7}
{\pgfqpoint{0pt}{0pt}}
{\pgfpoint{\mcSize+\mcThickness}{\mcSize+\mcThickness}}
{\pgfpoint{\mcSize}{\mcSize}}
{
\pgfsetcolor{\tikz@pattern@color}
\pgfsetlinewidth{\mcThickness}
\pgfpathmoveto{\pgfqpoint{0pt}{0pt}}
\pgfpathlineto{\pgfpoint{\mcSize+\mcThickness}{\mcSize+\mcThickness}}
\pgfusepath{stroke}
}}
\makeatother
\tikzset{every picture/.style={line width=0.75pt}} 

\begin{tikzpicture}[x=0.75pt,y=0.75pt,yscale=-1,xscale=1]

\draw  [draw opacity=0] (165,53) -- (365,53) -- (365,253) -- (165,253) -- cycle ; \draw   (215,53) -- (215,253)(265,53) -- (265,253)(315,53) -- (315,253) ; \draw   (165,103) -- (365,103)(165,153) -- (365,153)(165,203) -- (365,203) ; \draw   (165,53) -- (365,53) -- (365,253) -- (165,253) -- cycle ;
\draw  [pattern=_qbswpsvpp,pattern size=6pt,pattern thickness=0.75pt,pattern radius=0pt, pattern color={rgb, 255:red, 0; green, 0; blue, 0}] (179,69) -- (198,69) -- (198,88) -- (179,88) -- cycle ;
\draw  [pattern=_ws9792rnq,pattern size=6pt,pattern thickness=0.75pt,pattern radius=0pt, pattern color={rgb, 255:red, 0; green, 0; blue, 0}] (180,119) -- (199,119) -- (199,138) -- (180,138) -- cycle ;
\draw  [pattern=_uactf712w,pattern size=6pt,pattern thickness=0.75pt,pattern radius=0pt, pattern color={rgb, 255:red, 0; green, 0; blue, 0}] (230,69) -- (249,69) -- (249,88) -- (230,88) -- cycle ;
\draw  [pattern=_zwjo2m91v,pattern size=6pt,pattern thickness=0.75pt,pattern radius=0pt, pattern color={rgb, 255:red, 0; green, 0; blue, 0}] (180,219) -- (199,219) -- (199,238) -- (180,238) -- cycle ;
\draw  [pattern=_8p0iqeppb,pattern size=6pt,pattern thickness=0.75pt,pattern radius=0pt, pattern color={rgb, 255:red, 0; green, 0; blue, 0}] (180,166) -- (199,166) -- (199,185) -- (180,185) -- cycle ;
\draw  [pattern=_eamii8n54,pattern size=6pt,pattern thickness=0.75pt,pattern radius=0pt, pattern color={rgb, 255:red, 0; green, 0; blue, 0}] (330,119) -- (349,119) -- (349,138) -- (330,138) -- cycle ;
\draw  [pattern=_4uc27p9vl,pattern size=6pt,pattern thickness=0.75pt,pattern radius=0pt, pattern color={rgb, 255:red, 0; green, 0; blue, 0}] (330,70) -- (349,70) -- (349,89) -- (330,89) -- cycle ;
\draw  [pattern=_cwczrvvny,pattern size=6pt,pattern thickness=0.75pt,pattern radius=0pt, pattern color={rgb, 255:red, 0; green, 0; blue, 0}] (280,220) -- (299,220) -- (299,239) -- (280,239) -- cycle ;
\draw  [pattern=_vgeuoys8b,pattern size=6pt,pattern thickness=0.75pt,pattern radius=0pt, pattern color={rgb, 255:red, 0; green, 0; blue, 0}] (279,166) -- (298,166) -- (298,185) -- (279,185) -- cycle ;
\draw  [pattern=_hoandl6qy,pattern size=6pt,pattern thickness=0.75pt,pattern radius=0pt, pattern color={rgb, 255:red, 0; green, 0; blue, 0}] (279,118) -- (298,118) -- (298,137) -- (279,137) -- cycle ;
\draw  [pattern=_qk2okc6ka,pattern size=6pt,pattern thickness=0.75pt,pattern radius=0pt, pattern color={rgb, 255:red, 0; green, 0; blue, 0}] (279,70) -- (298,70) -- (298,89) -- (279,89) -- cycle ;
\draw  [pattern=_he3vxhz9l,pattern size=6pt,pattern thickness=0.75pt,pattern radius=0pt, pattern color={rgb, 255:red, 0; green, 0; blue, 0}] (230,219) -- (249,219) -- (249,238) -- (230,238) -- cycle ;
\draw  [pattern=_gsb68qa3i,pattern size=6pt,pattern thickness=0.75pt,pattern radius=0pt, pattern color={rgb, 255:red, 0; green, 0; blue, 0}] (230,166) -- (249,166) -- (249,185) -- (230,185) -- cycle ;
\draw  [pattern=_vtz415f9b,pattern size=6pt,pattern thickness=0.75pt,pattern radius=0pt, pattern color={rgb, 255:red, 0; green, 0; blue, 0}] (231,118) -- (250,118) -- (250,137) -- (231,137) -- cycle ;
\draw  [pattern=_oyha77j1e,pattern size=6pt,pattern thickness=0.75pt,pattern radius=0pt, pattern color={rgb, 255:red, 0; green, 0; blue, 0}] (329,220) -- (348,220) -- (348,239) -- (329,239) -- cycle ;
\draw  [pattern=_6u2xvqdwk,pattern size=6pt,pattern thickness=0.75pt,pattern radius=0pt, pattern color={rgb, 255:red, 0; green, 0; blue, 0}] (329,166) -- (348,166) -- (348,185) -- (329,185) -- cycle ;

\draw   (404,85) -- (454,85) -- (454,135) -- (404,135) -- cycle ;
\draw    (404.5,142) -- (453.5,142) ;
\draw [shift={(453.5,142)}, rotate = 180] [color={rgb, 255:red, 0; green, 0; blue, 0 }  ][line width=0.75]    (0,5.59) -- (0,-5.59)(10.93,-3.29) .. controls (6.95,-1.4) and (3.31,-0.3) .. (0,0) .. controls (3.31,0.3) and (6.95,1.4) .. (10.93,3.29)   ;
\draw [shift={(404.5,142)}, rotate = 0] [color={rgb, 255:red, 0; green, 0; blue, 0 }  ][line width=0.75]    (0,5.59) -- (0,-5.59)(10.93,-3.29) .. controls (6.95,-1.4) and (3.31,-0.3) .. (0,0) .. controls (3.31,0.3) and (6.95,1.4) .. (10.93,3.29)   ;
\draw  [pattern=_nyvzlbke7,pattern size=6pt,pattern thickness=0.75pt,pattern radius=0pt, pattern color={rgb, 255:red, 0; green, 0; blue, 0}] (419,194) -- (439,194) -- (439,214) -- (419,214) -- cycle ;
\draw    (418.5,220) -- (439.5,220) ;
\draw [shift={(439.5,220)}, rotate = 180] [color={rgb, 255:red, 0; green, 0; blue, 0 }  ][line width=0.75]    (0,3.35) -- (0,-3.35)(6.56,-1.97) .. controls (4.17,-0.84) and (1.99,-0.18) .. (0,0) .. controls (1.99,0.18) and (4.17,0.84) .. (6.56,1.97)   ;
\draw [shift={(418.5,220)}, rotate = 0] [color={rgb, 255:red, 0; green, 0; blue, 0 }  ][line width=0.75]    (0,3.35) -- (0,-3.35)(6.56,-1.97) .. controls (4.17,-0.84) and (1.99,-0.18) .. (0,0) .. controls (1.99,0.18) and (4.17,0.84) .. (6.56,1.97)   ;

\draw  (140.89,253) -- (381.98,253)(165,36.66) -- (165,277.04) (374.98,248) -- (381.98,253) -- (374.98,258) (160,43.66) -- (165,36.66) -- (170,43.66)  ;

\draw (421,147) node [anchor=north west][inner sep=0.75pt]   [align=left] {$\displaystyle H$};
\draw (424.5,223) node [anchor=north west][inner sep=0.75pt]   [align=left] {$\displaystyle h$};
\draw (152,252) node [anchor=north west][inner sep=0.75pt]   [align=left] {$\displaystyle 0$};
\draw (356,254) node [anchor=north west][inner sep=0.75pt]   [align=left] {$\displaystyle 1$};
\draw (150,46) node [anchor=north west][inner sep=0.75pt]   [align=left] {$\displaystyle 1$};
\draw (463,102) node [anchor=north west][inner sep=0.75pt]   [align=left] {$\displaystyle \omega ^{H}_{i}$};
\draw (450,198) node [anchor=north west][inner sep=0.75pt]   [align=left] {$\displaystyle \omega ^{h,H}_{i}$};
\end{tikzpicture}
    \caption{Illustration of Subsampled Data: $H=1/4, h=1/10$}
    \label{fig:subsampled data illustration}
\end{figure}
 Note that the choice of $\omega_i^{H}$ and $\omega^{h,H}_i$ being cubes here is for convenience of analysis only; results in this paper will generalize easily to regular domains with other shapes. 
\subsection{Basis Functions and Localization}
\label{subsec: Basis Functions and Localization}Before outlining our main contributions (which are in the next subsection), we make precise here the definition of the basis functions and their localization. Per the discussion in Subsection \ref{subsec: Background and Context} and especially the formula \eqref{eqn: ideal sol}, the basis function $\psi_i^{h,H}$ (we add the superscripts for notational clarity) is the solution of the following variational problem:
 \begin{equation}
    \label{eqn: optimization def basis}
    \begin{aligned}
    \psi_{i}^{h,H} = \text{argmin}_{\psi \in H_0^1(\Omega)}\quad  &\|\psi\|_{H_a^1(\Omega)}^2 \\
     \text{subject to}\quad &[\psi, \phi_j^{h,H}] = \delta_{i,j}\ \  \text{for}\ \  j \in I \, ,
 \end{aligned}
    \end{equation}
where, we have used the notation $\|\psi\|_{H_a^1(\Omega)}^2:=\int_{\Omega} a|\nabla \psi|^2$. This formulation is a consequence of the two properties that are mentioned in Subsection \ref{subsec: Background and Context}: 
\[\text{(I) } \operatorname{span}~\{\psi_i^{h,H}\}_{i\in I}=\operatorname{span}~\{ \cL^{-1}\phi_i^{h,H}\}_{i \in I} \quad \text{ and } \quad \text{(II) } [\psi_i^{h,H},\phi_j^{h,H}]=\delta_{ij}\, .\]
For ease of computation, in practice we will solve a localized version of \eqref{eqn: optimization def basis} instead:
 \begin{equation}
    \label{eqn: localized version}
    \begin{aligned}
    \psi_{i}^{h,H,l} = \text{argmin}_{\psi \in H_0^1(\rN^l(\omega_i^H))}\quad  &\|\psi\|_{H_a^1(\rN^l(\omega_i^H))}^2 \\
     \text{subject to}\quad &[\psi, \phi_j^{h,H}] = \delta_{i,j}\ \  \text{for}\ \  j \in I \, ,
 \end{aligned}
    \end{equation}
where $l \in \bN$ is called the \textit{oversampled layer}. We have $\rN^0(\omega_i^H)=\omega_i^H$, and recursively:
\begin{equation}
    \rN^l(\omega_i^H):= \bigcup \{ \omega_j^{H}, j \in I: \omega_j^H \cap \rN^{l-1}(\omega_i^H) \neq \emptyset\}\, .
\end{equation}
Then, the level-$l$ localized solution for Problem 2 is
\begin{equation}
\label{eqn: loc sol}
    u^{\text{loc},l}:=\sum_{i\in I}[u,\phi_i^{h,H}]\psi_i^{h,H,l}\, .
\end{equation}
By abuse of notation, we will equate $u^{\text{loc},\infty}=u^{\text{ideal}}$.
The energy error and $L^2$ error of this localized solution are written as
\begin{equation}
    \begin{aligned}
    &e^{h,H,l}_{1}(a,u)=\|u-u^{\text{loc},l}\|_{H_a^1(\Omega)}\, ,\\
    &e^{h,H,l}_{0}(a,u)=\|u-u^{\text{loc},l}\|_{L^2(\Omega)}\, .
    \end{aligned}
\end{equation}
For Problem 1, we also get a solution $\tilde{u}^{\text{loc},l}$ by using the localized basis functions $\{\psi_{i}^{h,H,l}\}_{i\in I}$ and the Galerkin method. This solution is different from $u^{\text{loc},l}$ in general, unless $l=\infty$, i.e., in the ideal case.
The corresponding energy error and $L^2$ error of $\tilde{u}^{\text{loc},l}$ are denoted by $\tilde{e}^{h,H,l}_{1}(a,u)$ and $\tilde{e}^{h,H,l}_{0}(a,u)$. 

We call $u^{\text{loc},l}$ the \textit{recovery solution} of Problem 2, and $\tilde{u}^{\text{loc},l}$ the \textit{Galerkin solution} of Problem 1. The computation costs of the two solutions are different -- the former only requires solving the basis functions, while the latter also needs to solve an upscaled equation. Their errors in the solution are called the recovery error and Galerkin error, respectively. 

Under the above setup, our precise goal in this paper is to understand how the recovery error and Galerkin error depend on the following three factors: 
\begin{enumerate}
    \item The coarse scale $H$ and subsampled lengthscale $h$;
    \item The oversampled layer $l$ (corresponded to the computational budget);
    \item The regularity of function $u$ (in function approximation, it is given as prior information; in multiscale PDEs, it is influenced by the right-hand side $f$).
\end{enumerate}
Note that the regularity of a function is also intimately connected to the dimension parameter $d$.
\subsection{Our Contributions}
\label{subsec: our controbution}
In the first part of this work, we consider the finite regime of the subsampled lengthscale, i.e., $h$ is a strictly positive number. 
\begin{itemize}
    \item We provide numerical experiments and theoretical analysis of these recovery and Galerkin errors. We show that for a fixed $h/H$, if $l=O(\log (1/H))$, then both energy errors are of $O(H)$ and both $L^2$ errors are of $O(H^2)$.
    \item Further, we decompose the error into two parts: the approximation error of the ideal solution and the localization error. We demonstrate that there is a competition between the two. Roughly, reducing $h$ worsens the former, while improving the latter, for a fixed $H$ and $l$. This leads to a novel trade-off that was not investigated before -- choosing an appropriate $h$ can benefit the final accuracy.
    \item Moreover, there appears a fundamental difference between $e^{h,H,l}_{0}(a,u)$ and the other three errors, when $d\geq 2$. For a fixed $l$ and $h/H$, the former remains bounded as $H \to 0$, while the other three blow up. 
    We characterize this phenomenon both theoretically and numerically.
\end{itemize}
In the second part of this work, we consider the small limit regime of $h$. When $d\geq 2$, the error estimates in the first part blow up as $h\to 0$. To remedy this issue in the context of scattered data approximation, we propose to use a singular weight function in the algorithm. The weight function puts more importance on the subsampled data and avoids the degeneracy, given the target function has improved regularity property around these data. Numerical experiments and theoretical analysis are presented to offer a quantitative explanation of this phenomenon.

\subsection{Related Works} We review the related works below.
\subsubsection{Numerical Upscaling} There have been vast literature on numerical upscaling of multiscale PDEs. For our context, i.e., elliptic PDEs with rough coefficients, rigorous theoretical results include Generalized Finite Element Methods (GFEM) \cite{babuvska1994special,babuska2011optimal}, Harmonic Coordinates \cite{owhadi2007metric}, Local Orthogonal Decomposition (LOD) \cite{malqvist_localization_2014,henning2013oversampling,kornhuber2018analysis,engwer2019efficient, hauck2021super,maier2021high}, Gamblets related approaches \cite{owhadi2011localized,owhadi2014polyharmonic,owhadi2015bayesian,owhadi_multigrid_2017,hou_sparse_2017,owhadi2019operator}, and generalizations of Multiscale Finite Element Methods (MsFEM) \cite{hou2015optimal,chung2018constraint,li2019convergence,fu2019edge,chen2020exponential,chen2021exponentially}, etc. Among them, the ones most related to this paper are LOD and Gamblets; the connection has been explained in Subsection \ref{subsec: a common approach}. Indeed, in Gamblets \cite{owhadi_multigrid_2017,owhadi2019operator}, the author has formulated the framework in the perspective of optimal recovery, bridging numerical upscaling to game-theoretical approaches and Gaussian process regressions for function recovery. This formulation connects our Problem 1 and Problem 2 in Subsection \ref{subsec: Background and Context}.

A main component in LOD and Gamblets is the localization problem -- the ideal multiscale basis functions need to be localized for efficient computation. In this paper, our localization strategy, as outlined in Subsection \ref{subsec: Basis Functions and Localization}, follows from the one in \cite{malqvist_localization_2014,owhadi_multigrid_2017}. The main difference is that our measurement function $\phi_i^{h,H}$ contains a subsampled lengthscale parameter, which makes the analysis more delicate. Moreover, in addition to showing a trade-off between approximation errors and localization errors regarding the oversampling parameter $l$, our setup allows us to discover another trade-off regarding the subsampled lengthscale $h$ -- a good choice of $h$ can improve the algorithm in \cite{malqvist_localization_2014, owhadi_multigrid_2017}. We also remark that the work \cite{li2018error} has considered a similar algorithm for convection-dominated diffusion equations, where $h$ is fixed to be the small scale grid size, but the analysis there did not reveal the trade-off here.
\subsubsection{Function Approximation}
Function approximation via scattered data is a classical problem in numerical analysis (interpolation), statistics (non-parametric regression), and machine learning (supervised learning). For the type of scattered data, the most frequently considered one is the pointwise data \cite{wendland2004scattered}. The subsampled data introduce an additional small scale parameter $h$, and are generalizations to pointwise data. Our earlier work \cite{chen2019function} performed some analysis on this aspect, and provides some theoretical foundation for this paper. The multiscale basis functions constructed for the subsampled data allow us to capture the heterogeneous behaviors of the target function.

The method in Subsection \ref{subsec: a common approach} connects to the graph Laplacian approach in semisupervised learning. In the machine learning literature, the degeneracy issue of graph Laplacians has long been studied, and various approaches have been proposed to remedy this issue. Among them, the one that is most related to this paper is the weighted graph Laplacian method \cite{shi2017weighted,calder2019properly}, which puts more weights around the labeled data to avoid degeneracy. The second part of this work presents some analysis for this type of idea in the context of numerical analysis.

\subsection{Organization} The rest of this paper is organized as follows. In Section \ref{subsec: Finite Regime of Subsampled Lengthscales} we discuss the regime that $0<h\leq H$. We present numerical experiments and theoretical analysis of these Galerkin errors in numerical upscaling, and recovery errors in function approximation. In Section \ref{subsec Small Limit Regime of Subsampled Lengthscales}, we consider the regime $h \to 0$, a case that degeneracy may occur. We use a singular weight function to deal with this issue both numerically and theoretically. Section \ref{sec:proof} contains all the proofs in this paper. We summarize, discuss, and conclude this paper in Section \ref{sec Concluding Remarks}.
\section{Finite Regime of Subsampled Lengthscales} 
\label{subsec: Finite Regime of Subsampled Lengthscales}
In this section, we study the finite regime of $h$, i.e., $0<h\leq H$. 
We start with the ideal solution $u^{\text{ideal}}$, or equivalently $u^{\text{loc},\infty}$, and then move to the localized solution $u^{\text{loc},l}$ and $\tilde{u}^{\text{loc},l}$  for finite $l$.  Experiments are presented first, followed with theoretical analysis. Special attention is paid to the dependence of accuracy on the coarse scale $H$, subsampled lengthscale $h$ and when in the localized case, the oversampling parameter $l$. 

\subsection{Experiments: Ideal Solution} 
\label{subsec: Experiments: Ideal Solution}
In this subsection, we perform a numerical study of the effect of $h$ in $e^{h,H,\infty}_{1}(a,u)$ and $e^{h,H,\infty}_{0}(a,u)$, for $d=1$ and $2$ respectively. 

In this ideal case, the recovery solution and Galerkin solution are the same, and in our computation, we directly solve a PDE to get these solutions. Theoretical analysis of these numerical results is given in Subsection \ref{subsec: Analysis: Ideal Solution}.

\subsubsection{One Dimensional Example}
\label{subsec: ideal 1d experiments}
We consider the domain $\Omega=[0,1]$. The rough coefficient $a(x)$ is a sample drawn from the random field
\begin{equation}
    \xi=1+0.5\times \sin (\sum_{k=1}^{100}\eta_k\cos(kx)+\zeta_k\sin(kx))\, ,
\end{equation}
where $\eta_k,\zeta_k, 1\leq k \leq 100$ are i.i.d. random variables uniformly distributed in $[-0.5,0.5]$; see the upper left of Figure \ref{fig: 1d, a f} for a single realization. The right-hand side $f$ is drawn from the Gaussian process $\cN(0,(-\Delta)^{-0.5-\delta})$ for $\delta=10^{-2}$; this guarantees $f \in H^{t}(\Omega)$ for any $t<\delta$ but not $t\geq \delta$; see the upper right of Figure \ref{fig: 1d, a f} for a single realization of this process. Note that this set-up of $f$ ensures that it is roughly an element in $L^2(\Omega)$ and has no apparent higher regularity. This is important because we do not want $f$ to be too regular to influence the results, as our focus is on $f\in L^2(\Omega)$.
 \begin{figure}[!htb]
    \centering
    \includegraphics[width=6cm]{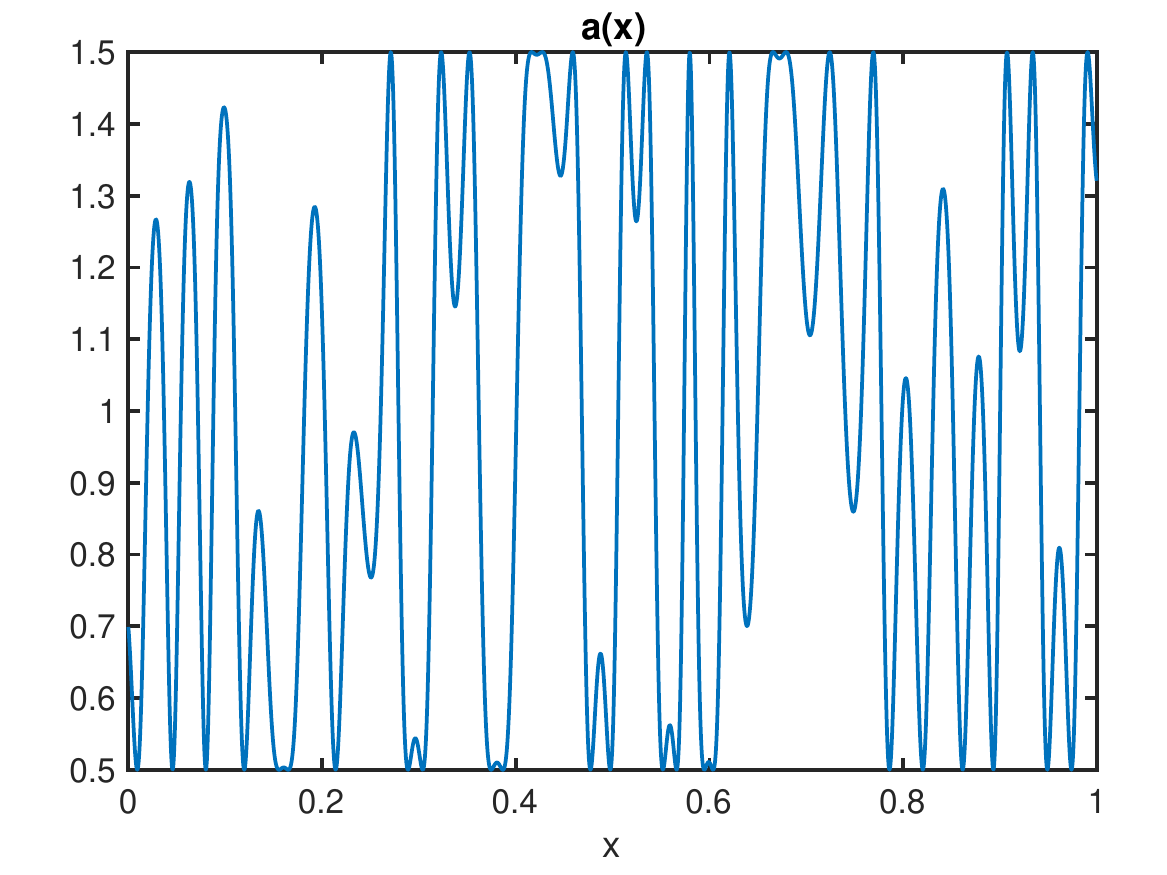}
    \includegraphics[width=6cm]{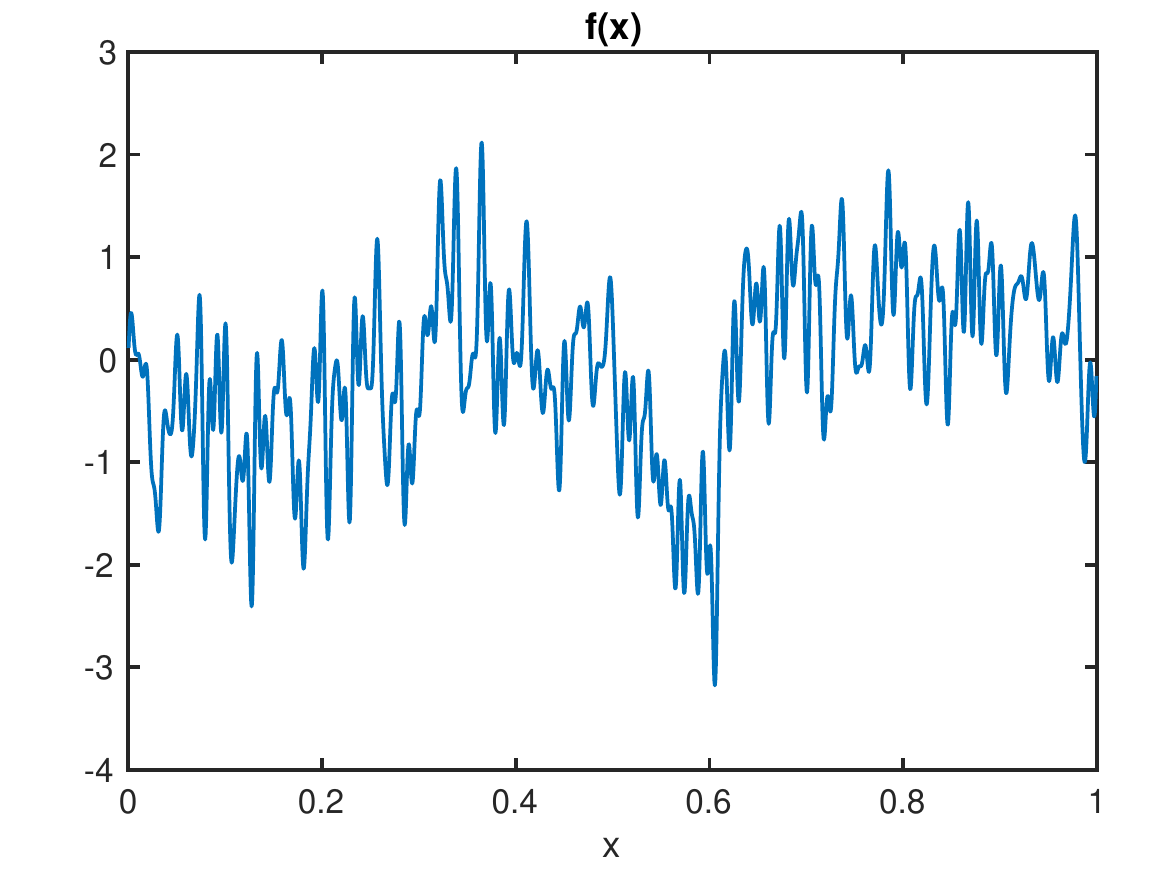}
    \includegraphics[width=6cm]{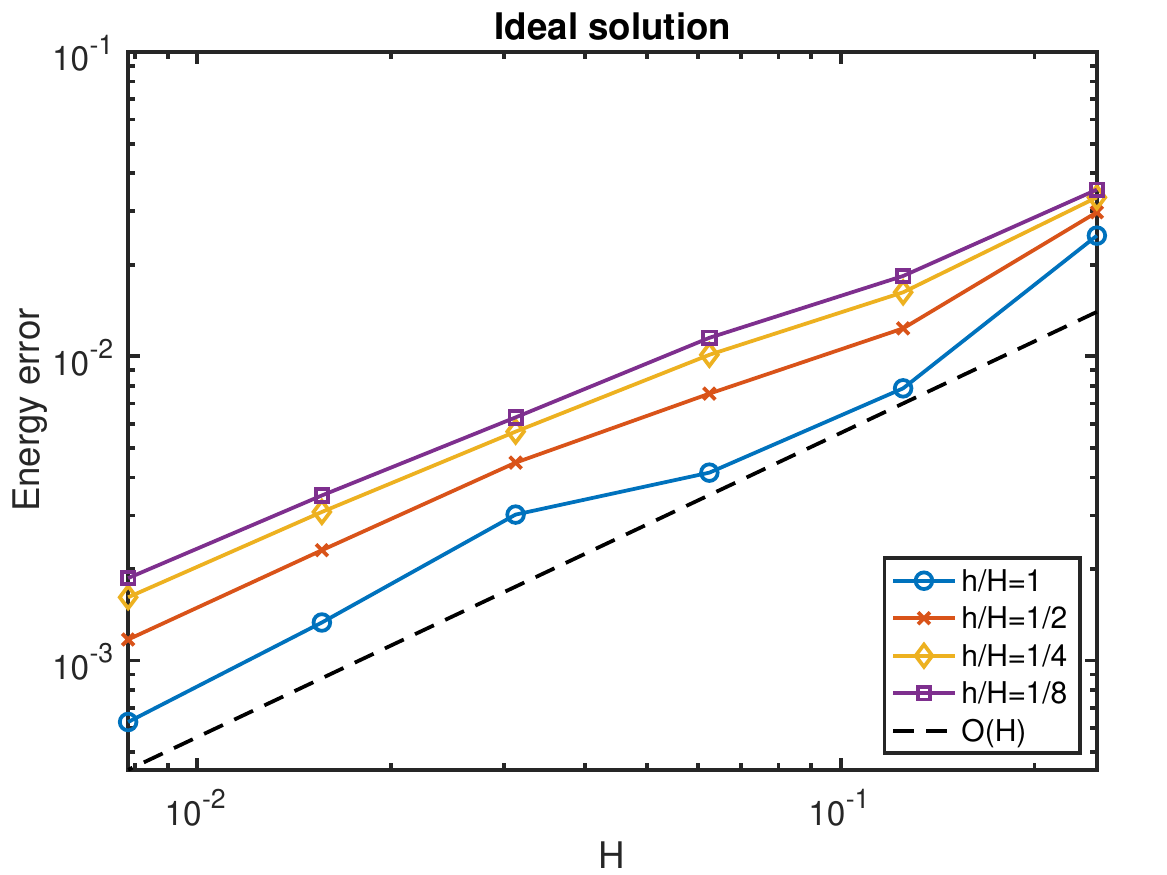}
    \includegraphics[width=6cm]{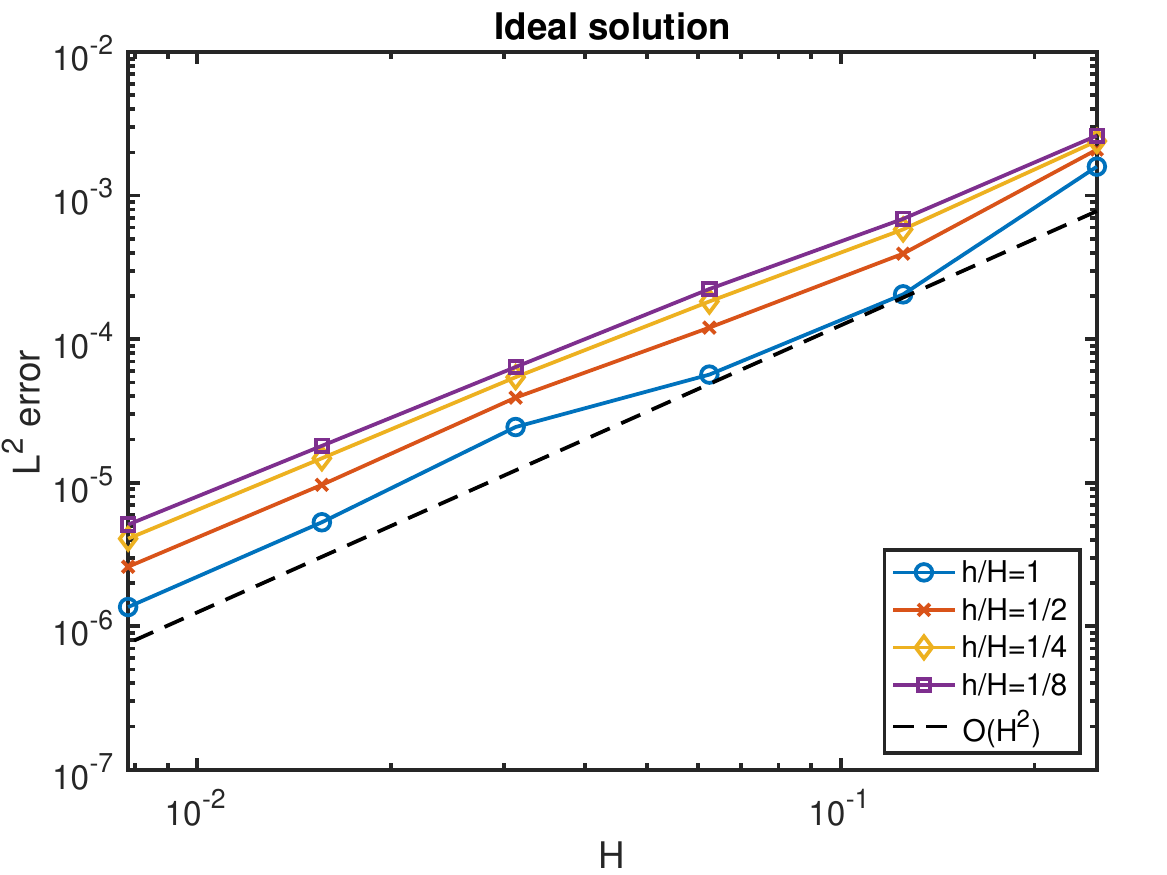}
    \caption{1D example, ideal solution. Upper left: $a(x)$; upper right: $f(x)$; lower left: energy error; lower right: $L^2$ error.}
    \label{fig: 1d, a f}
    \end{figure} 
    
In the lower part of Figure \ref{fig: 1d, a f}, we output the energy errors and $L^2$ errors of the ideal solution, $e^{h,H,\infty}_{1}(a,u)$ and $e^{h,H,\infty}_{0}(a,u)$, for $H=2^{-2},2^{-3},...,2^{-7}$ and the subsampled ratio $h/H=1,1/2,1/4,1/8$. The grid size we use to discretize the operator is set to be $2^{-11}$. These two figures lead to the following observations:
\begin{itemize}
    \item For the ideal solution, the energy error decays linearly with respect to the coarse scale $H$, while the $L^2$ error decays quadratically. 
    \item Decreasing $h$ leads to a decrease of accuracy.
\end{itemize}
In the next subsection, we move to a two dimensional example to further confirm these observations.
\subsubsection{Two Dimensional Example}
\label{subsec: ideal 2d experiments}
We consider $\Omega=[0,1]^2$. The coefficient $a(x)$ is chosen as 
\begin{equation}
\label{eqn: a(x)}
\begin{aligned} 
a(x)=\frac{1}{6}\left(\frac{1.1+\sin \left(2 \pi x_1 / \epsilon_{1}\right)}{1.1+\sin \left(2 \pi x_2 / \epsilon_{1}\right)}+\frac{1.1+\sin \left(2 \pi x_2 / \epsilon_{2}\right)}{1.1+\cos \left(2 \pi x_1 / \epsilon_{2}\right)}+\frac{1.1+\cos \left(2 \pi x_1 / \epsilon_{3}\right)}{1.1+\sin \left(2 \pi x_2 / \epsilon_{3}\right)}\right.\\\left.+\frac{1.1+\sin \left(2 \pi x_2 / \epsilon_{4}\right)}{1.1+\cos \left(2 \pi x_1 / \epsilon_{4}\right)}+\frac{1.1+\cos \left(2 \pi x_1 / \epsilon_{5}\right)}{1.1+\sin \left(2 \pi x_2 / \epsilon_{5}\right)}+\sin \left(4 x_1^{2} x_2^{2}\right)+1\right) \, ,\end{aligned}
\end{equation}
where $\epsilon_1=1/5$, $\epsilon_2=1/13$, $\epsilon_3=1/17$, $\epsilon_4=1/31$, $\epsilon_5=1/65$. For the right-hand side, we sample two independent one-dimensional process in the last subsection, denoted by $f_1(x_1)$ and $f_2(x_2)$, and we set
$f(x)=f_1(x_1)f_2(x_2)$. This guarantees $f \in H^{t}(\Omega)$ for any $t<\delta$ but not $t\geq \delta$ in two dimensions.
 \begin{figure}[!htb]
    \centering
    \includegraphics[width=6cm]{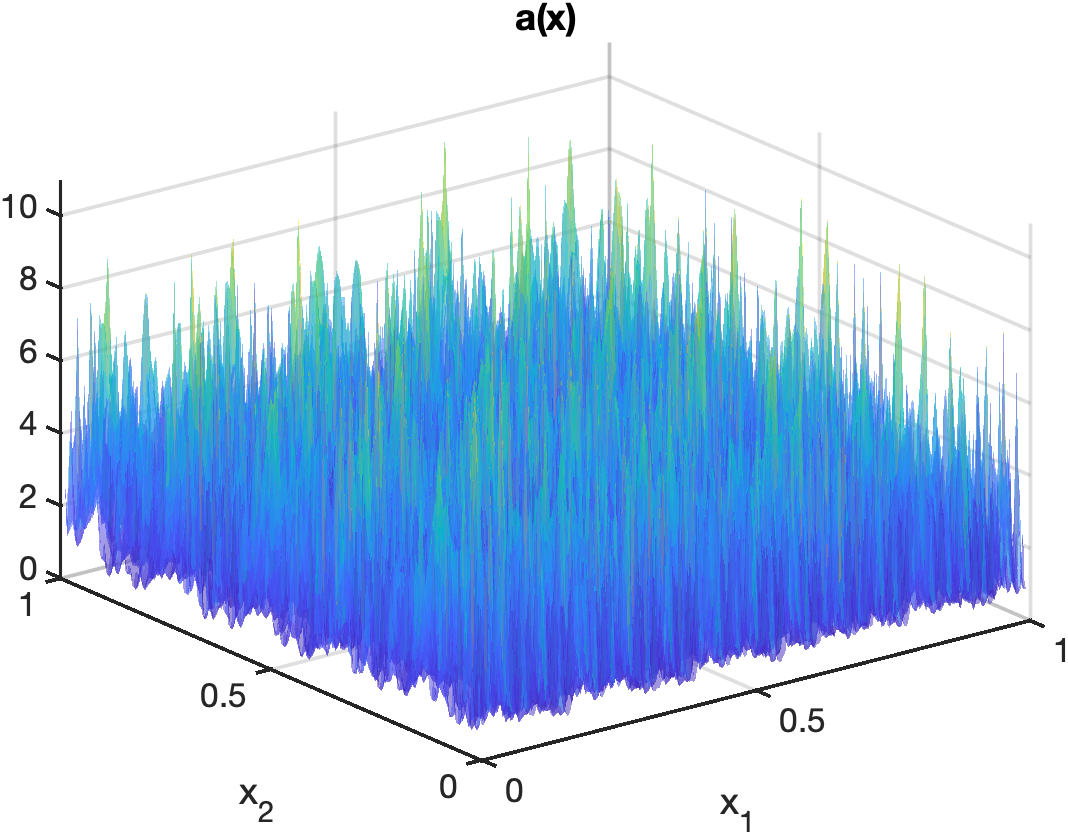}
    \includegraphics[width=6cm]{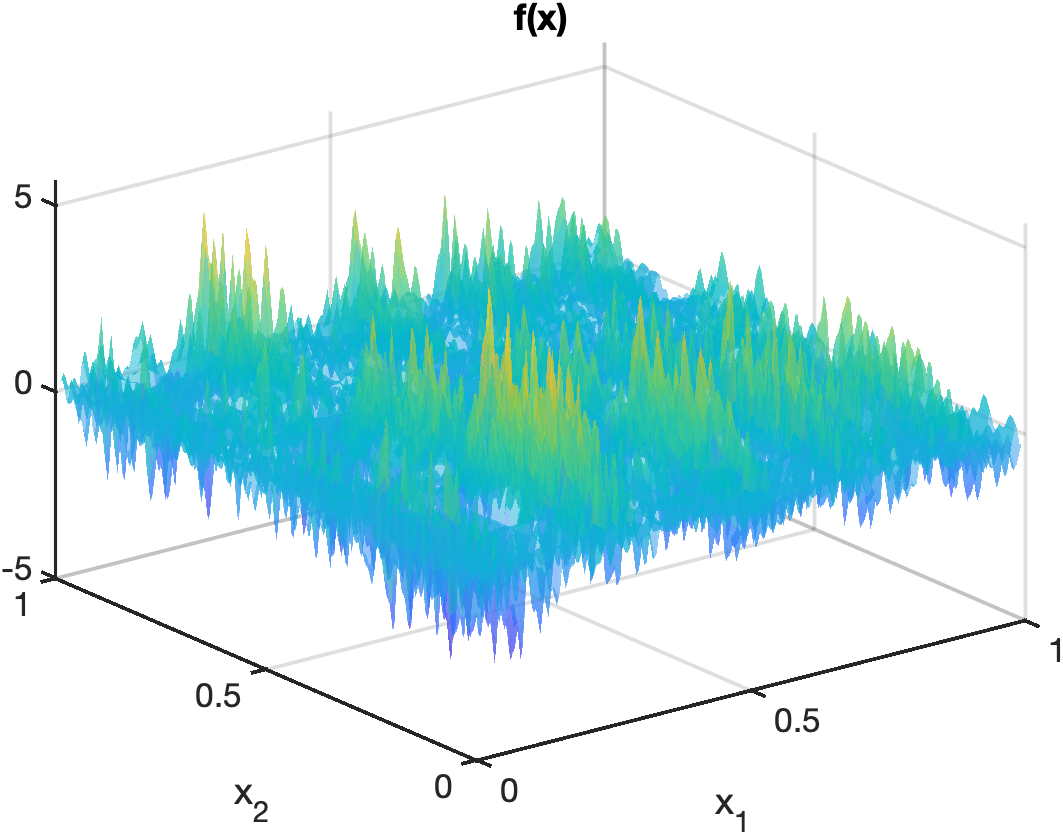}
    \includegraphics[width=6cm]{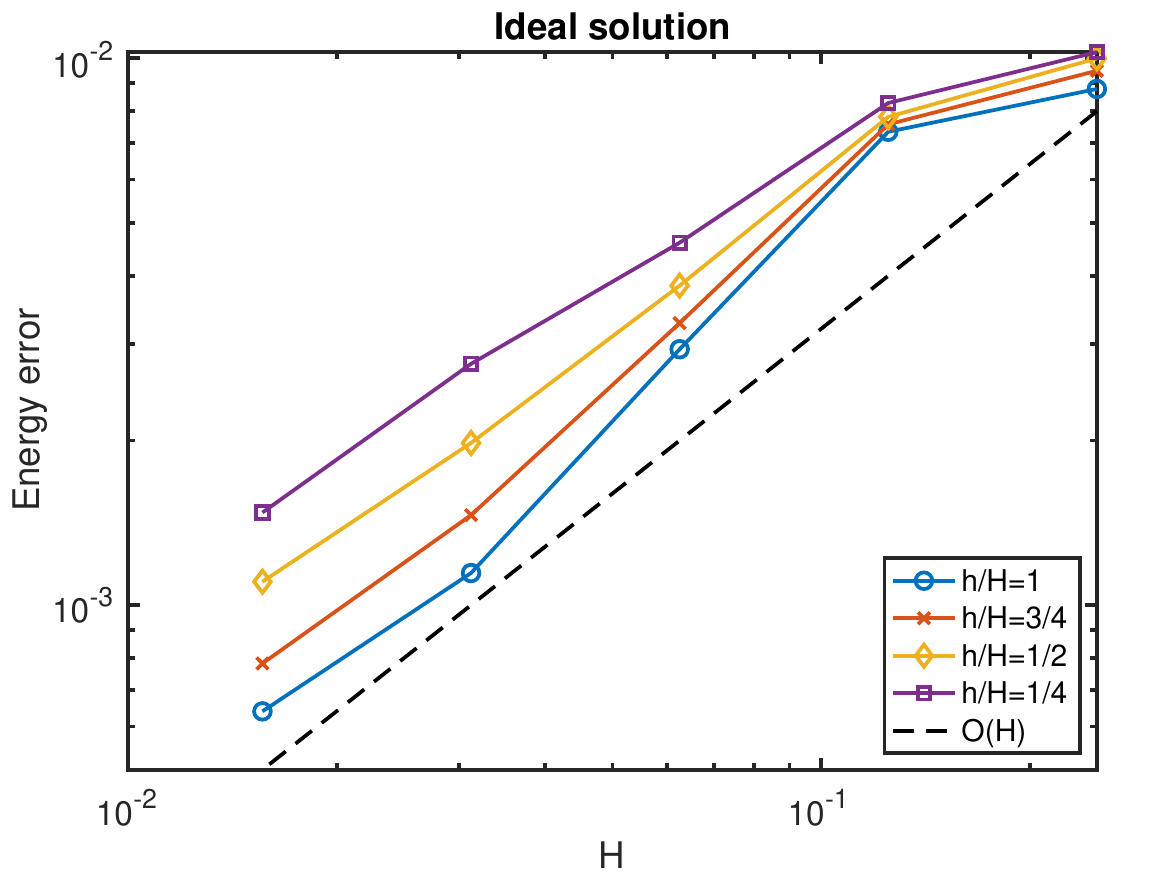}
    \includegraphics[width=6cm]{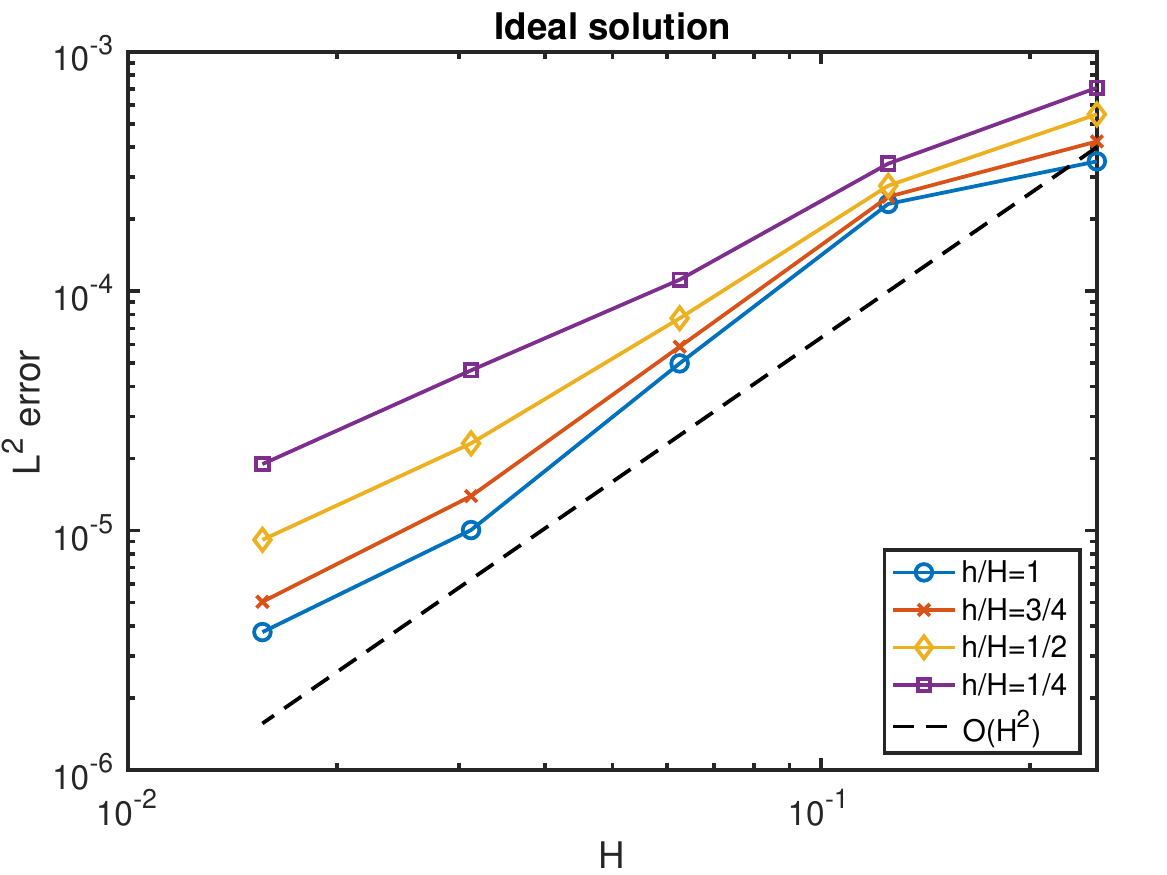}
    \caption{2D example, ideal solution. Upper left: $a(x)$; upper right: $f(x)$; lower left: energy error; lower right: $L^2$ error.}
    \label{fig: 2d, a f}
    \end{figure} 
    
In the upper part of Figure \ref{fig: 2d, a f}, we output $a(x)$ and a single realization of $f(x)$. The lower part depicts $e^{h,H,\infty}_{1}(a,u)$ and $e^{h,H,\infty}_{0}(a,u)$, for $H=2^{-2},2^{-3},...,2^{-6}$ and the subsampled ratio $h/H=1,3/4,1/2,1/4$. The grid size we use to discretize the operator is set to be $2^{-8}$. These two figures yield the same conclusions as those in the one dimensional case.
\subsection{Analysis: Ideal Solution}
\label{subsec: Analysis: Ideal Solution}
In this subsection, we move to the theoretical analysis of the ideal solution, to understand better of the above empirical observations. 

For this purpose, we use our earlier results in function approximation via subsampled data \cite{chen2019function}. Especially, Theorem 3.3 in \cite{chen2019function} implies the following result:

\begin{theorem}
\label{thm: err ideal sol}
For the ideal solution, it holds that
\begin{align}
    &e^{h,H,\infty}_{1}(a,u) \leq \frac{1}{\sqrt{a_{\min}}}C_1(d)H\rho_{2,d}(\frac{H}{h})\|\cL u\|_{L^2(\Omega)} \, ;\\
    &e^{h,H,\infty}_{0}(a,u) \leq \frac{1}{a_{\min}}C_1(d)^2H^2\left(\rho_{2,d}(\frac{H}{h})\right)^2\|\cL u\|_{L^2(\Omega)}\, ,
\end{align}
where, $C_1(d)$ is a constant that depends on the dimension $d$ only, and for $p,d\geq 1$, the function $\rho_{p,d}:\bR_{+}\to \bR_{+}$  is defined as:
\begin{equation}
    \label{eqn: sharp rate}
    \rho_{p,d}(t) =\left\{
    \begin{aligned}
    1, \quad &d < p \\
    (\log (1+t))^{\frac{d-1}{d}}, \quad &d=p  \\
    t^{\frac{d-p}{p}}, \quad  &d> p \, .
    \end{aligned}
    \right.
    \end{equation}
\end{theorem}

In Theorem \ref{thm: err ideal sol}, we get the upper bound of $e^{h,H,\infty}_{1}(a,u)$ and $e^{h,H,\infty}_{0}(a,u)$. The dependence of this upper bound on $h$ is determined by the function $\rho_{2,d}$. Note that it is a non-decreasing function, so as $h$ decreases, for a fixed $H$, the ratio $H/h$ increases, and the upper bound will also increase. One exception is when $d=1$, the upper bound remains constant when $h$ changes, and it is still finite even when $h$ approaches $0$. This phenomenon is in sharp contrast with the case $d\geq 2$, where as $h \to 0$, the upper bound blows up to infinity.

The above theoretical implications match what we have observed in the experiments -- reducing $h$ leads to a decrease of accuracy, both in $d=1$ and $d=2$; moreover, the deterioration of accuracy is more severe in $d=2$ than $d=1$.

Therefore, if one is adopting the ideal solution, without considering computational costs, then we would recommend choosing $h=H$, which achieves the best of both worlds with a theoretical guarantee and practical performance.

\begin{remark}
 Applying the above recommendation ($h=H$) is straightforward in the context of numerical upscaling -- we can choose the suitable upscaled coarse variables. Nevertheless, for scattered data approximation, the data acquisition step also matters. Our analysis suggests that for the sake of accuracy (in the case there is no burden of computational costs), it could be a good idea to make the lengthscale of the coarse data larger; this provides guidance for data collection in such a scenario.
\end{remark}
\subsection{Experiments: Localized Solution} Solving the ideal solution can be computationally expensive due to the global optimization problem \eqref{eqn: optimization def basis}. This is also why we stop at $H=2^{-6}$ and do not decrease $H$ further in the previous 2D experiments. For better practical algorithms, in this subsection, we move to the localized solution. We start with the numerical experiments for 1D and 2D, followed by theoretical analysis. In these experiments, we use the same functions $a(x)$ and $f(x)$ as in the ideal case.

In the localized scenario, the Galerkin solution in numerical upscaling and the recovery solution in scattered data approximation are different. Thus, we will compute them separately and compare the results. More precisely, for the Galerkin solution, we use the localized basis functions in the Galerkin framework to solve the PDE; for the recovery solution, it is simpler -- once the basis functions are computed, we readily get the recovery solution by using the available subsampled data and the formula \eqref{eqn: loc sol}. For both cases, the ground truth solution $u$ is given as a solution to a PDE.

\subsubsection{One Dimensional Example} We consider the 1D model in Subsection \ref{subsec: ideal 1d experiments}. We compute the Galerkin errors $\tilde{e}^{h,H,l}_{1}(a,u)$ and  $\tilde{e}^{h,H,l}_{0}(a,u)$ and the recovery errors $e^{h,H,l}_{1}(a,u)$ and $e^{h,H,l}_{0}(a,u)$, for $H=2^{-2},2^{-3},...,2^{-7}$, $h/H=1,1/2,1/4,1/8$ and $l=2,4$. The grid size we use to discretize the operator is set to be $2^{-11}$.

In Figure \ref{fig: 1d, localized l=2}, the oversampling parameter $l=2$. The upper part depicts the energy and $L^2$ errors of the Galerkin solution, while the lower part corresponds to that of the recovery solution. From the figure, we observe the following facts:
\begin{itemize}
    \item Due to localization, the error line of $h/H=1,1/2,1/4$ finally turns up as we make $H$ very small, deviating from what we have observed in the ideal solution. This implies the localization error matters a lot.
    \item Among the four choices, the case $h/H=1/8$ that corresponds to the smallest $h$, behaves the best for small $H$. It appears that decreasing $h$ may suppress the localization error to certain extent.
    \item The $L^2$ error of the recovery solution is more stable and accurate compared to the Galerkin solution, when $H$ is small. Especially, there is no obvious blow-up as $H$ becomes small.
\end{itemize}
\begin{figure}[!htb]
    \centering
    \includegraphics[width=6cm]{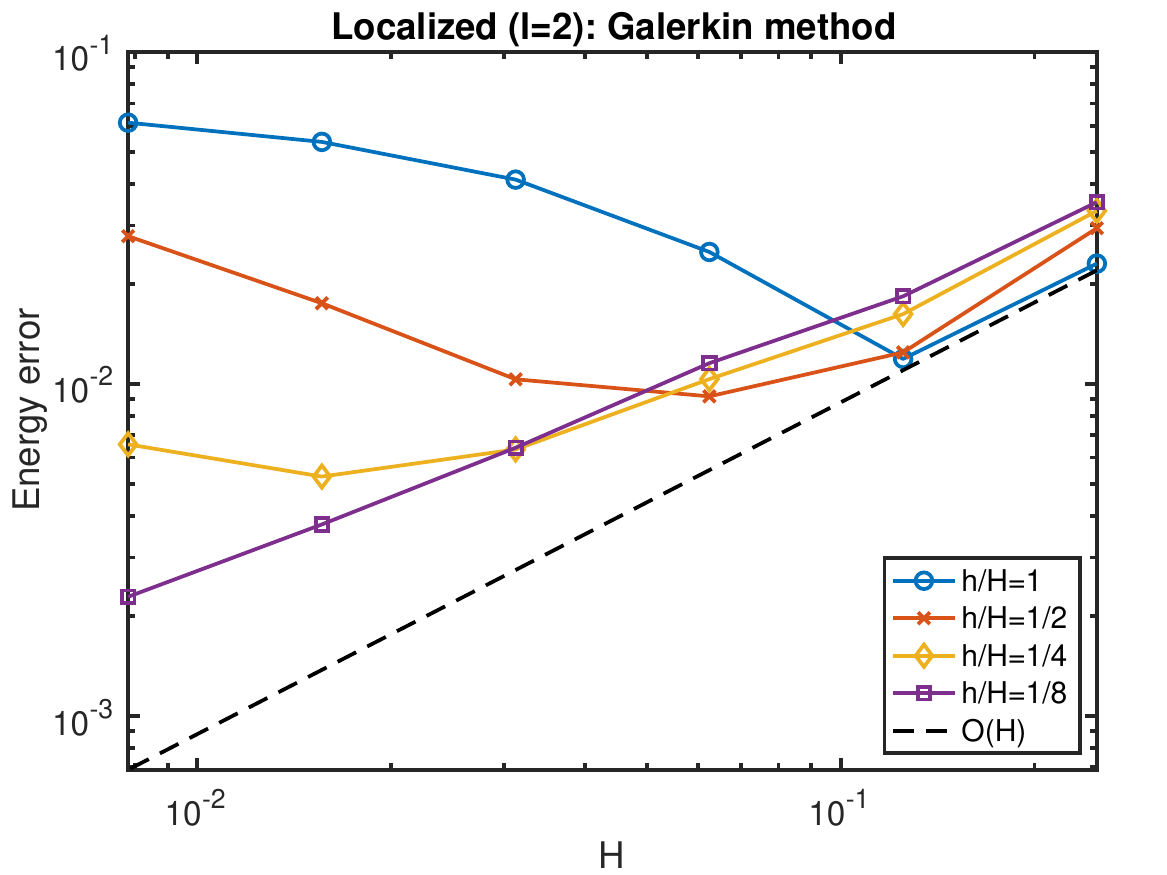}
    \includegraphics[width=6cm]{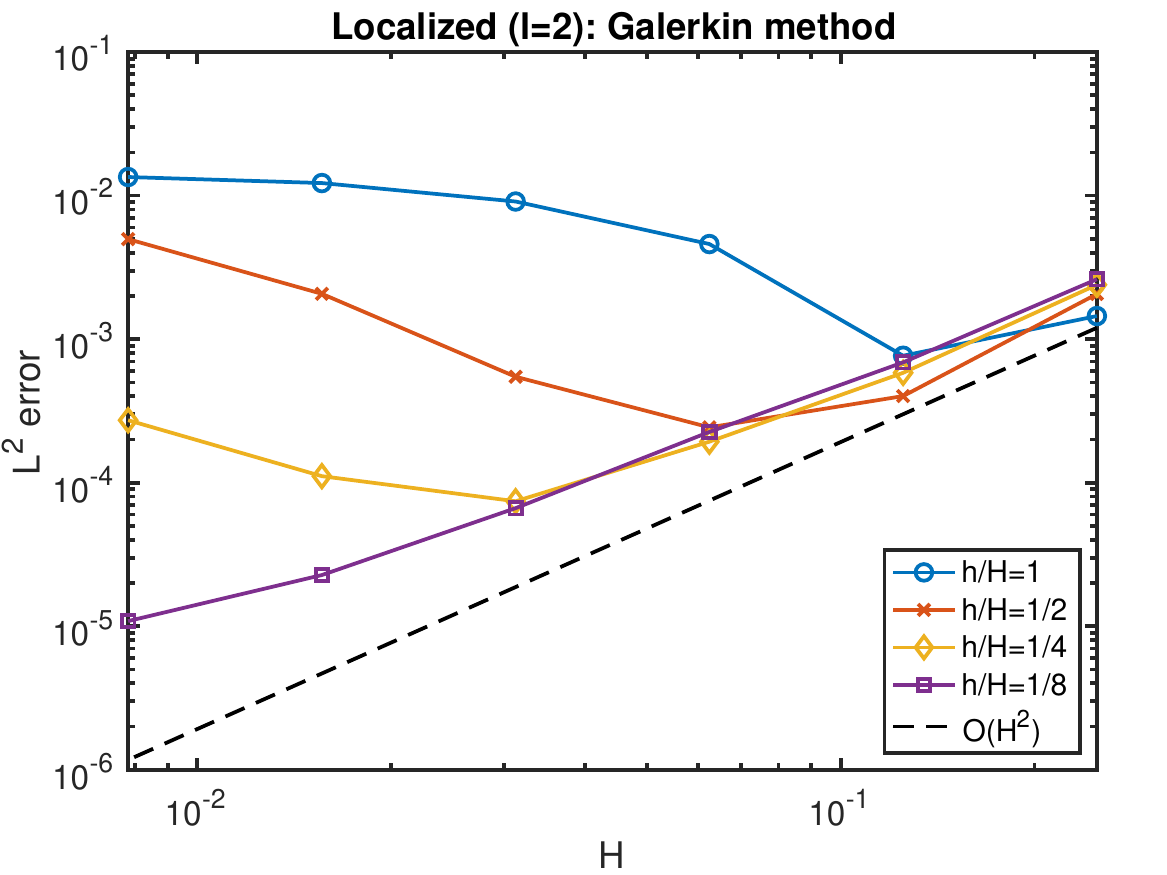}
    \includegraphics[width=6cm]{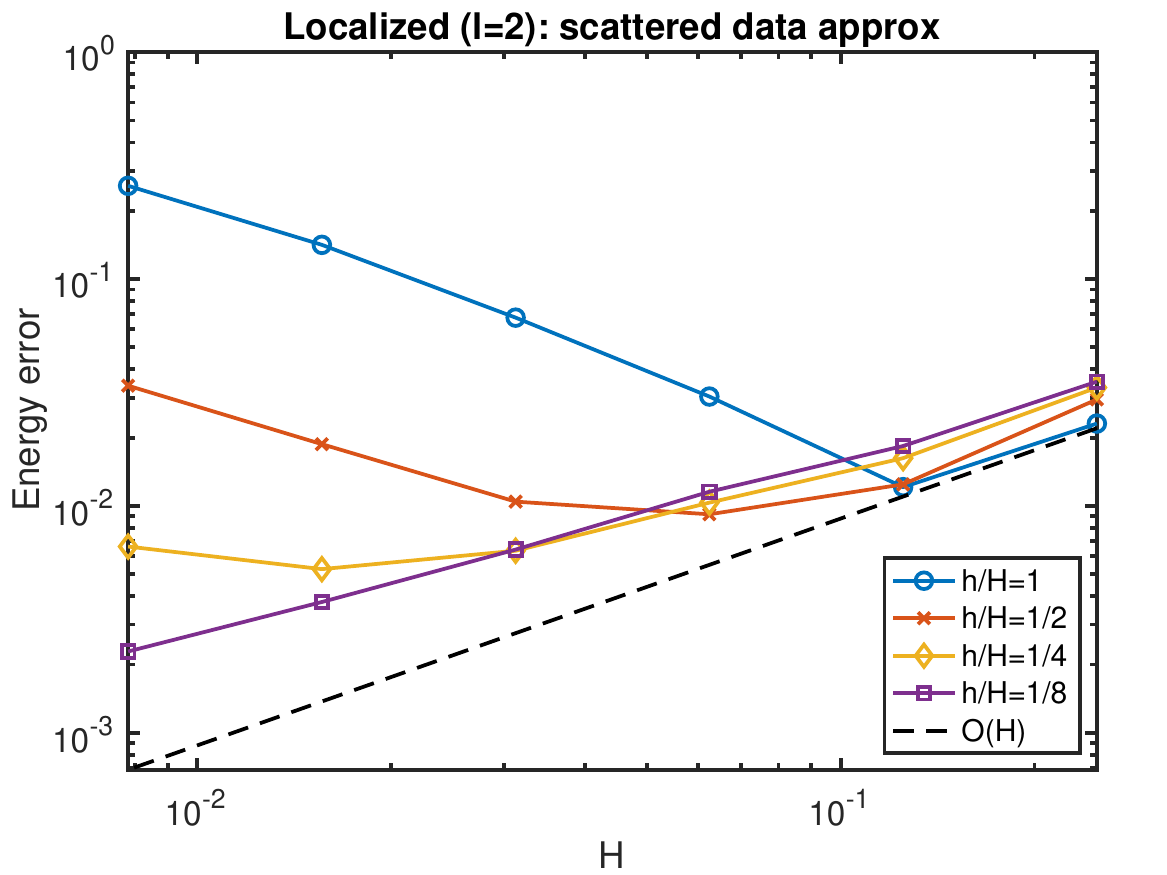}
    \includegraphics[width=6cm]{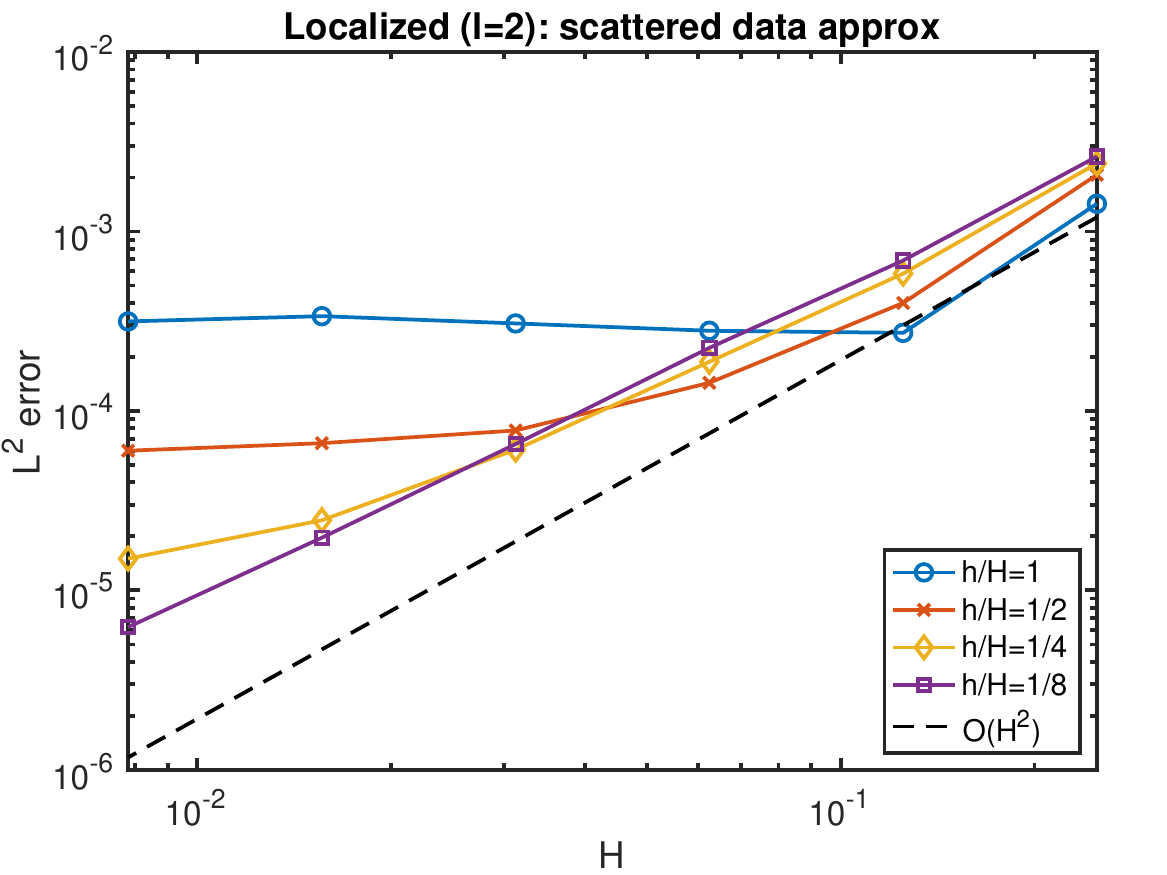}
    \caption{1D example, localized solution $l=2$. Upper left: $\tilde{e}^{h,H,l}_{1}(a,u)$; upper right: $\tilde{e}^{h,H,l}_{0}(a,u)$; lower left: $e^{h,H,l}_{1}(a,u)$; lower right: $e^{h,H,l}_{0}(a,u)$.}
    \label{fig: 1d, localized l=2}
    \end{figure} 
Next, we increase the oversampling parameter to $l=4$, and output the same set of observables in Figure \ref{fig: 1d, localized l=4}. Now, only the case $h/H=1$ leads to a turning up of the error line, while the other three cases lead to similar error lines as the ideal solution. The best choice among the four becomes $h/H=1/2$. Thus, as $l$ increases, the localized solution is approaching the ideal one, and choosing a larger $h$ would be good.
\begin{figure}[!htb]
    \centering
    \includegraphics[width=6cm]{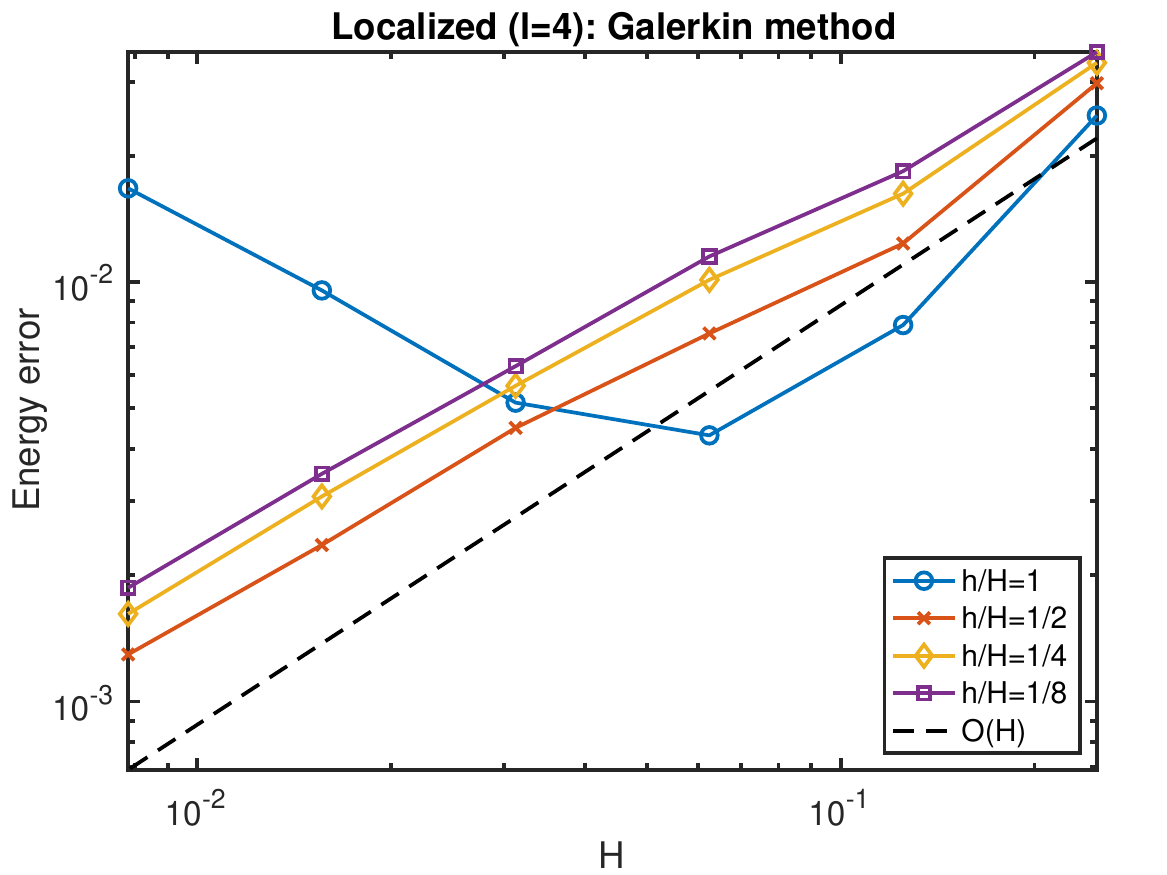}
    \includegraphics[width=6cm]{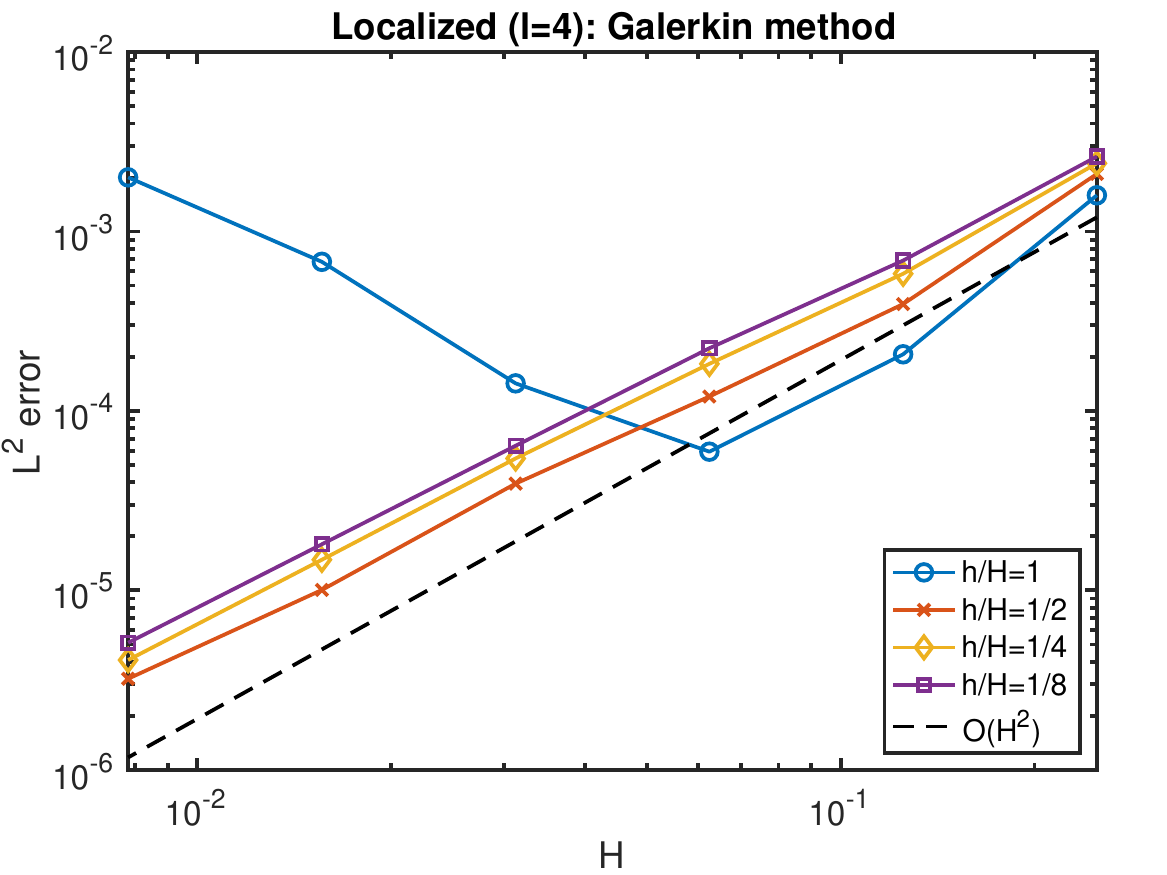}
    \includegraphics[width=6cm]{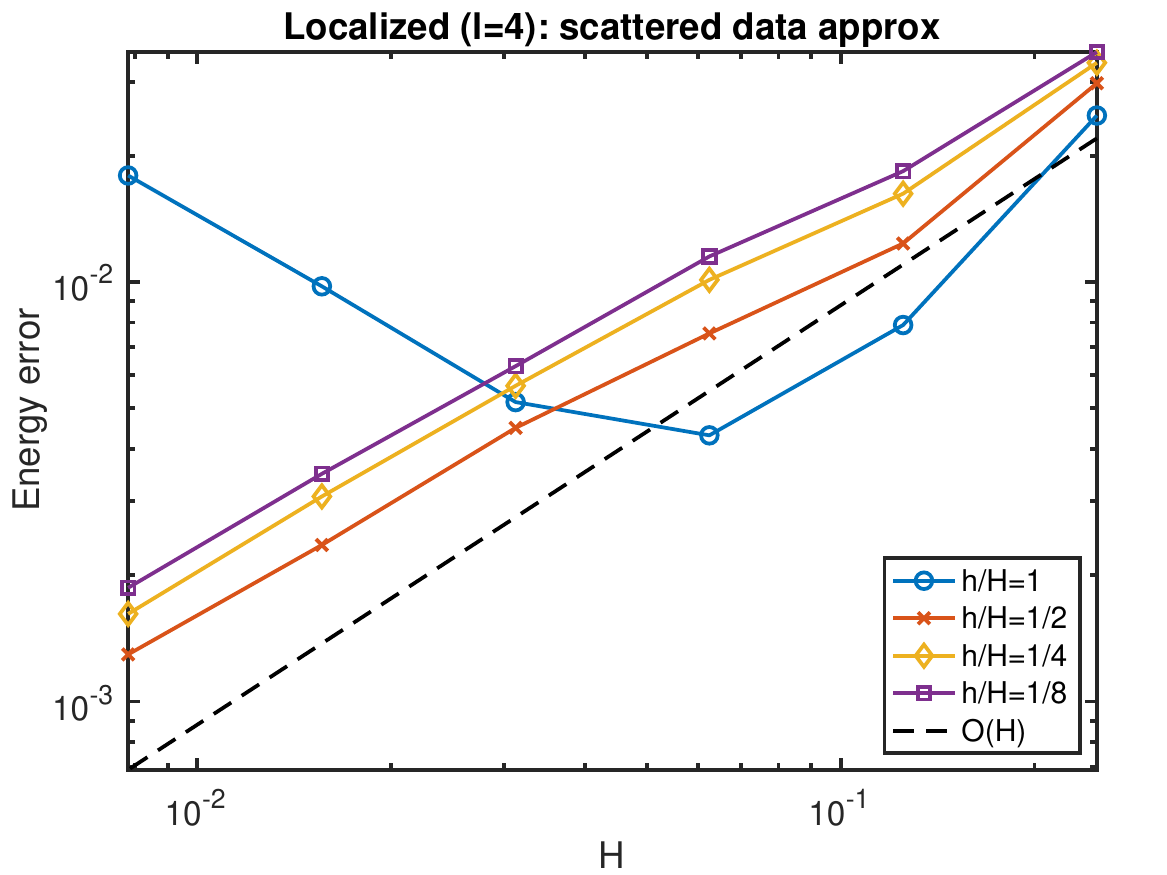}
    \includegraphics[width=6cm]{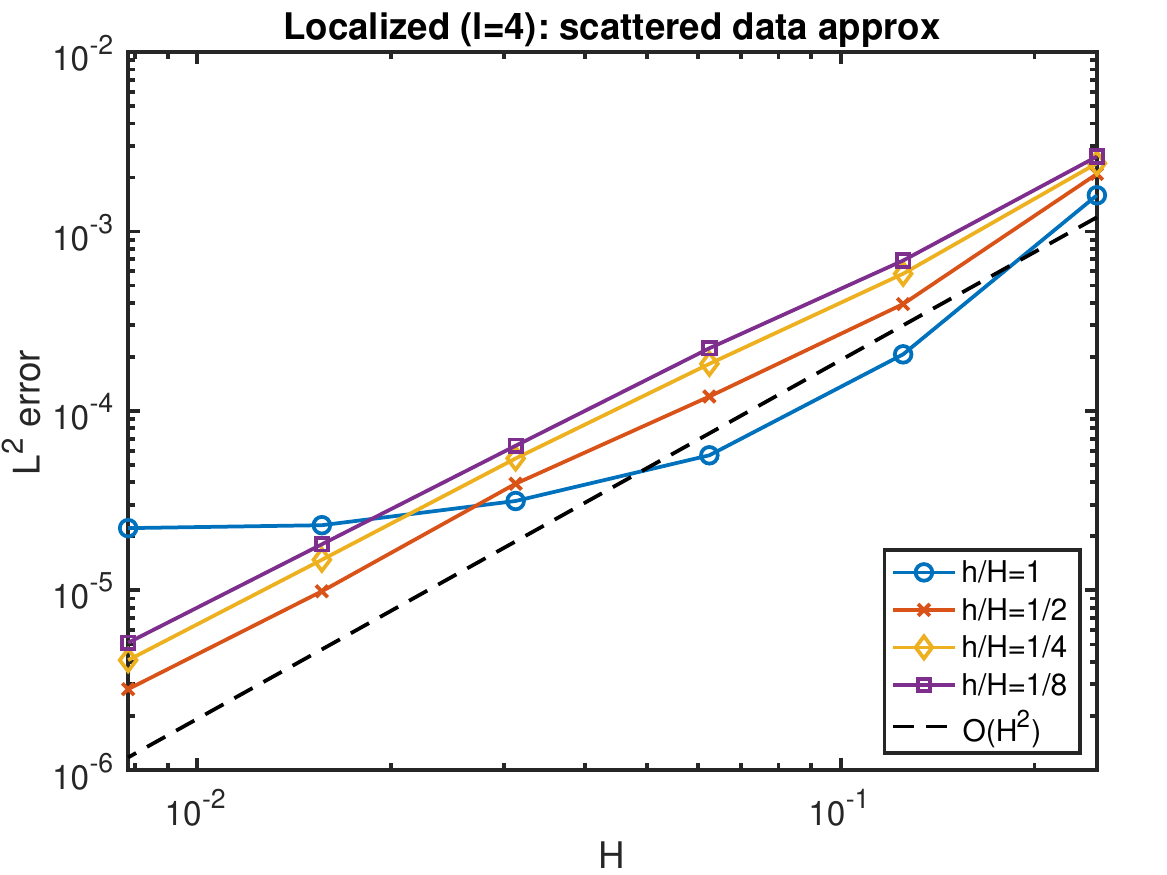}
    \caption{1D example, localized solution $l=4$. Upper left: $\tilde{e}^{h,H,l}_{1}(a,u)$; upper right: $\tilde{e}^{h,H,l}_{0}(a,u)$; lower left: $e^{h,H,l}_{1}(a,u)$; lower right: $e^{h,H,l}_{0}(a,u)$.}
    \label{fig: 1d, localized l=4}
    \end{figure} 
\subsubsection{Two Dimensional Example} In this subsection, we move to a two dimensional example that corresponds to the the ideal case in Subsection \ref{subsec: ideal 2d experiments}. As before, we compute the Galerkin errors $\tilde{e}^{h,H,l}_{1}(a,u)$ and $\tilde{e}^{h,H,l}_{0}(a,u)$ and the recovery errors $e^{h,H,l}_{1}(a,u)$ and $e^{h,H,l}_{0}(a,u)$, for $H=2^{-2},2^{-3},...,2^{-8}$, $h/H=1,3/4,1/2,1/4$ and $l=2,4$. The grid size we use to discretize the operator is set to be $2^{-10}$.

We start with $l=2$, in Figure \ref{fig: 2d, localized l=2}. Our observations are as follows:
\begin{itemize}
    \item All the error lines deviate from the desired $O(H)$ or $O(H^2)$ line to some extent, and among the four choices, the ratio $h/H=1/2$ performs the best when $H$ is small. 
    \item Compared to the 1D example, the localization errors in 2D are larger, since the deviation from the desired $O(H)$ or $O(H^2)$ line is more apparent.
    \item The error line exhibits a turning up behavior even for very small $h/H=1/4$. That means in the 2D case, small $h$ can also lead to large overall errors. This observation indeed matches our theory for the ideal solution, as $\rho_{2,d}(H/h)$ in Theorem \ref{thm: err ideal sol} will blow up as $h\to 0$, when $d=2$.
    \item When $H$ is small, the $L^2$ error of the recovery solution in the scattered data approximation is more accurate than the Galerkin solution in numerical upscaling. This phenomenon has also been observed in the 1D example.
\end{itemize}
Then, we increase the oversampling parameter to $l=4$. The results are output in Figure \ref{fig: 2d, localized l=4}. We observe a better accuracy and more stable behavior of the error lines compared to $l=2$. Now the best among the four ratios becomes $h/H=3/4$. Moreover, the relative behaviors of the three cases $h/H=3/4,1/2,1/4$ are very similar to that in the ideal solution, indicating that when $l=4$, the localization error may be small compared to the approximation error of the ideal solution.
\begin{figure}[!htb]
    \centering
    \includegraphics[width=6cm]{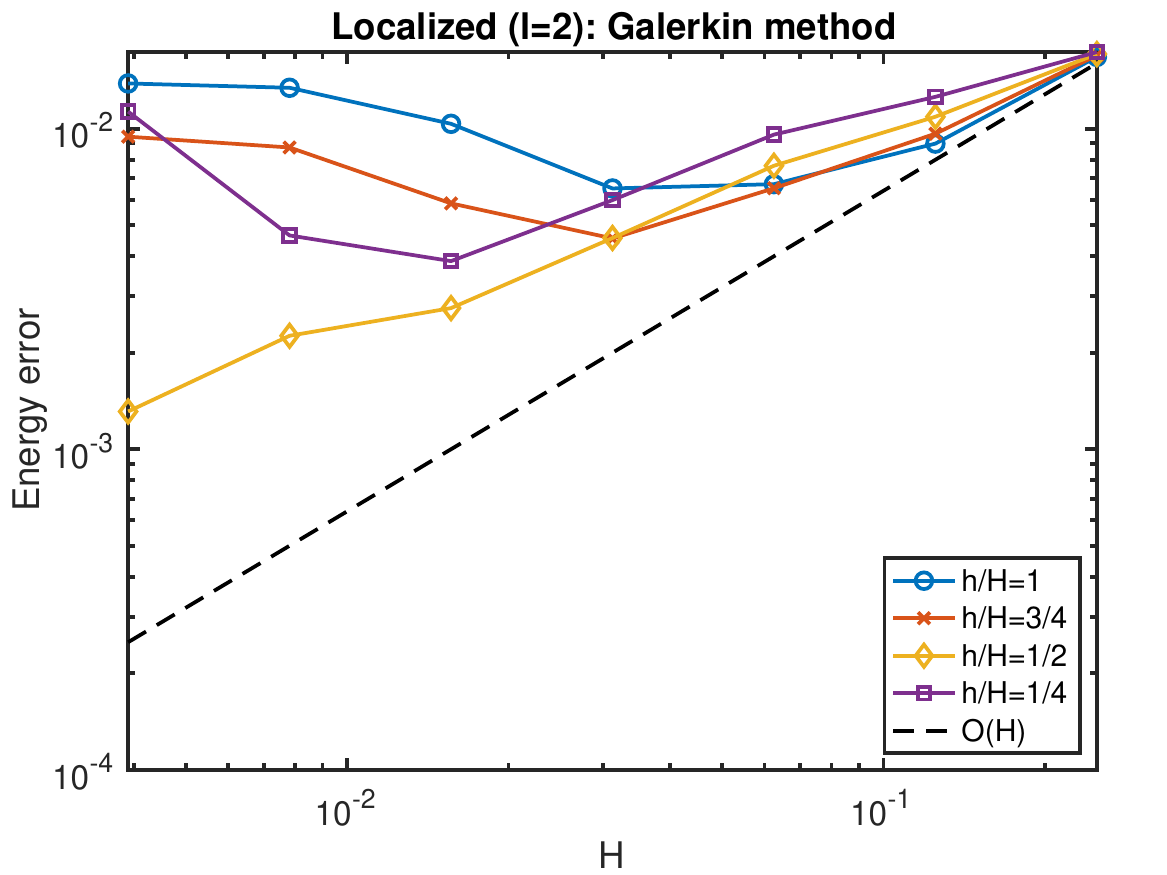}
    \includegraphics[width=6cm]{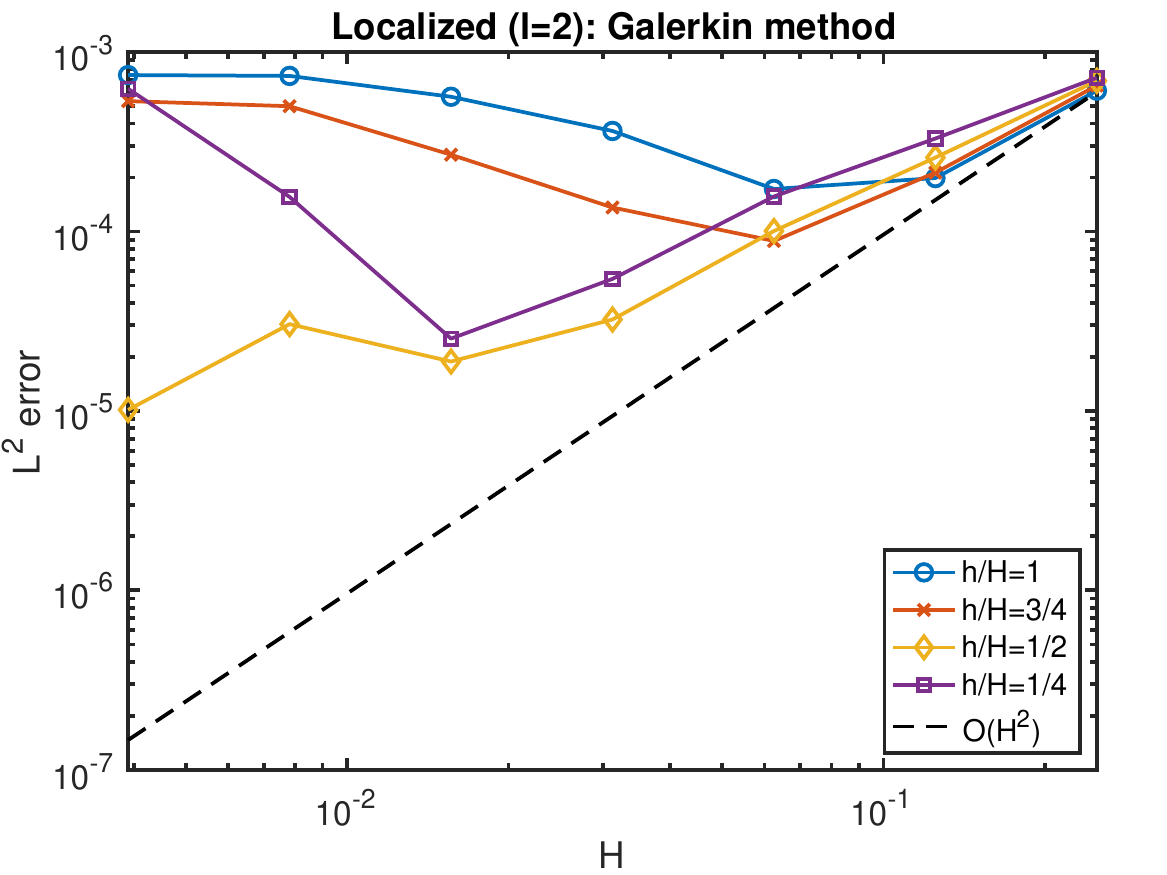}
    \includegraphics[width=6cm]{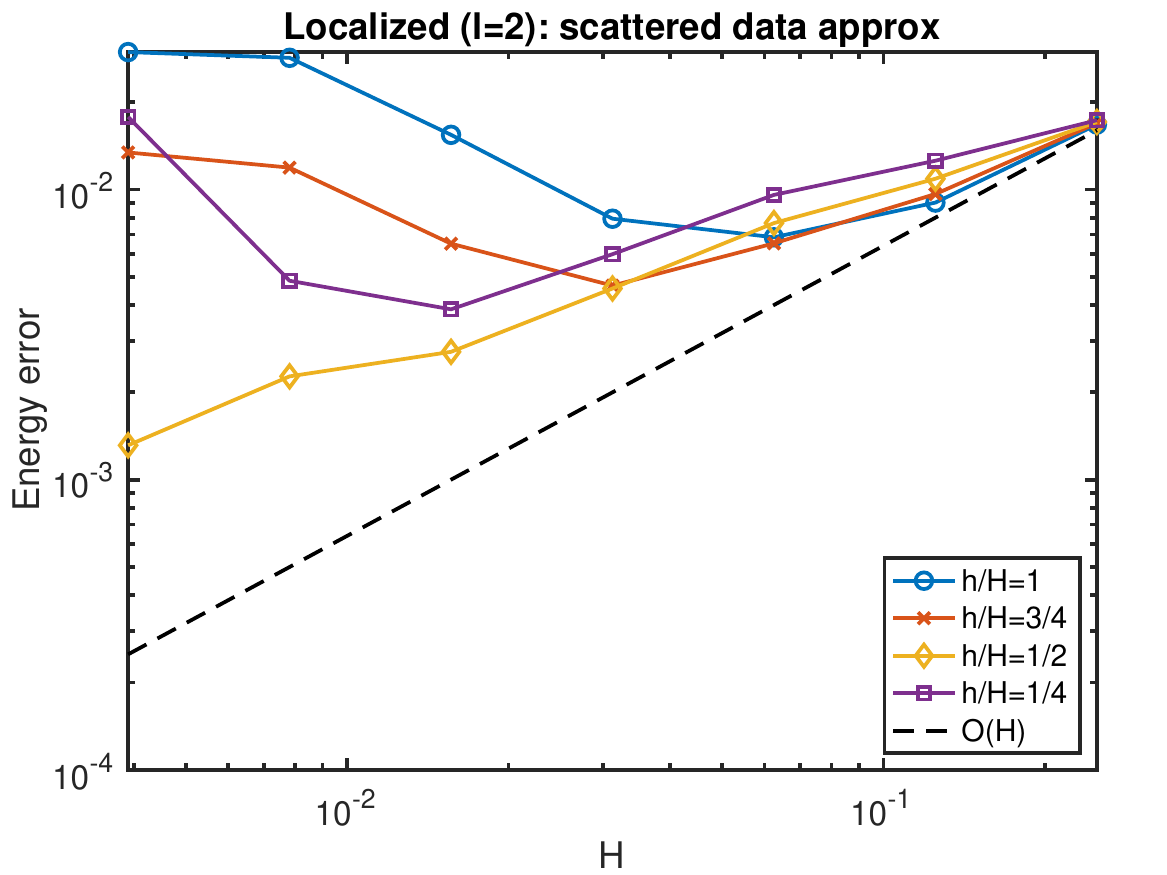}
    \includegraphics[width=6cm]{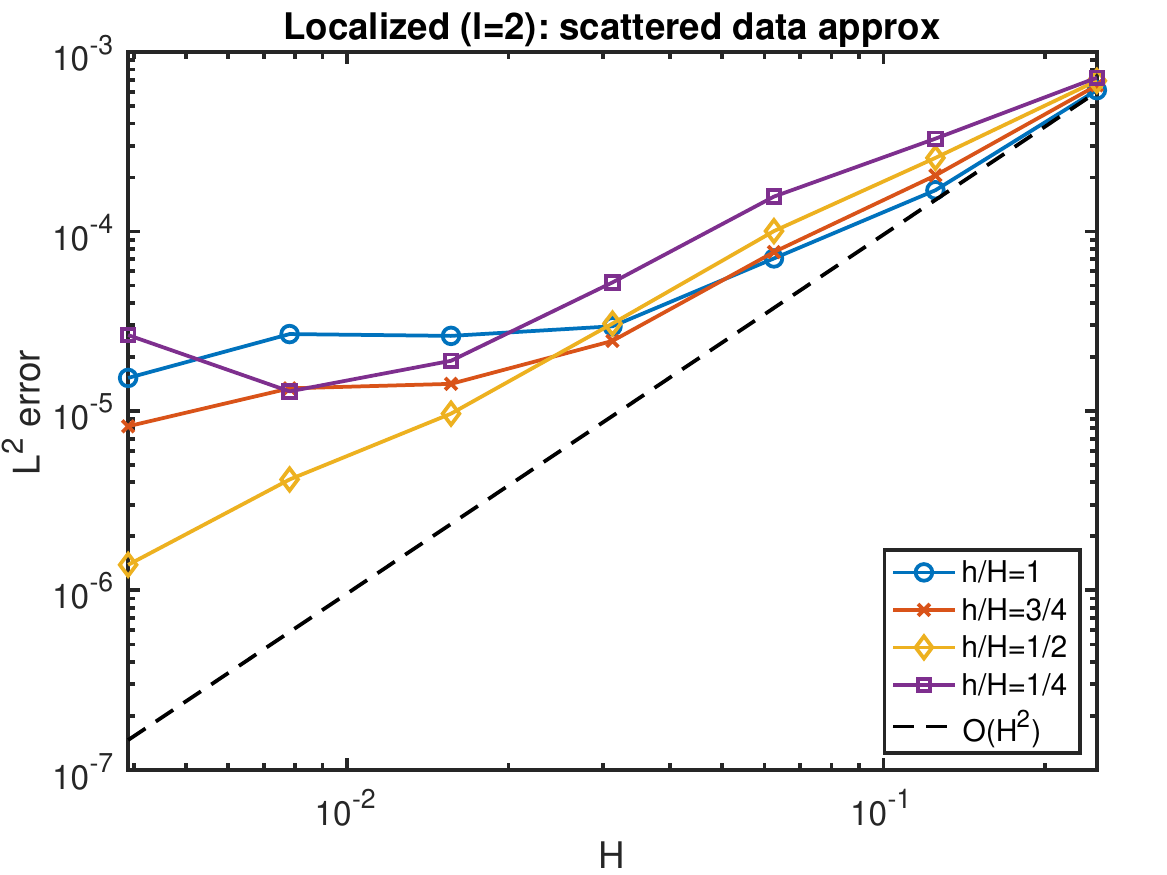}
    \caption{2D example, localized solution $l=2$. Upper left: $\tilde{e}^{h,H,l}_{1}(a,u)$; upper right: $\tilde{e}^{h,H,l}_{0}(a,u)$; lower left: $e^{h,H,l}_{1}(a,u)$; lower right: $e^{h,H,l}_{0}(a,u)$.}
    \label{fig: 2d, localized l=2}
    \end{figure} 
    
\begin{figure}[!htb]
    \centering
    \includegraphics[width=6cm]{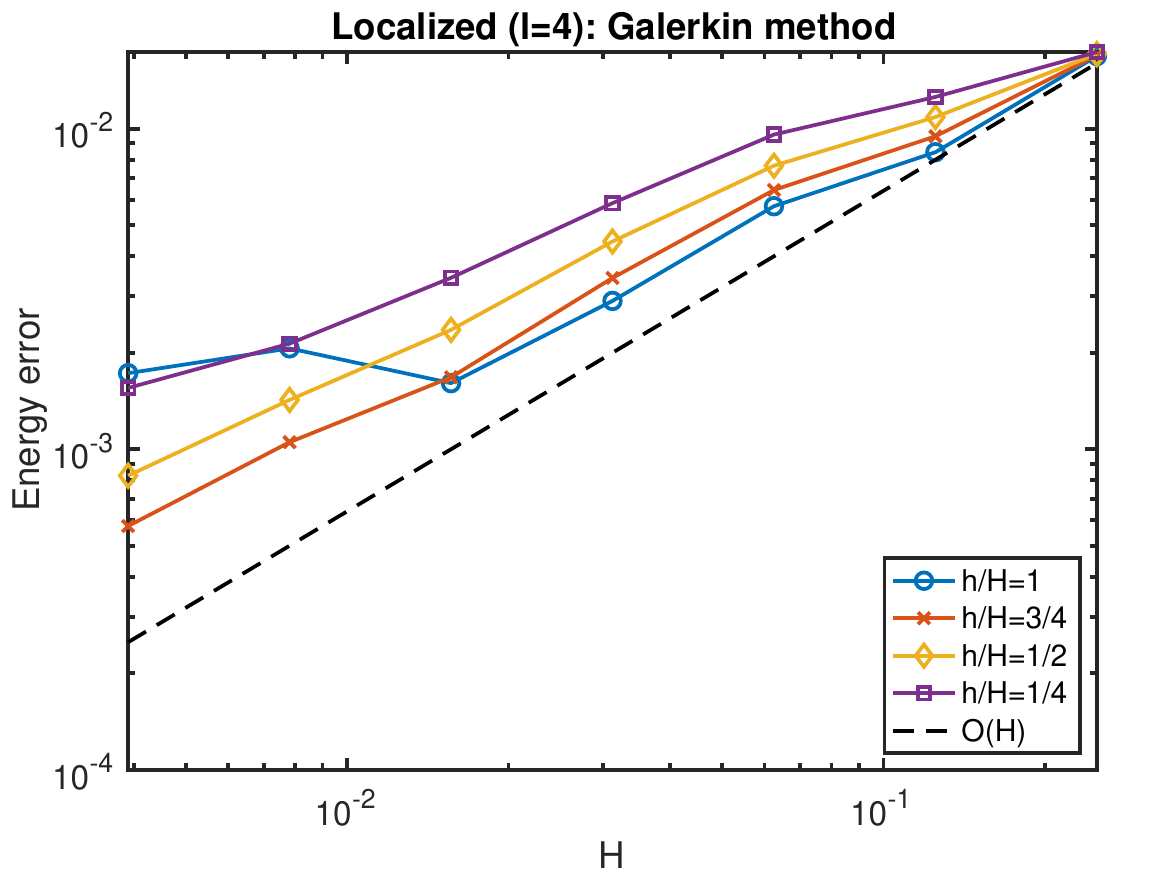}
    \includegraphics[width=6cm]{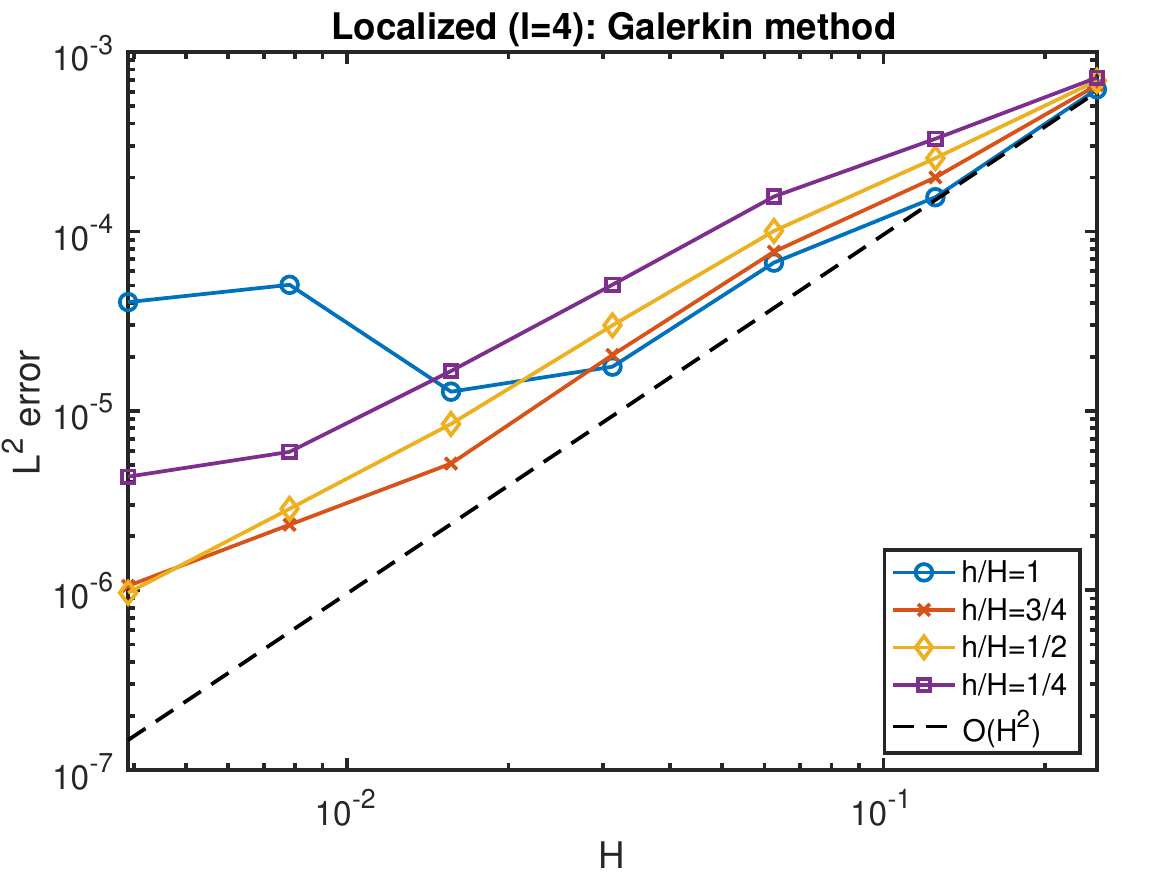}
    \includegraphics[width=6cm]{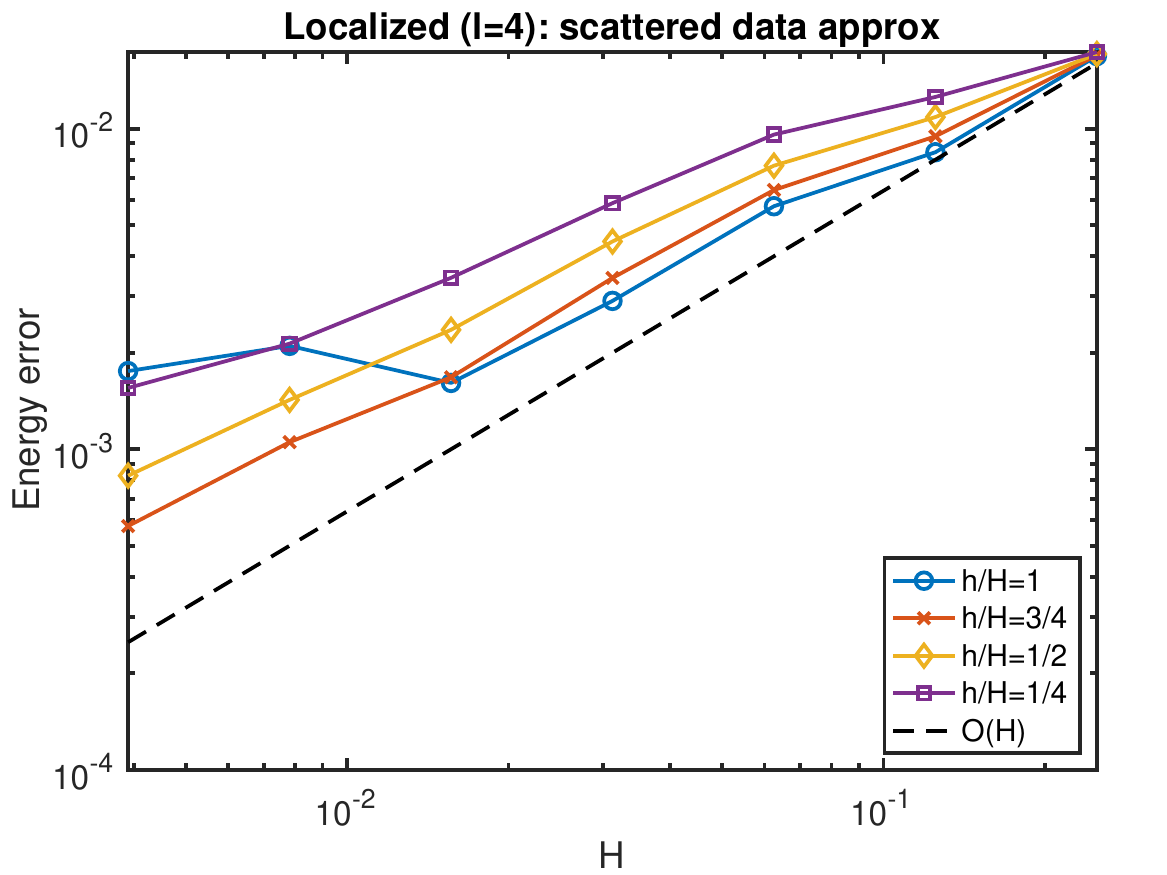}
    \includegraphics[width=6cm]{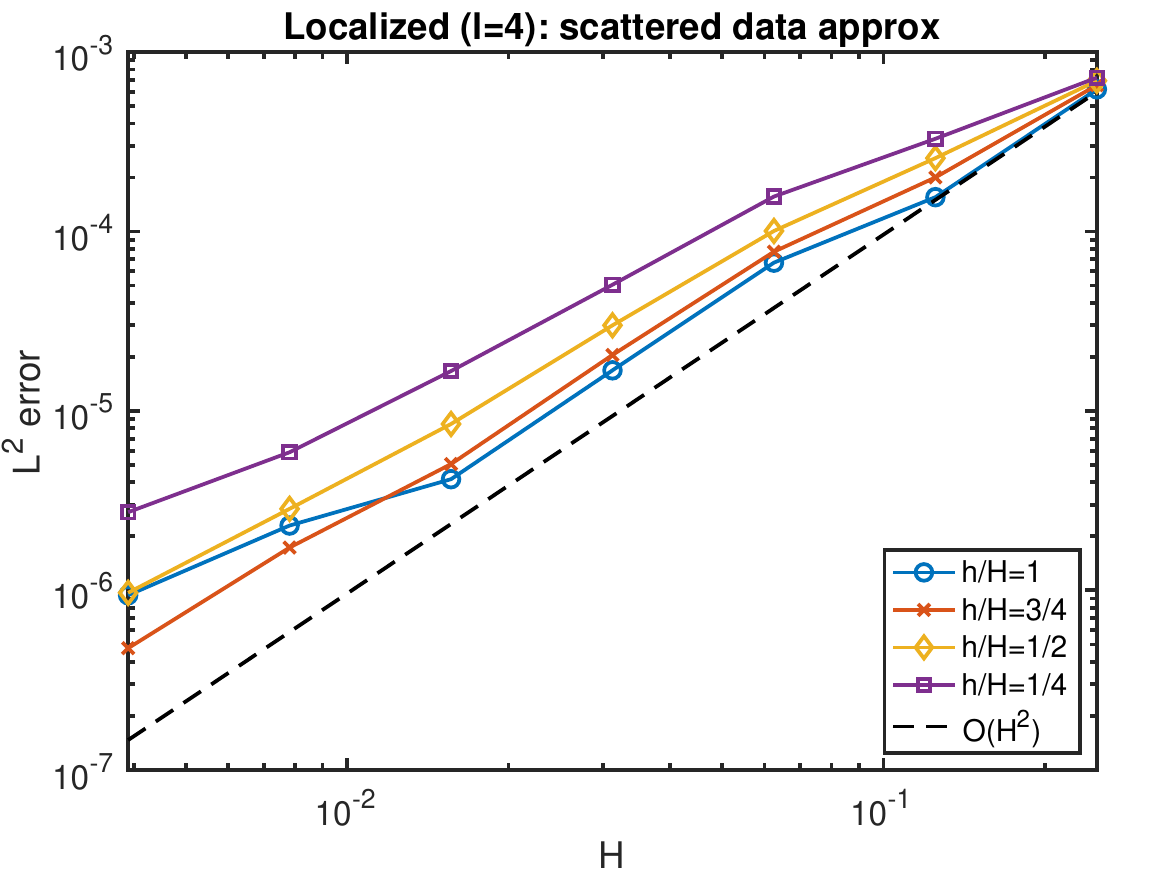}
    \caption{2D example, localized solution $l=4$. Upper left: $\tilde{e}^{h,H,l}_{1}(a,u)$; upper right: $\tilde{e}^{h,H,l}_{0}(a,u)$; lower left: $e^{h,H,l}_{1}(a,u)$; lower right: $e^{h,H,l}_{0}(a,u)$.}
    \label{fig: 2d, localized l=4}
    \end{figure}

\subsection{Analysis: Localized Solution}
\label{subsec: Analysis: Localized Solution}
In this subsection, we provide some theoretical analysis for the localized solution. To begin with, we summarize the main observations in the numerical experiments that we want to understand more deeply in our theoretical study..
\begin{enumerate}
    \item The error lines of the localized solution, $e_1^{h,H,l}(a,u), \tilde{e}_1^{h,H,l}(a,u)$ and also $\tilde{e}_0^{h,H,l}(a,u)$, turn up when $H$ is small, if $l$ is fixed;
    \item The localization error appears to become smaller as $h$ decreases -- for the overall error of the localized solution, there seems to be a competition between the approximation error of the ideal solution (which increases as $h$ decreases), and the localization error (which decreases as $h$ decreases). The strength of the competition depends on the oversampling parameter $l$;
    \item The $L^2$ error of the recovery solution is smaller compared to that of the Galerkin solution, i.e., $\tilde{e}_0^{h,H,l}(a,u)$ appears to be larger than $e_0^{h,H,l}(a,u)$, and for the latter, it does not blows up as $H$ becomes small.
\end{enumerate}
We will provide reasonable theoretical explanation of these observations. First, we introduce several useful notations.
\subsubsection{Notations}
\label{subsec: notation}
For any function $v \in H_0^1(\Omega)$, we write 
\begin{equation}
    \sfP^{h,H,l}v=\sum_{i\in I}~ [v,\phi_i^{h,H}]\psi^{h,H,l}_i\, .
\end{equation}
Moreover, we use the convention $\sfP^{h,H}v=\sum_{i\in I}~ [v,\phi_i^{h,H}]\psi^{h,H}_i$. These definitions lead to the relation $\sfP^{h,H,l}\psi_i^{h,H}=\psi_i^{h,H,l}$, which connects the ideal and localized basis functions.

Since we are mainly interested in how the error depends on $h,H,l$ and $u$, we use $A\lesssim B$ (resp. $A\gtrsim B$) to denote the condition $A \leq CB$ (resp. $A \geq CB$) for some constant $C$ independent of $h,H,l$ and $u$. If we have both $A\lesssim B$ and $A\gtrsim B$, then we will write $A\simeq B$. We use $\left<\cdot,\cdot\right>_a$ to denote the $a$-weighted inner product in $H_0^1(\Omega)$, i.e., $\left<u,v\right>_a:=\int_{\Omega} a \nabla u \cdot \nabla v$.
\subsubsection{Analysis}
To analyze the error of localized solutions, we first use the triangle inequality:
\begin{equation}
\label{eqn: err loc decompose}
\begin{aligned}
    e_1^{h,H,l}(a,u)&=\|u-\sfP^{h,H,l}u\|_{H_a^1(\Omega)}\\
    &\leq \|u-\sfP^{h,H}u\|_{H_a^1(\Omega)} + \|\sfP^{h,H}u-\sfP^{h,H,l}u\|_{H_a^1(\Omega)}\\
    &\lesssim H\rho_{2,d}(\frac{H}{h})\|\cL u\|_{L^2(\Omega)}+\|\sfP^{h,H}u-\sfP^{h,H,l}u\|_{H_a^1(\Omega)}\, ,
\end{aligned}
\end{equation}
where in the last inequality, we have used the estimate for the ideal solution. The second part $\|\sfP^{h,H}u-\sfP^{h,H,l}u\|_{H_a^1(\Omega)}$ is the localization error. Our main goal is to estimate this part of error. For this purpose, we have Theorem \ref{thm: error loc solution} below.
\begin{theorem}
\label{thm: error loc solution}
The following results hold:
\begin{enumerate}
    \item (Inverse estimate) For any $v \in \operatorname{span}~\{\psi^{h,H}_i\}_{i\in I}$ and in each $\omega_j^{h,H}$, $j \in I$, we have the estimate:
		\[\|\nabla \cdot (a \nabla v)\|_{L^2(\omega_j^{h,H})}\leq \frac{\sqrt{a_{\max}}C_2(d)}{h}\|v\|_{H_a^1(\omega_j^{h,H})}\, , \]
		where $C_2(d)$ is a constant that depends on $d$ only.
    \item (Exponential decay) For each $i\in I$ and $k \in \bN$, we have
    \begin{equation}
    \label{eqn: exp decay}
        \|\psi_i^{h,H}\|^2_{H_a^1(\Omega \backslash \rN^k(\omega_i^H))}\leq \left(\beta(h,H)\right)^k\|\psi_i^{h,H}\|_{H_a^1(\Omega)}^2
    \end{equation}
    where 
		\begin{equation}
		\label{def: beta(h,H)}
		    \beta(h,H)=\frac{C_0(d)\sqrt{\frac{a_{\max}}{a_{\min}}}\left(C_1(d)\rho_{2,d}(\frac{H}{h})+C_1(d)C_2(d)\frac{h}{H}\right)}{C_0(d)\sqrt{\frac{a_{\max}}{a_{\min}}}\left(C_1(d)\rho_{2,d}(\frac{H}{h})+C_1(d)C_2(d)\frac{h}{H}\right)+1}\, .
		\end{equation}
		Here, $C_0(d)$ is a universal constant dependent on $d$,
		$C_1(d)$ is the constant in Theorem \ref{thm: err ideal sol} while $C_2(d)$ is the constant in the inverse estimate.
	\item (Norm estimate) Suppose for each $i \in I$, $\phi_i^{h,H}$ is $L^1$ normalized in the sense that $\|\phi_i^{h,H}\|_{L^1(\omega_i^{h,H})}=1$, then the following estimate holds:
	\begin{equation}
	    \|\psi_{i}^{h,H}\|_{H_a^1(\Omega)}\lesssim \frac{1}{\rho_{2,d}(\frac{H}{h})} H^{d/2-1} \, .
	\end{equation}
	\item (Localization error per basis function) For each $i \in I$, it holds that
	\begin{equation}
	\begin{aligned}
	   &\|\psi_i^{h,H}-\psi_i^{h,H,l}\|_{H_a^1(\Omega)}\\ \lesssim &~H^{d/2-1} \min\left\{\left(\beta(h,H)\right)^{l/2}, \frac{1}{\rho_{2,d}(\frac{H}{h})}\right\}\, .
	\end{aligned}
	\end{equation}
	\item (Overall localization error) The following error estimate holds:
	\begin{equation}
	\label{eqn: Overall localization error}
	\begin{aligned}
	    &\|\sfP^{h,H}u-\sfP^{h,H,l}u\|_{H_a^1(\Omega)}\\
	    \lesssim &\min\left\{\left(\beta(h,H)\right)^{l/2}\rho_{2,d}(\frac{H}{h}),1\right\}\times \min\left\{\frac{l^{d/2}}{H},\frac{1}{H^{d/2+1}\rho_{2,d}(\frac{H}{h})}\right\}\|u\|_{L^{\infty}(\Omega)}\, .
	\end{aligned}
	\end{equation}
	\item (Overall recovery error) Suppose $d\leq 3$. For the energy recovery error, we have
	\begin{equation}
	\label{eqn: recovery e1}
	\begin{aligned}
	    e_1^{h,H,l}(a,u) \lesssim \bigg(H\rho_{2,d}(\frac{H}{h})
	    +&\min\left\{\left(\beta(h,H)\right)^{l/2}\rho_{2,d}(\frac{H}{h}),1\right\}\\&\times \min\left\{\frac{l^{d/2}}{H},\frac{1}{H^{d/2+1}\rho_{2,d}(\frac{H}{h})}\right\}\bigg)\|\cL u\|_{L^2(\Omega)}\, .
	\end{aligned}
	\end{equation}
	and for the $L^2$ recovery error, we have
	\begin{equation}
	\label{eqn: recovery e0}
	\begin{aligned}
	    e_0^{h,H,l}(a,u) \lesssim \bigg((H\rho_{2,d}(\frac{H}{h}))^2
	    +&\min\left\{1,H\rho_{2,d}(\frac{H}{h})\right\}\\&\times \min\left\{\left(\beta(h,H)\right)^{l/2}\rho_{2,d}(\frac{H}{h}),1\right\}\\
	    &\times \min\left\{\frac{l^{d/2}}{H},\frac{1}{H^{d/2+1}\rho_{2,d}(\frac{H}{h})}\right\}\bigg)\|\cL u\|_{L^2(\Omega)}\, .
	\end{aligned}
	\end{equation}
\item (Overall Galerkin error)
Suppose $d\leq 3$. The energy Galerkin error is upper bounded by the energy recovery error: $\tilde{e}_1^{h,H,l}(a,u)\leq e_1^{h,H,l}(a,u)$. For the $L^2$ Galerkin error, we have
	\begin{equation}
	\label{eqn: garlekin e0}
	\begin{aligned}
	    \tilde{e}_0^{h,H,l}(a,u) \lesssim \bigg(H\rho_{2,d}(\frac{H}{h})
	    +&\min\left\{\left(\beta(h,H)\right)^{l/2}\rho_{2,d}(\frac{H}{h}),1\right\}\\&\times \min\left\{\frac{l^{d/2}}{H},\frac{1}{H^{d/2+1}\rho_{2,d}(\frac{H}{h})}\right\}\bigg)^2\|\cL u\|_{L^2(\Omega)}\, .
	\end{aligned}
	\end{equation}
\end{enumerate}
\end{theorem}
\subsubsection{Implications} Before we move to the proof part, let us first discuss the implications of this theorem. We focus on the localization error in the final estimates.
\begin{itemize}
    \item Fix an $l$ and the ratio $H/h$. Due to \eqref{eqn: recovery e1} and \eqref{eqn: garlekin e0}, the localization error parts in $e_1^{h,H,l}(a,u)$, $\tilde{e}_1^{h,H,l}(a,u)$ and $\tilde{e}_0^{h,H,l}(a,u)$ will blow up as $H$ goes to $0$. In contrast, due to \eqref{eqn: recovery e0}, the localization error in $e_0^{h,H,l}(a,u)$ remains bounded in this limit. Indeed, it is bounded by
    \begin{align*}
        &H\rho_{2,d}(\frac{H}{h})\times \left(\beta(h,H)\right)^{l/2}\rho_{2,d}(\frac{H}{h})\times \frac{l^{d/2}}{H}\|\cL u\|_{L^2(\Omega)}\\
        \leq &~l^{d/2}\left(\beta(h,H)\right)^{l/2}\left(\rho_{2,d}(\frac{H}{h})\right)^2 \|\cL u\|_{L^2(\Omega)}\, ,
    \end{align*}
    which does not blow up as $H \to 0$. This reveals a distinguished behavior of $e_0^{h,H,l}(a,u)$ compared to the other three errors, which have been observed in our experiments. Our analysis explains this phenomenon. 
    \item For $e_1^{h,H,l}(a,u)$, our analysis shows that there is a competition between the approximation error of the ideal solution, $H\rho_{2,d}(\frac{H}{h})$ (we omit $\|\cL u\|_{L^2(\Omega)}$ for simplicity), and the localization error
    \[\min\left\{\left(\beta(h,H)\right)^{l/2}\rho_{2,d}(\frac{H}{h}),1\right\}\times \min\left\{\frac{l^{d/2}}{H},\frac{1}{H^{d/2+1}\rho_{2,d}(\frac{H}{h})}\right\}\, .\]
    Fix an $H$ and $l$. When $d\geq 2$, since $\lim_{h\to 0} \rho_{2,d}(\frac{H}{h})=\infty$, we have that as $h \to 0$, the approximation error goes to infinity, while the localization error goes to zero. When $d=1$, both two parts of errors remain bounded as $h \to 0$, and thus the competition is less pronounced; this matches what we have observed in our 1D experiments -- the effect of reducing $h$ is not as large as in our 2D example. 
    
    The existence of competition implies that in general, there should be a value of $h$ that leads to the best error for the fixed $H$ and $l$. Because the localization error decreases as $l$ increases, this optimal value would also increase for a larger $l$, as observed in our experiments.
    
    The above phenomenon also applies to other errors, i.e., the recover $L^2$ error $e_0^{h,H,l}(a,u)$ and the Galerkin errors $\tilde{e}_1^{h,H,l}(a,u)$ and $\tilde{e}_0^{h,H,l}(a,u)$.
    \item If we fix $H/h$, and want to have an overall error of $O(H)$ (for energy error) or $O(H^2)$ (for the $L^2$ error), then our estimates show that 
    \[l=O(\frac{\log H}{\log \beta(h,H)}) \]
    suffices for this goal. Note that $\beta(h,H)$ can be treated as a constant (less than $1$) when $H/h$ is fixed, so generally $l=O(\log (1/H))$ is enough. Moreover, our experiments demonstrate that we could do much better in practice -- a constant value of $l=2$ or $4$ behaves well for a wide range of $ H $ and $h$.
\end{itemize}
The three points above explain the questions that we raised at the beginning of Subsection \ref{subsec: Analysis: Localized Solution}. 
\begin{remark}
Though the presence of `$\min$' in many places of our estimates complicates the formula, they play critical roles in the above explanations, since we need to choose the correct term inside the `$\min$' to get the desired conclusion.
\end{remark}

\begin{remark}
    In Theorem \ref{thm: error loc solution}, the basis function $\psi_i^{h,H}$ has an exponential decay property; see \eqref{eqn: exp decay}. The localization error should heavily depend on the decay rate, so obtaining a tight bound of this rate is important here. 
In our analysis, we get the rate $\beta(h,H)$, which contains a term $\rho_{2,d}(H/h)$ that increases as $h$ decreases (when $d\geq 2$), and a term $h/H$ that decreases while $h$ decreases. The two mixed components may suggest a non-monotone behavior of the decay rate. Moreover, when $h \to 0$, we get $\beta(h,H) \to 1$, so the decay appears to deteriorate eventually for small $h$. On the other hand, it seems intuitive that once $h$ is small, the measurement region $\omega_i^{h,H}$ becomes more localized, and then the decay shall be amplified. To understand this problem better, we conduct a numerical experiment as follows. For the coefficient $a(x)$ in \eqref{eqn: a(x)} and $H=2^{-5}$, we compute the relative localization error $\frac{\|\psi_i^{h,H}-\psi_i^{h,H,l}\|^2_{H_a^1(\Omega)}}{\|\psi_i^{h,H}\|^2_{H_a^1(\Omega)}}$ for $h = 2^{-5},2^{-6},...,2^{-10}$ and $l=0,1,2,...,5$. The index $i$ is selected so that $\omega_i^H$ is centered in the domain $\Omega$. The result is shown in Fig. \ref{fig: 2d, decay rate}.
\begin{figure}[!htb]
    \centering
    \includegraphics[width=10cm]{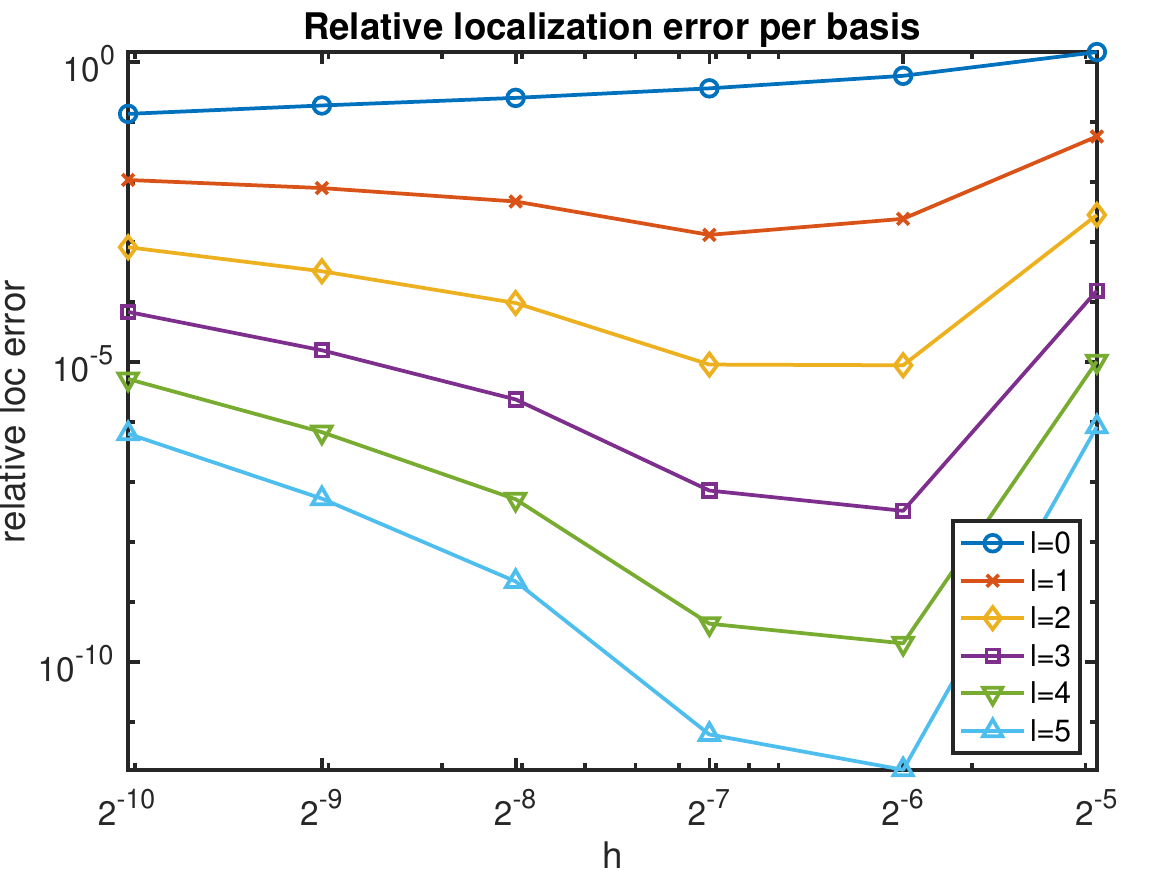}
    \caption{Relative localization error per basis function}
    \label{fig: 2d, decay rate}
    \end{figure} 
\end{remark}
From the figure, we observe that there is indeed a non-monotone behavior with respect to $h$ in the relative localization error. Among these choices of $h$ and $l$, we only see a monotone tendency for $l=0$. For other $l$, the value $h$ that leads to the minimal relative localization error increases as $l$ increases. For the $a(x)$ and $H$ considered, we can see $h/H = 1/2, 1/4$ lead to small errors in general, which also explains that this choice of $h$ works quite well in our previous experiments.  Overall, the above investigation suggests that our bound on the exponential decay and localization error can reasonably predict the behavior in practice. The decay is truly subtle regarding the small parameter $h$.

\begin{remark}
Our current result does not provide explicit clues on how to choose $h$ according to $l$ and $ H $ to achieve the best accuracy. Nonetheless, our experiments have shown that usually $h/H=3/4$ or $1/2$ behaves well, across a wide range of $H=2^{-8},2^{-7},...,2^{-2}$ and $l=2,4$, in the two dimensional problems. Providing more guidance on this aspect, either numerically or theoretically, is left as future work.
\end{remark}
\subsubsection{Proof Strategy}
The results in Theorem \ref{thm: error loc solution} are presented progressively. Our proofs will start from the first and move forward one by one to the seventh. We summarize the main ideas below, together with their connections to existing results in the literature. The detailed proof is in Subsection \ref{subsec Proof of Theorem thm: error loc solution}.
\begin{enumerate}
    \item The inverse estimate is obtained due to a scaling argument -- that is why there is the subsampled scale $h$ appeared. (Subsection \ref{subsec: inv estimate})
    \item Based on the inverse estimate and the subsampled Poincar\'e inequality (see Proposition 2.5 in \cite{chen2019function}), we can establish the exponential decay property via a Caccioppoli type of argument. The logical line of our proof here is similar to that of the original LOD method (Lemma 3.4 in \cite{malqvist_localization_2014}) and Gamblets (Theorem 3.9 in \cite{owhadi_multigrid_2017}), while now we need to be careful to make every estimate adaptive to the small scale parameter $h$. (Subsection \ref{subsec: exp decay})
    \item For the norm estimate, we construct critical examples whose energy norm leads to a desired upper bound. The critical example here is similar to the one we used before to prove the optimality of the subsampled Poincar\'e inequality (see Proposition 2.6 in \cite{chen2019function}). This type of profile has also been studied in the context of semi-supervised learning; see Theorem 2 in \cite{Semi-supervised}. (Subsection \ref{subsec: norm estimate})
    \item The localization error per basis function is established by combining the exponential decay estimate and the norm estimate. Our results contain two parts inside the `min' operation. The idea of proving the first part is similar to that of Lemma 3.4 in \cite{malqvist_localization_2014}. The second part is a direct application of the norm estimate. Both parts are important. The first part captures the exponential decay property, while the second part captures the behavior with respect to small $h$ -- when $d\geq 2$, this estimate implies the localization error per basis function vanishes as $h$ goes to $0$. (Subsection \ref{subsec: Localization Per Basis Function})
    \item To move from the localization error per basis function to the overall localization error, we also proceed in two directions. The first one follows the idea of proving Lemma 3.5 in \cite{malqvist_localization_2014}, leading to an upper bound of $O(l^{d/2}/H)$, which remains bounded as $h \to 0$. On the other hand, we can use simple triangle inequality, which yields an estimate of $O\left(1/\left(H^{d/2+1}\rho_{2,d}(\frac{H}{h})\right)\right)$, which is worse in the power of $H$ than the first one, but can capture the limit as $h\to 0$, i.e., it vanishes as $h\to 0$. The combination of the two leads to the final estimate. (Subsection \ref{subsec: Overall Localization Error})
    \item It is straightforward to go from overall localization error to the energy recovery error by a triangle inequality. For the $L^2$ recovery error, we can bound it through the energy error in two ways, with or without using the subsampled Poincar\'e inequality. This leads to a further `min' operation in the final estimate. (Subsection \ref{subsec: Overall Recovery Error})
    \item The energy Galerkin error is upper bounded by the energy recover error according to the Galerkin orthogonality. The $L^2$ Galerkin error is obtained by the standard Aubin-Nitsche trick. (Subsection \ref{subsec: overall Galerkin error})
\end{enumerate}

\section{Small Limit Regime of Subsampled Lengthscales} 
\label{subsec Small Limit Regime of Subsampled Lengthscales}
In the last section, we have made a detailed study of the recovery error and Galerkin error with respect to $h,H$, and $l$. We observe that there is a deterioration of accuracy as $h$ becomes small, especially for $d\geq 2$ -- the benefit of small localization errors by a very small $h$ is overwhelmed by the curse of induced large approximation errors. Due to this reason, in our experiments, we choose the ratio $h/H$ to be not too small -- we select $h/H\geq 1/8$ in 1D and $h/H\geq 1/4$ in 2D. Our theoretical analysis also collaborates with these observations, as the function $\rho_{2,d}(H/h)$ that appears in the error estimate will blow up as $h/H \to 0$ for $d\geq 2$.

Therefore, we are advised not to use a very small $h$. While this is a practical suggestion in the problem of numerical upscaling, since we have the freedom of choosing the upscaled variables and thus can avoid this pathological phenomenon, in the problem of scattered data approximation, we may not have such flexibility due to the prevalent physical constraints for data measurements. As we often encounter recovery problems in high dimensions with scattered data that possibly have a very small lengthscale, e.g., pointwise data, it is natural to ask that whether we could get an accurate recovery even in the $h\to 0$ regime. The analysis above implies that this goal is not achievable in general for the model problem we have considered. Thus, we need to put stronger assumptions on the function $u$ to be approximated.

Since the degeneracy of accuracy for $d\geq 2$ can be partially attributed to the low regularity of the target function $u$, that is, when $d\geq 2$, functions in $H^1(\Omega)$ may not have a well-defined pointwise value (according to the Sobolev embedding theorem \cite{evans_partial_2010}), a natural idea is to assume $u$ to be more regular. There has been some work in which $u$ is assumed to be in $W^{k,2}(\Omega)$ for some larger $k$ \cite{high_order_regularization}; this assumption ensures the continuity of the function. Alternatively, one can assume $u \in W^{1,p}(\Omega)$ and increase $p$ -- when $p > d$, the degeneracy issue disappears; see \cite{p_Laplacian_Jordan, p_Laplacian_analysis, Lipschitz_learning, Lipschitz_learning_analysis}. 

The above assumptions of better regularity on $u$, either via increasing $k$ or $p$, require to modify the recovery algorithm substantially -- in the former, the basis functions are obtained by replacing the $H_a^1(\Omega)$ norm in \eqref{eqn: optimization def basis} by a high order norm, similar to the polyharmonic splines and their rough version \cite{owhadi2014polyharmonic}; in the latter, the recovery function is obtained by minimizing the $W^{1,p}(\Omega)$ norm subject to the observed data.

Here, to stick to the formulation \eqref{eqn: optimization def basis} and thus the main theme of this paper, we consider to improve the regularity via choosing a singular weight function $a(x)$. Naturally, in order to make the recovery non-degenerate regarding a vanishing $h$, we need to put more importance on the coarse data of a small lengthscale $h$. Thus, we could assume the function is ``nearly flat'' around the data location by using a singular $a(x)$ such that $\int_\Omega a|\nabla u|^2 < \infty$ -- this guarantees the information content of coarse data even for very small $h$. 
We will make this intuition more quantitative in this section.
\subsection{Numerical Experiment} As before, we start with some numerical experiment. We choose $d=2$ and $\Omega=[0,1]^2$. The ground truth function $u$ is depicted in the upper-left of Figure \ref{fig: small h figures}. The coarse scale $H=2^{-2}$, and suppose for now we collect subsampled data with lengthscale $h=H/2=2^{-3}$; the grid size $h_g$ is set to be $2^{-7}$. In the upper-right of Figure \ref{fig: small h figures}, we plot the ideal recovery solution by using $a(x)=1$, the subsampled data $[u,\phi_i^{h,H}], i\in I$ and the ideal basis functions $\{\psi_i^{h,H}\}_{i \in I}$. We observe that to certain extent, the recovery solution can capture the large scale property of $u$.

Then, we decrease the subsampled lengthscale -- we choose $h=2^{-4}\cdot H=2^{-6}$. The recovery solution obtained by solving \eqref{eqn: optimization def basis} with $a(x)=1$ is in the lower-left of Figure \ref{fig: small h figures}. The degeneracy issue becomes apparent -- there are many spikes in the recovery solution, and the locations of these spikes are the data positions. This confirms our understanding that a small $h$ leads to a degenerate recovery.

Now, we define a weight function as follows. For each local patch $\omega_i^{H}, i \in I$, its center is denoted by $x_i \in \omega_i^{H}$. We write $X^H=\bigcup_{i=1}^I \{x_i^H \}$ and $\sfd(x,X^H)$ is the Euclidean distance from $x$ to the set $X^H$. The weight function is defined as
\begin{equation}
\label{eqn: W(x) experiments}
    W(x)=\left(\frac{H}{\sfd(x,X^H) }\right)\log^2\left(1+\frac{H}{\sfd(x,X^H)}\right)\, .
\end{equation}
It is singular at the center of our subsampled data; see Figure \ref{fig: small h W(x) contour}. In the lower-right of Figure \ref{fig: small h figures}, we we construct the recovery solution by solving \eqref{eqn: optimization def basis} with $a(x)=W(x)$. To avoid numerical instability in the experiment, we we use a regularized version of the singular weight as follows:
\begin{equation}
    W(x;h_g)=\left(\frac{H}{\max\{h_g,\sfd(x,X^H)\} }\right)\log^2\left(1+\frac{H}{\max\{h_g,\sfd(x,X^H)\}}\right)\, ,
\end{equation}
where $h_g$ is the grid size. From the figure, we observe that the recovery solution appears much better than the one based on $a(x)=1$. It captures most of the large scale behaviors. Moreover, it is visually smoother -- due to the singular weight function, the impact of the subsampled data does propagate to other points in the domain.

\begin{figure}[!htb]
    \centering
    \includegraphics[width=6cm]{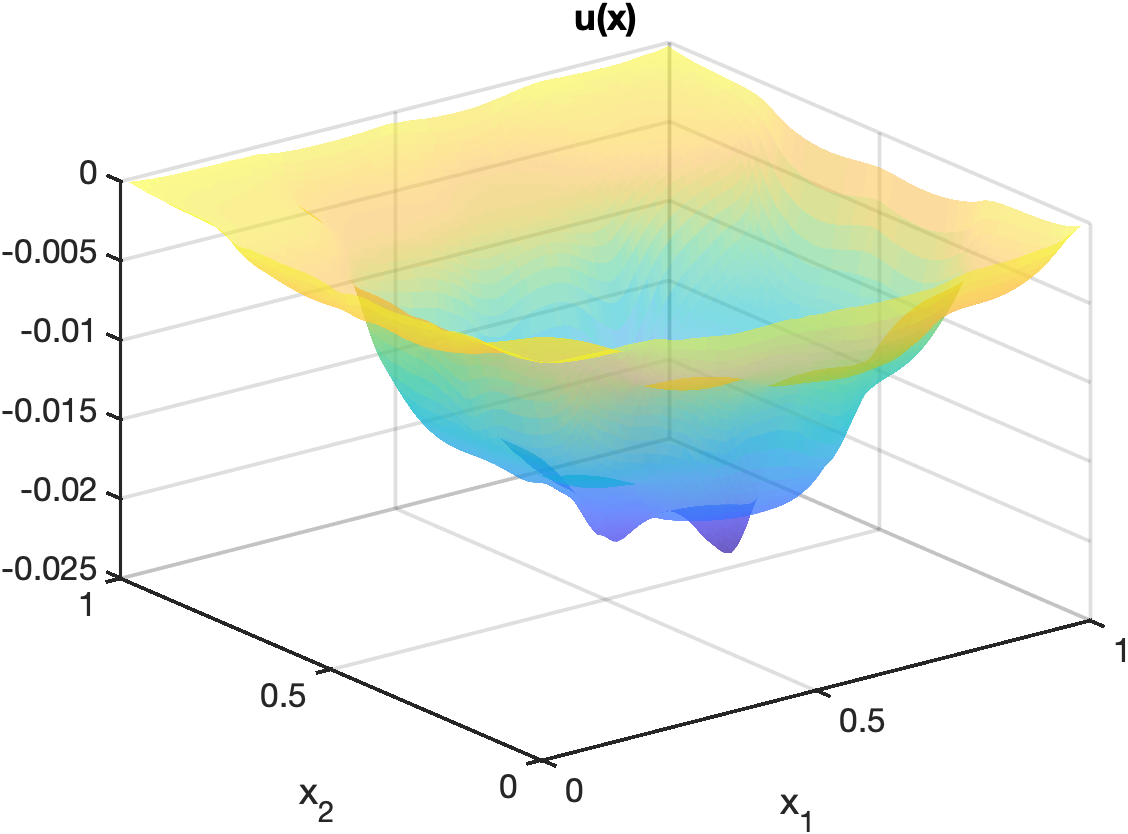}
    \includegraphics[width=6cm]{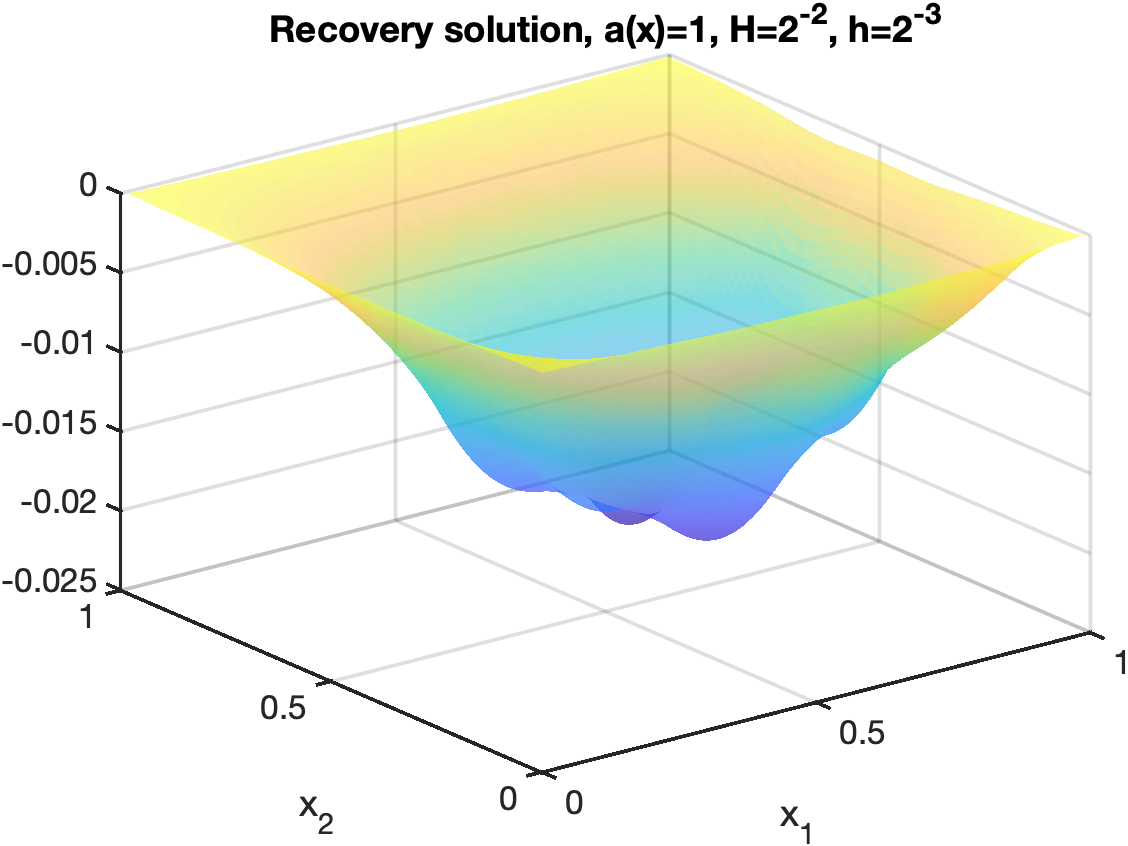}
    \includegraphics[width=6cm]{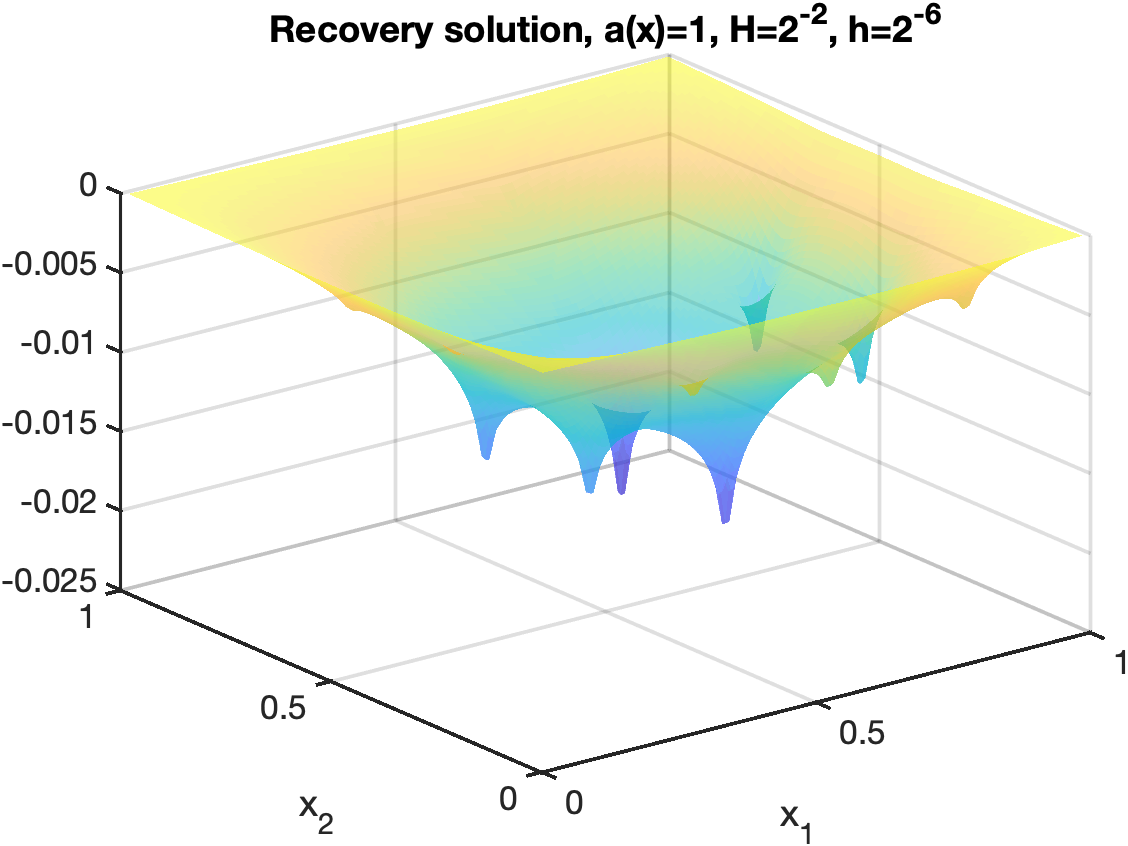}
    \includegraphics[width=6cm]{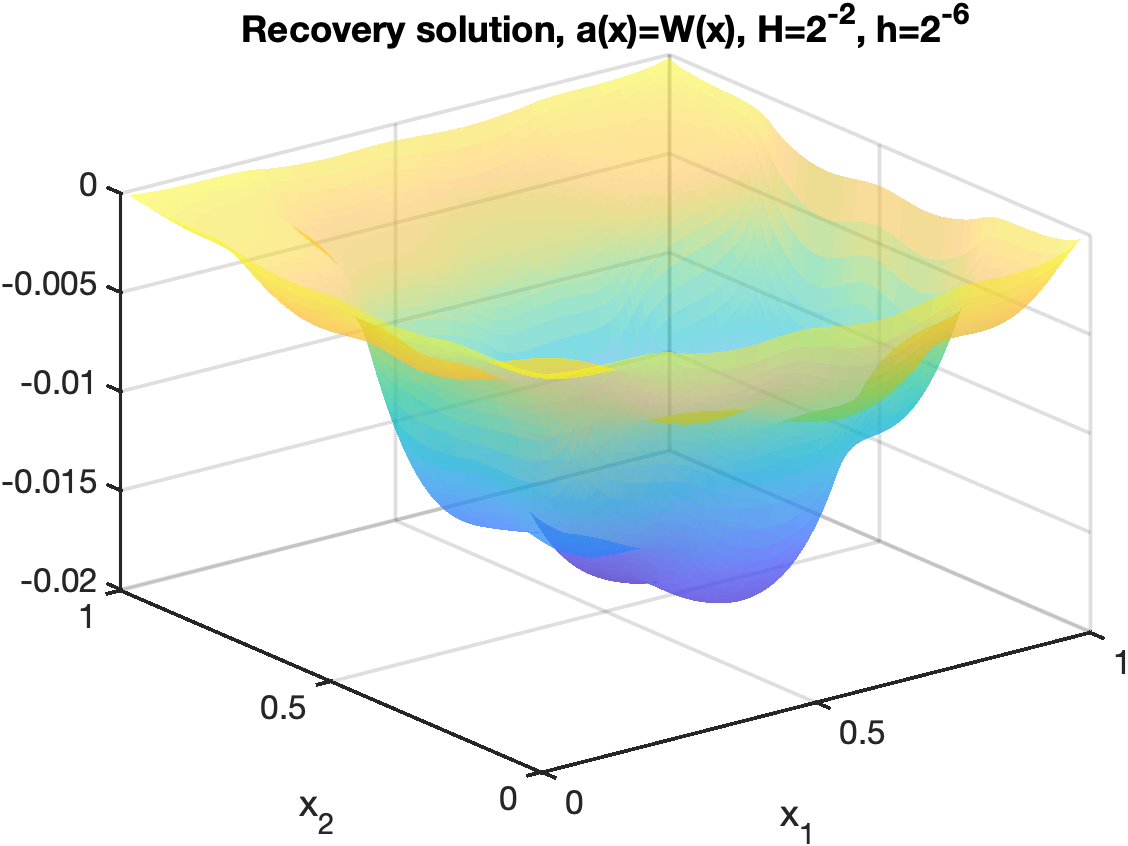}
    \caption{Upper left: $u(x)$; upper right: recovery solution, $h/H=1/2$ and $a(x)=1$; lower left: recovery solution, $h/H=1/2^4$ and $a(x)=1$; lower right: recovery solution, $h/H=1/2^4$ and $a(x)=W(x)$.}
    \label{fig: small h figures}
\end{figure} 

\begin{figure}[!htb]
    \centering
    \includegraphics[width=6cm]{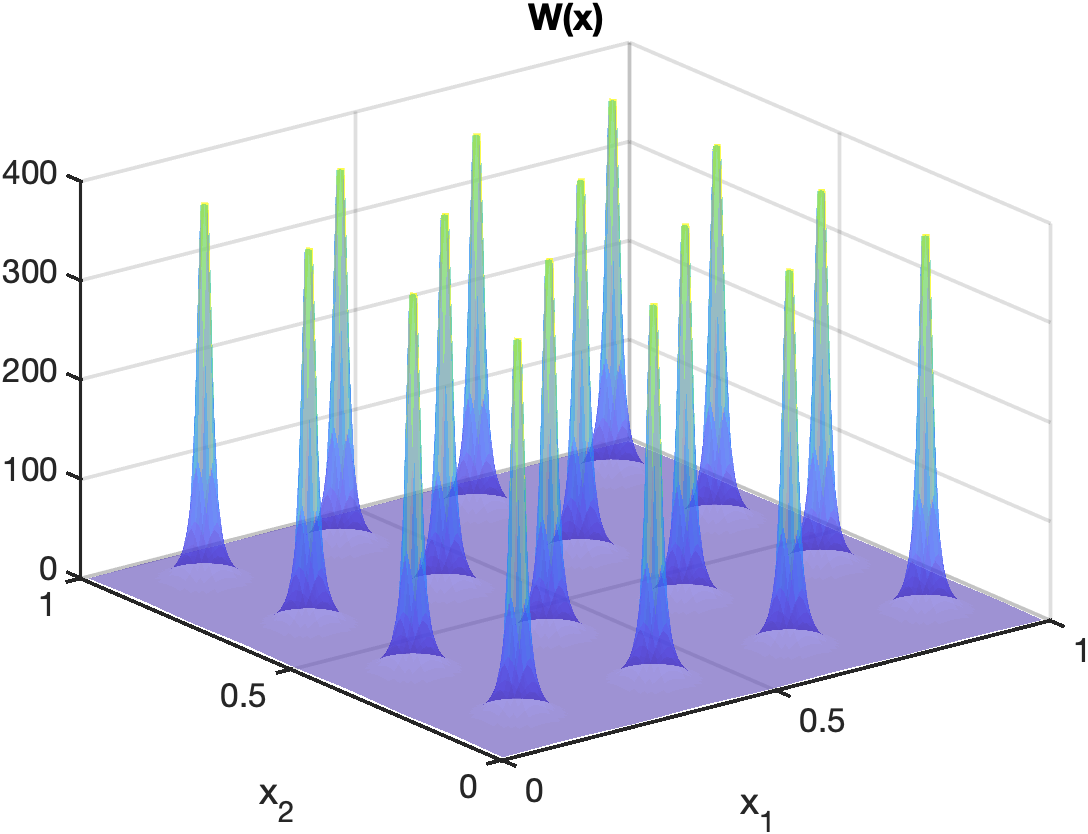}
    \includegraphics[width=6cm]{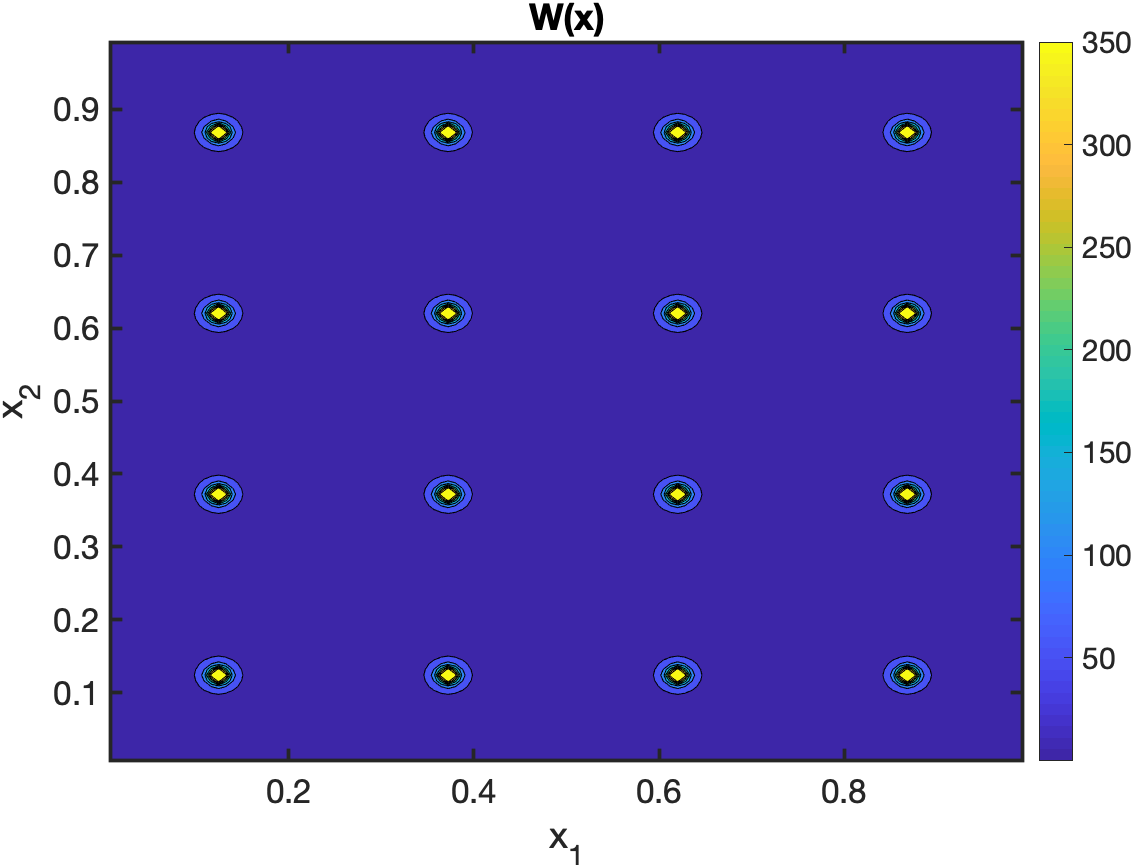}
    \caption{Left: figure of $W(x)$; right: contour of $W(x)$}
    \label{fig: small h W(x) contour}
\end{figure} 
\begin{remark}
The idea of function recovery based on a weight function that puts more importance around the data regions has been used in semisupervised learning and image processing \cite{shi2017weighted}, through using a weighted graph Laplacian. Recently, the work \cite{calder2019properly} proposed a properly weighted Laplacian that attains a well-defined continuous limit. Our earlier work \cite{chen2019function} also discussed a similar weighted discovery. In the next subsection, we will provide some theoretical analysis of this recovery based on results in \cite{chen2019function}, assuming $u(x)$ belonging to a weighted function space.
\end{remark}

\subsection{Analysis: Weighted Inequality} 
For simplicity, in dimension $d\geq 2$, we consider the following class of weight functions:
\begin{equation}
    W_{\gamma,H}(x)=\left(\frac{H}{\sfd(x,X^H) }\right)^{d-2+\gamma}\, ,
\end{equation}
where $\gamma > 0$. Indeed, the additional $\log$ term in \eqref{eqn: W(x) experiments} only makes the problem easier, since it makes the function blow up even faster. 

We use the same notation as in Subsection \ref{subsec: notation}. Then, we have the following theorem:
\begin{theorem}
\label{thm: weighted estimates}
Let $d\geq 2$ and $\gamma>0$. Fix an $H$, and we choose $a(x)=W_{\gamma,H}(x)$. Then the following results hold:
\begin{enumerate}
    \item If $\|u\|_{H_a^1(\Omega)}<\infty$, then the $L^2$ error of the ideal solution satisfies
    \begin{equation}
        e_0^{h,H,\infty}(a,u)\lesssim C(\gamma)H\|u\|_{H_a^1(\Omega)}\, .
    \end{equation}
    \item If $-\nabla \cdot (a\nabla u)=f\in L^2(\Omega)$, then the energy error of the ideal solution satisfies
    \begin{equation}
        e_1^{h,H,\infty}(a,u)\lesssim C(\gamma)H\|f\|_{L^2(\Omega)}\, ;
    \end{equation}
    and the $L^2$ error satisfies
    \begin{equation}
        e_0^{h,H,\infty}(a,u)\lesssim C(\gamma)H^2\|f\|_{L^2(\Omega)}\, .
    \end{equation}
\end{enumerate}
Here, $C(\gamma)$ represents a positive constant that depends on $\gamma$ only, and can vary its value from place to place.
\end{theorem}
The proof is deferred to Subsection \ref{subsec: Proof of Theorem weighted estimates}. We observe from the theorem that, the upper bound of the accuracy is independent of the subsampled scale $h$, which implies that it is still valid in the small $h$ limit. This is in sharp contrast with the estimates in Theorem \ref{thm: err ideal sol}, where the upper bound blows up as $h\to 0$. The key here is the use of a singular weight function that puts more importance on the subsampled data.

We also use a numerical experiment to demonstrate this theorem. We choose $d=2$, $\Omega=[0,1]^2$ and $H=2^{-2}$. The parameter $\gamma=1$. We use the mechanism in Subsection \ref{subsec: ideal 2d experiments} to generate a right-hand side $f\in L^2(\Omega)$, and $u$ solves \[-\nabla \cdot (W_{\gamma,H}\nabla u)=f\, .\]
The grid size is set to be $2^{-8}$. We choose $h=2^{-3},2^{-4},...,2^{-7}$. For each $h$, we collect the data $[u,\phi_i^{h,H}],i\in I$ and compute the ideal recovery solutions by solving \eqref{eqn: optimization def basis} with $a(x)=1$ and $a(x)=W_{\gamma,H}(x)$ respectively. We output the $H_0^1(\Omega)$ and $L^2(\Omega)$ error of these recovery solutions in Figure \ref{fig:weighted recovery errors}.
\begin{figure}[!htb]
    \centering
    \includegraphics[width=6cm]{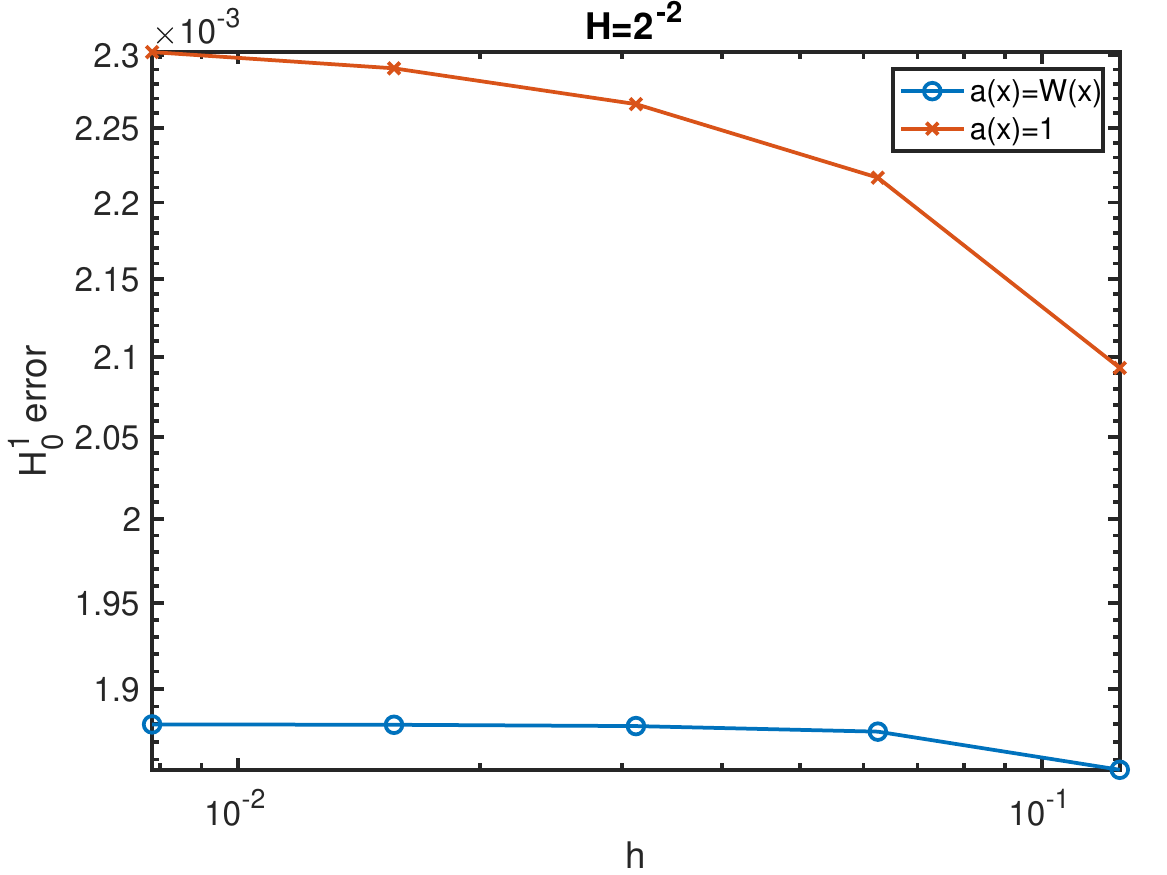}
    \includegraphics[width=6cm]{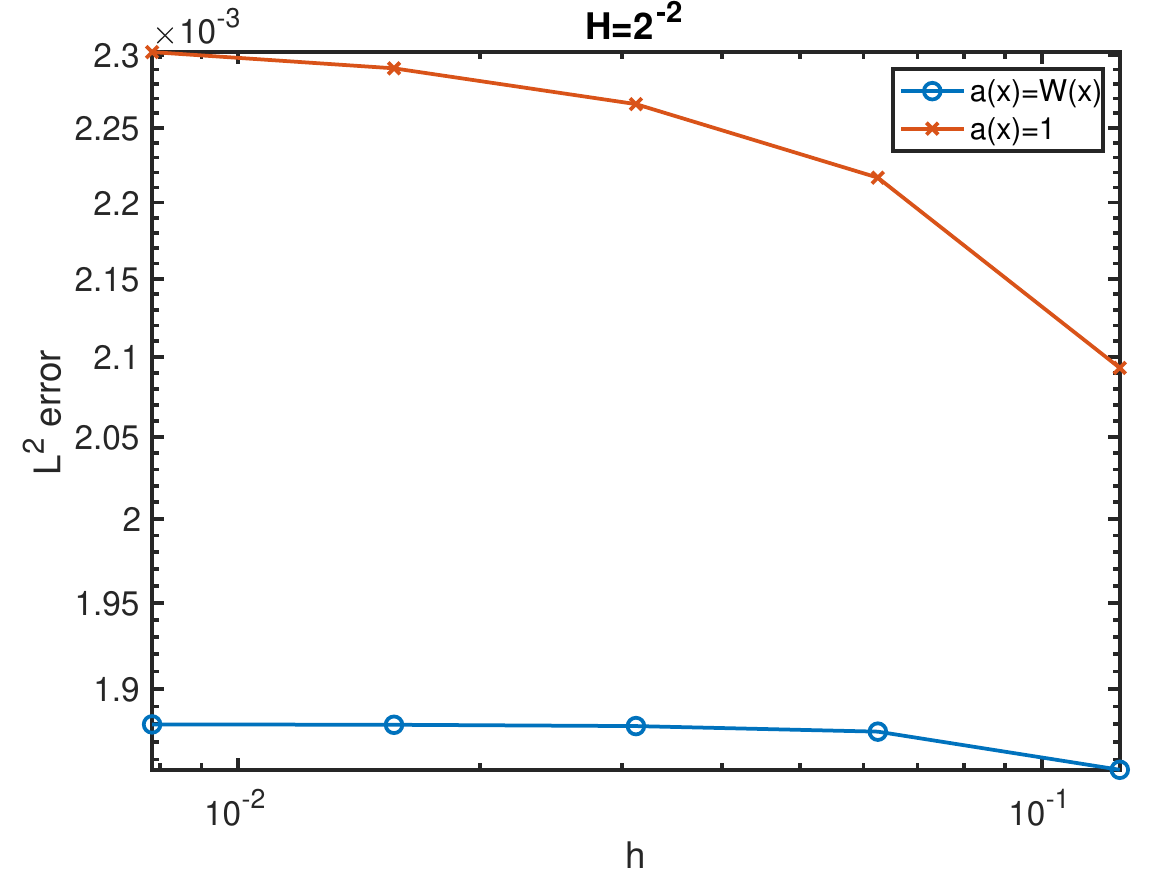}
    \caption{The $H_0^1(\Omega)$ and $L^2(\Omega)$ errors for different $h$, using constant $a(x)$ or singular weighted $a(x)$. Left: $H_0^1(\Omega)$ error; right: $L^2(\Omega)$ error}
    \label{fig:weighted recovery errors}
\end{figure} 
From this figure, we observe that the recovery errors using $a(x)=1$ will increase as $h$ decrease, while those using $a(x)=W_{\gamma,H}(x)$ lead to a flattened curve with respect to $h$. This matches our theoretical predictions. Since in this example the dimension $d=2$, the blow-up rate predicted by Theorem \ref{thm: err ideal sol} is only logarithmic, so even though $h$ is very small, the overall accuracy is still not too bad.

\section{Proofs} This section provides all the proofs in this paper.
\label{sec:proof}
\subsection{Proof of Theorem \ref{thm: error loc solution}} 
\label{subsec Proof of Theorem thm: error loc solution}
There are seven sub-results in this theorem. We prove them one by one.
\subsubsection{Inverse Estimate}
\label{subsec: inv estimate}
	In the domain $\omega_j^{h,H}$, we have $\nabla \cdot (a \nabla v)=c_i\phi_i^{h,H}$ for some $c_i \in \bR$. Let $v=v_1+v_2$ such that
	\[\nabla \cdot (a \nabla v_1)= \nabla \cdot (a \nabla v)=c_i\phi_i^{h,H}\ \text{in}\  \omega_j^{h,H}, \quad v_1|_{\partial \omega_j^{h,H}}=0\, ,\]
	and for the second part,
	\[\nabla \cdot (a \nabla v_2)= 0\ \text{in}\  \omega_j^{h,H},\quad v_2|_{\partial \omega_j^{h,H}}=v|_{\partial \omega_j^{h,H}}\, .\]
	We have the orthogonality: $\int_{\omega_j^{h,H}}a\nabla v_1 \cdot \nabla v_2 =0$. Thus, it holds that
	\begin{equation}
	\label{eqn:h-b decompose}
	    \|v\|_{H_a^1(\omega_j^{h,H})}\geq \|v_1\|_{H_a^1(\omega_j^{h,H})}\, .
	\end{equation}
	For $v_1$, we use the elliptic estimate: 
	\[\|v_1\|_{H_a^1(\omega_j^{h,H})}\geq \frac{1}{\sqrt{a_{\max}}}\|\nabla \cdot (a \nabla v_1)\|_{H^{-1}(\omega_j^{h,H})}=\frac{1}{\sqrt{a_{\max}}}\|c_j\phi_j^{h,H}\|_{H^{-1}(\omega_j^{h,H})}\, .\]
	By a scaling argument, we obtain
	\[\|\phi_j^{h,H} \|_{L^2(\omega_j^{h,H})}\leq \frac{C_2(d)}{h}\|\phi_j^{h,H}\|_{H^{-1}(\omega_j^{h,H})}\, , \]
	for a constant $C_2(d)$ dependent on $d$. Then, it follows that
	\begin{equation}
	\label{eqn: inv estimate 2}
	    \|v_1\|_{H_a^1(\omega_j^{h,H})}\geq \frac{h}{\sqrt{a_{\max}}C_2(d)}\|c_j\phi_j^{h,H} \|_{L^2(\omega_j^{h,H})}=\frac{h}{\sqrt{a_{\max}}C_2(d)}\|\nabla \cdot (a \nabla v)\|_{L^2(\omega_j^{h,H})}\, .
	\end{equation}
	 Combining \eqref{eqn:h-b decompose} and \eqref{eqn: inv estimate 2}, we arrive at the desired result:
\[\|\nabla \cdot (a \nabla v)\|_{L^2(\omega_j^{h,H})}\leq \frac{\sqrt{a_{\max}}C_2(d)}{h}\|v\|_{H_a^1(\omega_j^{h,H})}\, . \]
\subsubsection{Exponential Decay}
\label{subsec: exp decay}
Fix $i \in I$. For ease of notations, we will write $\psi^{h,H}_i$ by $\psi$, and $\rN^{k}(\omega_i^H)$ by $S_k$ in this proof. 		
	
	First, we choose a cut-off function $\eta $ with value $0$ in $S_k$ and value $1$ in $S_{k+1}^c$ such that it satisfies $\eta \geq 0$ and  $\|\nabla \eta\|_{\infty}\leq C_0(d)/H$ for some universal constant $C_0(d)$ dependent on $d$. An example of $\eta$ could be 
	\[\eta(x)=\frac{\text{dist}(x,S_k)}{\text{dist}(x,S_k)+\text{dist}(x,S_{k+1}^c)}\, . \]
	Then, we obtain the relation:
	\begin{equation}
	\label{eqn: exp decay 0}
	    \|\psi\|^2_{H_a^1(\Omega \backslash S_{k+1})}=\int_{\Omega \backslash S_{k+1}} \nabla \psi \cdot a \nabla \psi \leq \int_{\Omega} \eta \nabla \psi \cdot a \nabla \psi \, .
	\end{equation}
	Using some algebra, we have 
	\begin{align*}
	    \eta\nabla \psi \cdot a \nabla \psi&=\nabla (\eta \psi)\cdot a\nabla \psi-(\nabla \eta)\cdot a\psi \nabla \psi\\
	    &=\nabla\cdot(\eta \psi a \nabla \psi)-\eta \psi \nabla \cdot (a\nabla \psi)-(\nabla \eta)\cdot a\psi \nabla \psi\, .
	\end{align*}
	 Integrating the above formula in $\Omega$ and applying the divergence theorem yields
	\begin{equation}
	    \label{eqn: exp decay 1}
	    \int_{\Omega} \eta \nabla \psi \cdot a \nabla \psi
	\leq \left|\int_{\Omega} -\eta \psi\nabla \cdot (a \nabla \psi)\right| + \left|\int_{\Omega} (\nabla \eta) \cdot\psi a \nabla \psi\right| \, .
	\end{equation}
	For the first term in \eqref{eqn: exp decay 1}, we have 
	\begin{equation}
	\label{eqn: exp decay 2}
	 \begin{aligned}
	\int_{\Omega} -\eta \psi\nabla \cdot (a \nabla \psi)&\overset{(a)}{=}\sum_{\omega_j^H \subset S_{k+1} \backslash S_k} \int_{\omega_j^H} -\eta \psi\nabla \cdot (a \nabla \psi)\\
	&\overset{(b)}{=}\sum_{\omega_j^H \subset S_{k+1} \backslash S_k} \int_{\omega_j^{h,H}} -\eta \psi\nabla \cdot (a \nabla \psi)\\
	&\overset{(c)}{=}\sum_{\omega_j^H \subset S_{k+1} \backslash S_k} \int_{\omega_j^{h,H}} -\left(\eta-\eta(x_j)\right) \psi\nabla \cdot (a \nabla \psi)\\
	&\overset{(d)}{\leq} \sum_{\omega_j^H \subset S_{k+1} \backslash S_k}\frac{C_0(d)h}{H}\|\psi\|_{L^2(\omega_j^{h,H})}\|\nabla \cdot (a \nabla \psi)\|_{L^2(\omega_j^{h,H})}\, ,
	\end{aligned}
	\end{equation}
	where,
	\begin{itemize}
	    \item in $(a)$, we have used the fact that $\eta$ is supported in $\Omega \backslash S_{k}$; moreover, in $\Omega \backslash S_{k+1}$, $\eta=1$ and $\nabla \cdot (a \nabla \psi)=\sum_j c_j \phi_j^{h,H}$ for some $c_j \in \bR$, and we have relied on the property $\int_{\omega_j^{H}} \phi^{h,H}_j \psi= 0$ for $\omega_j^{H} \in \Omega \backslash S_{k+1}$;
	    \item in $(b)$, we have used the fact that $\phi_j^{h,H}$ is supported in $\omega_j^{h,H}$;
	    \item in $(c)$, we have relied on the fact $\int_{\omega_j^{h,H}} \phi^{h,H}_j \psi= 0$ for $\omega_j^{h,H} \in \Omega \backslash S_k$ so we can subtract $\eta$ by the constant $\eta(x_j)$ for $x_j$ being the center of $\omega^{h,H}_j$;
	    \item in $(d)$ we have used the gradient bound on $\eta$ and the Cauchy-Schwarz inequality.
	\end{itemize}
	  For the term $\|\nabla \cdot (a \nabla \psi)\|_{L^2(\omega_j^{h,H})}$, we apply the inverse estimate established earlier, which leads to 
	\begin{align*}
	\eqref{eqn: exp decay 2}&\leq \frac{C_0(d)h}{H}\frac{\sqrt{a_{\max}}C_2(d)}{h} \sum_{\omega_j^H \subset S_{k+1} \backslash S_k} \|\psi\|_{L^2(\omega_j^{h,H})}\|\psi\|_{H^1_a(\omega_j^{h,H})}\\
	&\overset{(e)}{\leq} \frac{C_0(d)h}{H}\sqrt{a_{\max}}C_2(d)C_1(d)\sum_{\omega_j^H \subset S_{k+1} \backslash S_k} \|\nabla \psi\|_{L^2(\omega_j^{h,H})}\|\psi\|_{H^1_a(\omega_j^{h,H})}\\
	&\overset{(f)}{\leq} \frac{C_0(d)C_1(d)C_2(d)h\sqrt{a_{\max}}}{H}\|\nabla \psi\|_{L^2(S_{k+1}\backslash S_k)}\|\psi\|_{H_a^1(S_{k+1}\backslash S_k)}\\
	&\leq \frac{C_0(d)C_1(d)C_2(d)h}{H} \sqrt{\frac{a_{\max}}{a_{\min}}} \|\psi\|^2_{H^1_a(S_{k+1} \cap S_k^c)}\, ,
	\end{align*} 
	where in $(e)$, we have used the Poincar\' e inequality, based on the fact $\int_{\omega^{h,H}_j} \psi \phi^{h,H}_j=0$. The constant in the Poincar\'e inequality can be chosen the same as the one in Theorem \ref{thm: err ideal sol}, i.e., $C_1(d)$; for details see Proposition 2.5 and Theorem 3.3 in \cite{chen2019function}. The step $(f)$ is by the Cauchy-Schwarz inequality.
	
	For the second term in \eqref{eqn: exp decay 1}, we have 
	\begin{align*}
	\int_{\Omega} (\nabla \eta) \cdot \psi a \nabla \psi &= \int_{S_{k+1} \backslash S_k} (\nabla \eta) \cdot \psi a \nabla \psi\\
	&=\sum_{\omega_j^H \subset S_{k+1} \backslash S_k} \int_{\omega_j^H} (\nabla \eta)\cdot \psi a \nabla \psi\\
	&\leq \frac{C_0(d)\sqrt{a_{\max}}}{H}\sum_{\omega_j^H \subset S_{k+1} \backslash S_k}\|\psi\|_{L^2(\omega_j^H)} \|\psi\|_{H^1_a(\omega_j^H)}\\
	&\overset{(g)}{\leq} \frac{C_0(d)\sqrt{a_{\max}}}{H}\sum_{\omega_j^H \subset S_{k+1} \backslash S_k}H\rho_{2,d}(\frac{H}{h})C_1(d)\|\nabla \psi\|_{L^2(\omega_j^H)}\|\psi\|_{H_a^1(\omega_j^H)}\\
	&\leq C_0(d)C_1(d)\rho_{2,d}(\frac{H}{h})\sqrt{\frac{a_{\max}}{a_{\min}}}\|\psi\|^2_{H^1_a(S_{k+1} \backslash S_k)}\, ,
	\end{align*}
	where in step $(g)$, we have used the subsampled Poincar\'e inequality (Proposition 2.5 in \cite{chen2019function}) and the fact $\int_{\omega_j^H} \phi^{h,H}_j \psi= 0$. 
	
	Combining the estimates of the two terms and \eqref{eqn: exp decay 0}, we get
	\[\|\psi\|^2_{H_a^1(\Omega \backslash S_{k+1})} \leq C_0(d)(C_1(d)\rho_{2,d}(\frac{H}{h})+C_1(d)C_2(d)\frac{h}{H})\sqrt{\frac{a_{\max}}{a_{\min}}}\|\psi\|^2_{H^1_a(S_{k+1} \backslash S_k)}\, .\]
	Writing $\|\psi\|^2_{H^1_a(S_{k+1} \backslash S_k)}=\|\psi\|^2_{H^1_a(\Omega \backslash S_k)}-\|\psi\|^2_{H^1_a(\Omega \backslash S_{k+1})}$, we then arrive at 
	\[\|\psi\|^2_{H^1_a(\Omega \backslash S_{k+1})}\leq \beta(h,H) \|\psi\|^2_{H^1_a(\Omega \backslash S_{k})}\leq ... \leq \left(\beta(h,H)\right)^{k+1} \|\psi\|^2_{H^1_a(\Omega)}\, ,  \]
	where 
	\[\beta(h,H)=\frac{C_0(d)\sqrt{\frac{a_{\max}}{a_{\min}}}\left(C_1(d)\rho_{2,d}(\frac{H}{h})+C_1(d)C_2(d)\frac{h}{H}\right)}{C_0(d)\sqrt{\frac{a_{\max}}{a_{\min}}}\left(C_1(d)\rho_{2,d}(\frac{H}{h})+C_1(d)C_2(d)\frac{h}{H}\right)+1}\, . \]
\subsubsection{Norm Estimate} 
\label{subsec: norm estimate}
Let us recall the definition of $\psi_i^{h,H}$ and $\psi_i^{h,H,l}$ for $l=0$:
\begin{equation}
    \begin{aligned}
    \psi_{i}^{h,H} = \text{argmin}_{\psi \in H_0^1(\Omega)}\quad  &\|\psi\|_{H_a^1(\Omega)}^2 \\
     \text{subject to}\quad &[\psi, \phi_j^{h,H}] = \delta_{i,j}\ \  \text{for}\ \  j \in I \, .
 \end{aligned}
    \end{equation}
 \begin{equation}
 \label{eqn: proof norm estimate 1}
    \begin{aligned}
    \psi_{i}^{h,H,0} = \text{argmin}_{\psi \in H_0^1(\omega_i^H)}\quad  &\|\psi\|_{H_a^1(\omega_i^H)}^2 \\
     \text{subject to}\quad &[\psi, \phi_i^{h,H}] = 1\, .
 \end{aligned}
    \end{equation}   
Clearly, $\|\psi_i^{h,H}\|_{H_a^1(\Omega)}\leq \|\psi_i^{h,H,0}\|_{H_a^1(\omega_i^H)}$ so it suffices to estimate the latter. Without loss of generality, we can assume $\omega_i^{H}$ is centered at $0$, so that $\omega_i^{h,H}=[-h/2,h/2]^d$ and $\omega_i^{H}=[-H/2,H/2]^d$.

First, we choose $v \in H_0^1(\omega_i^H)$ to be a cut-off function that equals $1$ in $[-H/4,H/4]^d$ and equals $0$ outside $\omega_i^H$. Moreover, $v \geq 0$ and $\|\nabla v\|_{\infty}\lesssim 1/H$. Then, we have
\[[v,\phi_i^{h,H}]=\frac{1}{h^d}\int_{[-h/2,h/2]^d}v \simeq 1\, ,\]
and
\[\|v\|^2_{H_a^1(\omega_i^H)}\lesssim \int_{\omega_i^{H}} |\nabla v|^2\lesssim H^d\cdot \frac{1}{H^2}\lesssim H^{d-2}\, . \]
Define $w=v/[v,\phi_i^{h,H}]$, then $w$ satisfies the constraint in \eqref{eqn: proof norm estimate 1}, and $\|w\|_{H_a^1(\omega_i^H)}\lesssim H^{d/2-1}$, which leads to $\|\psi_i^{h,H,0}\|_{H_a^1(\omega_i^H)}\lesssim H^{d/2-1}$.  Thus, the case $d=1$ is proved.

Second, we deal with the case $d=2$. Suppose $h\leq H/2$, and we choose
\begin{equation*}
    v(x)=\left\{
    \begin{aligned}
    1-\frac{\log(1+\frac{4|x|}{h})}{\log(1+\frac{H}{h})}, \quad &|x|\leq \frac{H}{4}\\
    0, \quad &|x|> \frac{H}{4} \, .
    \end{aligned}
    \right.
\end{equation*}
We have $v(x)\leq 1$, and for $|x|\leq h/4$, $v(x)\geq 1-\frac{\log(2)}{\log(3)}\gtrsim 1$. Therefore, it holds that
\[[v,\phi_i^{h,H}]=\frac{1}{h^d}\int_{[-h/2,h/2]^d}v \simeq 1\, .\]
Then, we calculate the energy norm of $v$ as follows:
\begin{equation*}
    \begin{aligned}
    \|v\|_{H_a^1(\omega_i^H)}^2&\lesssim\frac{1}{\log^2(1+\frac{H}{h})}\int_{B(0,H/4)}\left(\frac{1}{h+4|x|}\right)^2\, \rd x\\
    &\lesssim \frac{1}{\log^2(1+\frac{H}{h})} \int_0^{H/4}\frac{r}{(4r+h)^2}\, \rd r\, .
    \end{aligned}
\end{equation*}
We write $\int_0^{H/4}\frac{r}{(4r+h)^2}\, \rd r=\int_0^{h/2}\frac{r}{(4r+h)^2}\, \rd r+\int_{h/2}^{H/4} \frac{r}{(4r+h)^2} \, \rd r \lesssim \int_0^{h/2}\frac{1}{h}\, \rd r+\int_{h/2}^{H/4} \frac{1}{r} \, \rd r$ $\lesssim \log(1+\frac{H}{h})$. Thus, it follows that \[\|v\|_{H_a^1(\omega_i^H)}\lesssim \left(\frac{1}{\log(1+\frac{H}{h})}\right)^{1/2}=\frac{1}{\rho_{2,d}(\frac{H}{h})}\, .\]
This concludes the proof for the case $h\leq H/2$. When $h>H/2$, we use the result in the first step $\|v\|_{H_a^1(\omega_i^H)}\lesssim H^{d/2-1}\lesssim 1 \lesssim \frac{1}{\rho_{2,d}(\frac{H}{h})}$. The case $d=2$ is proved.

Finally, when $d\geq 3$, we choose $v$ in a similar fashion as in the first step, such that $v=1$ in $[-h/4,h/4]^d$ and $v=1$ outside $[-h/2,h/2]^d$. Moreover, $v \geq 0$ and $\|\nabla v\|_{\infty}\lesssim 1/h$. Following the same argument in the first step, we will arrive at \[\|\psi_i^{h,H,0}\|_{H_a^1(\omega_i^H)}\lesssim h^{d/2-1}= \frac{1}{\rho_{2,d}(\frac{H}{h})}H^{d/2-1}\, ,\] which completes the proof.
\subsubsection{Localization Per Basis Function}
\label{subsec: Localization Per Basis Function}
We define a space \[V^{h,H}:=\{v \in H_0^1(\Omega): [v, \phi_j^{h,H}]=0, j\in I\}\, .\] Then, by the optimality of $\psi_i^{h,H}$ and $\psi_i^{h,H,l}$ in their corresponding optimization problems, we have 
$\left<\psi_i^{h,H}, v\right>_a=0$ for any $v \in V^{h,H}$ and $\left<\psi_i^{h,H,l}, v\right>_a=0$ for any $v \in V^{h,H}\bigcap H_0^1(\rN^l(\omega_i^H))$. Thus, $\left<\psi_i^{h,H}-\psi_i^{h,H,l},v\right>_a=0$ for any $v \in V^{h,H}\bigcap H_0^1(\rN^l(\omega_i^H))$. 

Then, we define $\chi_i^{h,H}=\psi_i^{h,H}-\psi_i^{h,H,0}$ and $\chi_i^{h,H,l}=\psi_i^{h,H,l}-\psi_i^{h,H,0}$. We have $\psi_i^{h,H}-\psi_i^{h,H,l}=\chi_i^{h,H}-\chi_i^{h,H,l}$ and $\chi_i^{h,H,l} \in V^{h,H}\bigcap H_0^1(\rN^l(\omega_i^H))$.

Based on the above fact and the orthogonality, we get
\begin{equation}
\begin{aligned}
    \|\psi_i^{h,H}-\psi_i^{h,H,l}\|_{H_a^1(\Omega)}^2&=\|\chi_i^{h,H}-\chi_i^{h,H,l}\|_{H_a^1(\Omega)}^2\\
    &\leq \|\chi_i^{h,H}-v\|_{H_a^1(\Omega)}^2\, ,
\end{aligned}
\end{equation}
for any $v \in V^{h,H}\bigcap H_0^1(\rN^l(\omega_i^H))$. We take \[v=\eta \chi_i^{h,H}-\sfP^{h,H,0}(\eta\chi_i^{h,H})\, ,\] where $\eta$ is a cut-off function that equals $1$ in $\rN^{l-1}(\omega_i^H)$ and equals $0$ outside $\rN^{l}(\omega_i^H)$. Moreover, $\eta\geq 0$ and $\|\nabla \eta\|_{\infty}\lesssim 1/H$. This $v$ belongs to $V^{h,H}\bigcap H_0^1(\rN^l(\omega_i^H))$ because both $\eta \chi_i^{h,H}$ and $\sfP^{h,H,0}(\eta\chi_i^{h,H})$ belong to $H_0^1(\rN^l(\omega_i^H))$, and by definition, $[\eta \chi_i^{h,H}-\sfP^{h,H,0}(\eta\chi_i^{h,H}),\phi_j^{h,H}]=0, j\in I $. Then, it follows that
\begin{equation}
\label{eqn: proof loc per basis 1}
    \begin{aligned}
    \|\chi_i^{h,H}-v\|_{H_a^1(\Omega)}^2&=\|(1-\eta)\chi_i^{h,H}-\sfP^{h,H,0}\left(\eta\chi_i^{h,H}\right)\|_{H_a^1(\Omega)}^2\\
    &=\|(1-\eta)\chi_i^{h,H}-\sfP^{h,H,0}\left((1-\eta)\chi_i^{h,H}\right)\|_{H_a^1(\Omega)}^2\, ,
    \end{aligned}
\end{equation}
where we have used the fact $\sfP^{h,H,0}\chi_i^{h,H}=0$. To move further, we need to use the following Lemma:
\begin{lemma}
\label{lemma: P 0 stable}
The operator $\sfP^{h,H,0}$ is stable under the norm $\|\cdot\|_{H_a^1(\Omega)}$. More precisely, we have for any $w \in H_0^1(\Omega)$, it holds
\[\|\sfP^{h,H,0}w\|_{H_a^1(\Omega)}\lesssim \|w\|_{H_a^1(\Omega)}\, . \]
\end{lemma}
\begin{proof}[Proof of Lemma \ref{lemma: P 0 stable}] 
By definition, $\psi_i^{h,H,0}$ is supported in $\omega_i^H$, and $\sfP^{h,H,0}w=\sum_{i\in I} [w,\phi_i^{h,H}]\psi_i^{h,H,0}$. Thus, we have
\begin{equation}
\label{eqn: proof lemma P 0 stable}
    \begin{aligned}
    \|w-\sfP^{h,H,0}w\|_{H_a^1(\Omega)}^2&=\sum_{i\in I} \int_{\omega_i^H} a\left|\nabla (w-[w,\phi_i^{h,H}]\psi_i^{h,H,0})\right|^2\\
    &\leq \sum_{i\in I} \int_{\omega_i^H} a\left|\nabla w\right|^2 =\|w\|_{H_a^1(\Omega)}^2\, ,
    \end{aligned}
\end{equation}
where we have used the fact that in each $\omega_i^H$, it holds \[\int_{\omega_i^H} a \nabla (w-[w,\phi_i^{h,H}]\psi_i^{h,H,0}) \cdot \nabla \psi_i^{h,H,0} =0\, , \] according to the definition of $\psi_i^{h,H,0}$. Equation \eqref{eqn: proof lemma P 0 stable} implies $\sfP^{h,H,0}$ is stable.
\end{proof}
Using Lemma \ref{lemma: P 0 stable}, we proceed as follows:
\begin{equation}
\label{eqn: proof loc per basis 2}
    \begin{aligned}
        \eqref{eqn: proof loc per basis 1}&\lesssim \|(1-\eta)\chi_i^{h,H}\|_{H_a^1(\Omega)}^2\\
        &=\int_{S_{l}\backslash S_{l-1}} a^2|(\nabla \eta)\chi_i^{h,H}|^2 + \int_{S_{l}\backslash S_{l-1}} a^2|\eta \nabla \chi_i^{h,H}| +\|\chi_i^{h,H}\|^2_{H_a^1(\Omega\backslash S_l)}\, ,
    \end{aligned}
\end{equation}
where we have used the notation $S_l=\rN^l(\omega_i^{H})$. For the first term in \eqref{eqn: proof loc per basis 2}, we have
\begin{equation}
\label{eqn: proof loc per basis 3}
    \begin{aligned}
    \int_{S_{l}\backslash S_{l-1}} a^2|(\nabla \eta)\chi^{h,H}_i|^2&=\sum_{\omega_j^H \subset S_{l}\backslash S_{l-1}}  \int_{\omega_j^H} a^2|(\nabla \eta)\chi^{h,H}_i|^2\\
    &\lesssim \sum_{\omega_j^H \subset S_{l}\backslash S_{l-1}} \frac{1}{H^2}\cdot H^2\left(\rho_{2,d}(\frac{H}{h})\right)^2\|\chi^{h,H}_i\|_{H_a^1(\omega_j^H)}^2\\
    &=\left(\rho_{2,d}(\frac{H}{h})\right)^2\|\chi^{h,H}_i\|_{H_a^1(S_{l}\backslash S_{l-1})}^2\, .
    \end{aligned}
\end{equation}
In the above inequality, we have used the gradient bound of $\eta$, the subsampled Poincare inequality (due to the property $[\chi_i^{h,H},\phi_j^{h,H}]=0$). Therefore, we obtain
\begin{equation}
\begin{aligned}
    \eqref{eqn: proof loc per basis 2}&\lesssim (1+\left(\rho_{2,d}(\frac{H}{h})\right)^2)\|\chi^{h,H}_i\|_{H_a^1(S_{l}\backslash S_{l-1})}^2+\|\chi_i^{h,H}\|^2_{H_a^1(\Omega\backslash S_l)}\\
    &\lesssim (1+\left(\rho_{2,d}(\frac{H}{h})\right)^2) \|\chi_i^{h,H}\|^2_{H_a^1(\Omega\backslash S_{l-1})}\, .
\end{aligned}
\end{equation}
Using the fact $\|\chi_i^{h,H}\|^2_{H_a^1(\Omega\backslash S_{l-1})}=\|\psi_i^{h,H}\|^2_{H_a^1(\Omega\backslash S_{l-1})}$, the exponential decay property and norm estimate of $\psi_i^{h,H}$, we finally obtain
\[ \|\psi_i^{h,H}-\psi_i^{h,H,l}\|_{H_a^1(\Omega)} \lesssim H^{d/2-1}\left(\beta(h,H)\right)^{l/2}\left(1+\frac{1}{\rho_{2,d}(\frac{H}{h})}\right)\, .\]
On the other hand, we have
\[ \|\psi_i^{h,H}-\psi_i^{h,H,l}\|_{H_a^1(\Omega)} \leq\|\psi_i^{h,H}\|_{H_a^1(\Omega)}+\|\psi_i^{h,H,l}\|_{H_a^1(\Omega)}\lesssim H^{d/2-1}\frac{1}{\rho_{2,d}(\frac{H}{h})}\, ,\]
due to the norm estimate established before. Thus, finally we obtain
\[ \|\psi_i^{h,H}-\psi_i^{h,H,l}\|_{H_a^1(\Omega)} \lesssim H^{d/2-1} \cdot \min\left\{\left(\beta(h,H)\right)^{l/2}\left(1+\frac{1}{\rho_{2,d}(\frac{H}{h})}\right), \frac{1}{\rho_{2,d}(\frac{H}{h})}\right\}\, .\]
Note that $1 \leq 1+\frac{1}{\rho_{2,d}(\frac{H}{h})}\leq 1 + \frac{1}{\rho_{2,d}(1)}$, we could further simplify the the upper bound by
\[ \|\psi_i^{h,H}-\psi_i^{h,H,l}\|_{H_a^1(\Omega)} \lesssim H^{d/2-1} \cdot \min\left\{\left(\beta(h,H)\right)^{l/2}, \frac{1}{\rho_{2,d}(\frac{H}{h})}\right\}\, .\]

\subsubsection{Overall Localization Error}
\label{subsec: Overall Localization Error}
Let $w=\sfP^{h,H}u-\sfP^{h,H,l}u$, then
\begin{equation}
\label{eqn: proof of Overall Localization Error 1}
    \begin{aligned}
    \|w\|_{H_a^1(\Omega)}^2&=\sum_{i\in I}[u,\phi_i^{h,H}]\left<w,\psi_i^{h,H}-\psi_i^{h,H,l}\right>_a\, .
    \end{aligned}
\end{equation}
For each $i$, to deal with the term $\left<w,\psi_i^{h,H}-\psi_i^{h,H,l}\right>_a$, we introduce a cut-off function $\eta$ that equals $0$ in $\rN^{l}(\omega_i^H)$ and equals $1$ in $\Omega \backslash \rN^{l+1}(\omega_i^H)$; moreover, $\eta \geq 0$ and $\|\nabla \eta\|_{\infty}\lesssim 1/H$. We define 
\[v=\sum_{\omega_j^H \subset \Omega \backslash \rN^{l}(\omega_i^H)} [\eta w,\phi_j^{h,H}]\psi_j^{h,H,0} \in H_0^1(\Omega \backslash \rN^{l}(\omega_i^H))\, . \]
Then $\eta w-v\in V^{h,H}\bigcap H_0^1(\Omega \backslash \rN^{l}(\omega_i^H))$. Thus, we have $\left<\eta w-v,\psi_i^{h,H}-\psi_i^{h,H,l}\right>=0$ because $\eta w-v$ has a different support with that of $\psi_i^{h,H,l}$, and $\left<\psi_i^{h,H}, v\right>_a=0$ for any $v \in V^{h,H}$; see the first paragraph in Subsection \ref{subsec: Localization Per Basis Function}, Therefore, we get
\begin{equation}
\begin{aligned}
    &\left<w,\psi_i^{h,H}-\psi_i^{h,H,l}\right>_a\\
    =&\left<w-\eta w+v,\psi_i^{h,H}-\psi_i^{h,H,l}\right>_a\\
    \leq &\left(\|(1-\eta)w\|_{H_a^1(\rN^{l}(\omega_i^H))}+ \|v\|_{H_a^1(\rN^{l+1}(\omega_i^H)\backslash \rN^{l}(\omega_i^H))}\right)\|\psi_i^{h,H}-\psi_i^{h,H,l}\|_{H_a^1(\Omega)}\, ,
\end{aligned}
\end{equation}
where we have used the fact that $v$ is supported in $\rN^{l+1}(\omega_i^H)\backslash \rN^{l}(\omega_i^H)$. Then, by construction of $v$, we have $\|v\|_{H_a^1(\rN^{l+1}(\omega_i^H)\backslash \rN^{l}(\omega_i^H))}\lesssim \|\eta w\|_{H_a^1(\rN^{l+1}(\omega_i^H)\backslash \rN^{l}(\omega_i^H))}$; the proof of this property is similar to that of Lemma \ref{lemma: P 0 stable}. Now, by using the fact $[w,\phi_j^{h,H}]=0$ and the subsampled Poincare inequality, we obtain
\[\|(1-\eta)w\|_{H_a^1(\rN^{l}(\omega_i^H))}+ \|\eta w\|_{H_a^1(\rN^{l+1}(\omega_i^H)\backslash \rN^{l}(\omega_i^H))}\lesssim \rho_{2,d}(\frac{H}{h}) \|w\|_{H_a^1(\rN^{l+1}(\omega_i^H))}\, .\]
Therefore, $\left<w,\psi_i^{h,H}-\psi_i^{h,H,l}\right>_a \lesssim \rho_{2,d}(\frac{H}{h}) \|w\|_{H_a^1(\rN^{l+1}(\omega_i^H))}\|\psi_i^{h,H}-\psi_i^{h,H,l}\|_{H_a^1(\Omega)}$. Then combining this estimate with
\eqref{eqn: proof of Overall Localization Error 1}, we arrive at
\begin{equation}
\begin{aligned}
    \|w\|_{H_a^1(\Omega)}^2&\lesssim \rho_{2,d}(\frac{H}{h})\sum_{i\in I}[u,\phi_i^{h,H}] \|w\|_{H_a^1(\rN^{l+1}(\omega_i^H))}\|\psi_i^{h,H}-\psi_i^{h,H,l}\|_{H_a^1(\Omega)}\\
    &\lesssim \rho_{2,d}(\frac{H}{h})\|u\|_{L^{\infty(\Omega})} l^{d/2}\|w\|_{H_a^1(\Omega)}\left(\sum_{i\in I} \|\psi_i^{h,H}-\psi_i^{h,H,l}\|_{H_a^1(\Omega)}^2 \right)^{1/2}\, ,
\end{aligned}   
\end{equation}
where the last step is by the Cauchy-Schwarz inequality. Combining the above estimate with the result in the last subsection (notice that the cardinality of $I$ is $1/H^d$), we get
\begin{equation}
    \|w\|_{H_a^1(\Omega)}\lesssim \min\left\{\left(\beta(h,H)\right)^{l/2}\rho_{2,d}(\frac{H}{h}),1\right\}\cdot \frac{l^{d/2}}{H}\|u\|_{L^{\infty}(\Omega)}\, .
\end{equation}
On the other hand, we can also bound
\begin{equation}
\begin{aligned}
    \|w\|_{H_a^1(\Omega)}&\leq \sum_{i\in I}|[u,\phi_i^{h,H}]|\cdot \|\psi_i^{h,H}-\psi_i^{h,H,l}\|_{H_a^1(\Omega)}\\
    &\lesssim \|u\|_{L^{\infty}(\Omega)}H^{-d}\cdot H^{d/2-1} \cdot \min\left\{\left(\beta(h,H)\right)^{l/2}, \frac{1}{\rho_{2,d}(\frac{H}{h})}\right\}\\
    &\lesssim \min\left\{\left(\beta(h,H)\right)^{l/2}\rho_{2,d}(\frac{H}{h}),1\right\} \cdot \frac{1}{H^{d/2+1}\rho_{2,d}(\frac{H}{h})}\|u\|_{L^{\infty}(\Omega)}\, .
\end{aligned}
\end{equation}
Therefore, we can write
\begin{equation}
    \|w\|_{H_a^1(\Omega)}\lesssim \min\left\{\left(\beta(h,H)\right)^{l/2}\rho_{2,d}(\frac{H}{h}),1\right\}\cdot \min\left\{\frac{l^{d/2}}{H},\frac{1}{H^{d/2+1}\rho_{2,d}(\frac{H}{h})}\right\}\|u\|_{L^{\infty}(\Omega)}\, .
\end{equation}
\subsubsection{Overall Recovery Error} 
\label{subsec: Overall Recovery Error}
When $d\leq 3$, we have $\|u\|_{L^\infty(\Omega)}\lesssim \|\cL u\|_{L^2(\Omega)}$; for details see Theorems 8.22 and 8.29 in \cite{gilbarg2015elliptic}. Combining the estimates in \eqref{eqn: err loc decompose} and \eqref{eqn: Overall localization error} leads to the estimate of the energy recovery error. For the $L^2$ recovery error, similar to \eqref{eqn: err loc decompose}, we have
\begin{equation}
    e_0^{h,H,l}(a,u)\lesssim (H\rho_{2,d}(\frac{H}{h}))^2\|\cL u\|_{L^2(\Omega)}+\|\sfP^{h,H}u-\sfP^{h,H,l}u\|_{L^2(\Omega)}\, .
\end{equation}
The second term $\|\sfP^{h,H}u-\sfP^{h,H,l}u\|_{L^2(\Omega)}$ is the $L^2$ localization error. We can simply bound it by: 
\begin{equation}
    \|\sfP^{h,H}u-\sfP^{h,H,l}u\|_{L^2(\Omega)}\leq \|\sfP^{h,H}u-\sfP^{h,H,l}u\|_{H_a^1(\Omega)}\, .
\end{equation}
On the other hand, notice that $[\sfP^{h,H}u-\sfP^{h,H,l}u,\phi_i^{h,H}]=0$ for any $i\in I$, we can use the subsampled Poincar\'e inequality so that
\begin{equation}
\begin{aligned}
    \|\sfP^{h,H}u-\sfP^{h,H,l}u\|^2_{L^2(\Omega)}&=\sum_{i\in I}\int_{\omega_i^H} |\sfP^{h,H}u-\sfP^{h,H,l}u|^2\\
    &\lesssim (H\rho_{2,d}(\frac{H}{h}))^2 \int_{\omega_i^H} a|\nabla (\sfP^{h,H}u-\sfP^{h,H,l}u)|^2 \\
    &= (H\rho_{2,d}(\frac{H}{h}))^2\|\sfP^{h,H}u-\sfP^{h,H,l}u\|^2_{H_a^1(\Omega)}\, .
\end{aligned}
\end{equation}
Therefore, we obtain
\begin{equation}
    \|\sfP^{h,H}u-\sfP^{h,H,l}u\|_{L^2(\Omega)}\leq \min\left\{1, H\rho_{2,d}(\frac{H}{h})\right\}\|\sfP^{h,H}u-\sfP^{h,H,l}u\|_{H_a^1(\Omega)}\, .
\end{equation}
Using the estimate of the energy error, we arrive at the final estimate.
\subsubsection{Overall Galerkin Error}
\label{subsec: overall Galerkin error}
The estimate for the energy Galerkin error is straightforward due to the Galerkin orthogonality. The $L^2$ error is estimated using the standard Aubin-Nitsche trick in finite element theory, which leads to square of the energy error. This completes the proof.
\subsection{Proof of Theorem \ref{thm: weighted estimates}}
\label{subsec: Proof of Theorem weighted estimates}
We start with the first case, i.e., $\|u\|_{H_a^1(\Omega)}<\infty$. By definition,
\begin{equation*}
    e_0^{h,H,\infty}(a,u)=\|u-\sfP^{h,H}u\|_{L^2(\Omega)}\, .
\end{equation*}
We have the relation $[u-\sfP^{h,H}u,\phi^{h,H}_j]=0$ for any $j\in I$. Thus, using the weighted Poincar\'e inequality in \cite{chen2019function} (Theorem 4.3 and Example 1), we can estimate the error as follows:
\begin{equation}
    \begin{aligned}
        \|u-\sfP^{h,H}u\|_{L^2(\Omega)}^2&=\sum_{i\in I}\|u-\sfP^{h,H}u\|_{L^2(\omega_i^H)}^2\\
        & \lesssim C(\gamma)^2H^2\sum_{i\in I}\|u-\sfP^{h,H}u\|_{H_a^1(\omega_i^H)}^2\\
        & \lesssim C(\gamma)^2H^2\|u\|_{H_a^1(\Omega)}^2\, ,
    \end{aligned}
\end{equation}
where in the last step, we have used the fact that $\|u-\sfP^{h,H}u\|_{H_a^1(\Omega)}\leq \|u\|_{H_a^1(\Omega)}$ due to the energy orthogonality. The first case is proved.

For the second case, by energy orthogonality of the recovery, we get
\begin{equation}
    \begin{aligned}
        e_1^{h,H,\infty}(a,u)\leq \|u-v\|_{H_a^1(\Omega)}\, ,
    \end{aligned}
\end{equation}
for any $v \in \text{span}~\{\psi_i^{h,H}\}_{i\in I}$. We can write $v=\cL^{-1}(\sum_{i\in I}c_i\phi_i^{h,H})$ for some $c_i$. Then, it holds that
\begin{equation}
\label{eqn: proof weighted estimate}
    \begin{aligned}
        \|u-v\|_{H_a^1(\Omega)}^2&=[u-v,\cL(u-v)]\\
        &=[u-v,f-\sum_{i\in I}c_i\phi_i^{h,H}]\\
        &=\sum_{i\in I} \int_{\omega_i^H}(u-v)(f-c_i\phi_i^{h,H})\, .
    \end{aligned}
\end{equation}
We choose $c_i=\int_{\omega_i^H} f$, so that
\begin{equation}
    \begin{aligned}
       \sum_{i\in I} \int_{\omega_i^H}(u-v)(f-c_i\phi_i^{h,H})&=\sum_{i\in I} \int_{\omega_i^H} \left(u-v-\int_{\omega_i^H} (u-v)\phi_i^{h,H}\right)f\\
       &\lesssim C(\gamma)\sum_{i\in I}H\|u-v\|_{H_a^1(\omega_i^H)}\|f\|_{L^2(\omega_i^{H})}\\
       &\leq C(\gamma)H \|u-v\|_{H_a^1(\Omega)}\|f\|_{L^2(\Omega)}\, ,
    \end{aligned}
\end{equation}
where in the second inequality, we use the Cauchy-Schwarz inequality and the weighted Poincar\'e inequality (Theorem 4.3 and Example 1 in \cite{chen2019function}). Thus, finally we get $\|u-v\|_{H_a^1(\Omega)} \lesssim C(\gamma)H\|f\|_{L^2(\Omega)}$, which implies the desired energy error estimate. The $L^2$ error estimate is obtained by using the standard Aubin-Nitsche trick in the finite element theory.
\section{Concluding Remarks} We summarize, discuss, and conclude this paper in this section.
\label{sec Concluding Remarks}
\subsection{Summary} In this paper, we performed a detailed study of a specific approach that connects the problem of numerical upscaling and function approximation, in the context that the target function is a solution to some multiscale elliptic PDEs with rough coefficients. Our main focus is on a subsampled lengthscale that appears in the coarse data of both problems. We investigated, both numerically and theoretically, the effect of $h$ on the recovery errors (for function approximation) and Galerkin errors (for numerical upscaling), given no computational constraints (ideal solution) or limited computational budgets (localized solution with a finite $l$), and given different regularity assumptions on the target function ($a(x)\in L^{\infty}(\Omega)$ or a singular $a(x)$). Our results imply that
\begin{itemize}
    \item There is a trade-off between approximation errors (of ideal solutions) and localization errors (due to finite $l$) regarding the subsampled lengthscale $h$, in addition to the oversampling parameter $l$.
    \item Due to the finite $l$ caused by our limited computational budget, the Galerkin solution and recovery solution are different in general. The former behaves better in the energy accuracy, while the latter stands out in the $L^2$ accuracy.
    \item When the target function is ``nearly flat" around the data locations, the subsampled data with a very small $h$ can still contain much coarse scale information. Thus, we would recommend to take our measurements there as a first choice.
\end{itemize}
The more quantitative descriptions of these main results are established by our numerical experiments and analytic studies based on tools such as the finite element theory, the subsampled Poincar\'e inequality, and weighted inequalities.
\subsection{Discussions} There could be multiple future directions: \begin{itemize}
    \item A better understanding of the trade-off regarding $h$ and $l$: how to choose optimal $l$ and $h$ adaptively with respect to $u$ or $f$. Our current results do not address this question fully.
    \item Other localization strategies: our localization in Subsection \ref{subsec: Basis Functions and Localization} follows from that in \cite{malqvist_localization_2014, owhadi_multigrid_2017}, and there are other possibilities, for example, the one in \cite{henning2013oversampling} or \cite{kornhuber2018analysis}, which leads to error estimates that does not blow up as $H\to 0$. It is of interest to understand how the subsampled lengthscale influences the accuracy in that context.
    \item Other measurement functions: as we mentioned earlier in Subsection \ref{subsec: Subsampled Lengthscales}, the choice of $\phi_i^{h,H}$ to be indicator functions in subsampled cubes is only for simplicity of analysis. Thus, results in this paper could be generalized to other types of subsampled measurement functions, for example, subsampled finite element tent functions.
    \item Generalization to high order models: the approach in Subsection \ref{subsec: a common approach} applies to a general operator $\cL$ that can be high order elliptic operators. This also connects to our discussion in Subsection \ref{subsec Small Limit Regime of Subsampled Lengthscales} regarding a high order model to avoid the degeneracy issues. It is of interest to study the effect of $h,l$ and also the order of the operator $\cL$ simultaneously on the recovery and Galerkin errors. 
    \item Coupling of two problems: we have considered a common approach that connects two class of problems. A natural question is about a hybrid model: suppose we have the domain $\Omega$ split into two smaller domains $\Omega_1$ and $\Omega_2$. In $\Omega_1$, we have a multiscale PDE $\cL u=f$ with known $f$, and in $\Omega_2$ we have some subsampled data $[u,\phi_i], i \in I$. How shall we take the advantages of the PDE model in $\Omega_1$ and the measured data in $\Omega_2$ to recover an accurate $u$? This can be a very fundamental problem in combining physics and data science.
\end{itemize}
\subsection{Conclusion} Overall, we have explored the connection between numerical upscaling for multiscale PDEs and scattered data approximation for heterogeneous functions, focusing on the roles of a subsampled lengthscale $h$ and the localization parameter $l$. We believe it sheds light on the interplay of the lengthscale of coarse data, the computational costs, the regularity of the target function, and the accuracy of approximations and numerical simulations.

\bibliographystyle{siamplain}
\bibliography{ref}
\end{document}